\documentclass[11pt]{amsart}
\usepackage[numeric]{amsrefs}

\DeclareRobustCommand{\VAN}[3]{#2}

\usepackage{tikz-cd}
\tikzstyle{printersafe}=[snake=snake,segment amplitude=0 pt]
\usepackage[font=small,format=plain,labelfont=bf,up,textfont=it,up]{caption}
\usepackage{hyperref}
\usepackage{stmaryrd}
\usepackage{amsmath}
\usepackage{mathtools} 
\usepackage{thm-restate}
\usepackage{mathrsfs}
\usepackage{amsfonts}
\usepackage{amssymb} 
\usepackage{amsthm}
\usepackage{stmaryrd}
\usepackage[OT2,T1]{fontenc}

\tikzcdset{scale cd/.style={every label/.append style={scale=#1},
    cells={nodes={scale=#1}}}}

\DeclareSymbolFont{cyrletters}{OT2}{wncyr}{m}{n}
\DeclareMathSymbol{\Sha}{\mathalpha}{cyrletters}{"58}

\newtheorem{Thm}[subsubsection]{Theorem}
\newtheorem{Lem}[subsubsection]{Lemma}
\newtheorem{Prop}[subsubsection]{Proposition}
\newtheorem{Cor}[subsubsection]{Corollary}
\newtheorem{Conj}[subsubsection]{Conjecture}
\newtheorem{mainThm}{Theorem}
\theoremstyle{definition}
\newtheorem{Def}[subsubsection]{Definition}

	\numberwithin{equation}{subsection}
	\newtheorem{prop}[subsubsection]{Proposition}

	\theoremstyle{definition}

	\theoremstyle{remark}

	\newtheorem*{assump*}{Assumption}

	\theoremstyle{definition}

\theoremstyle{remark}
\newtheorem{Vb}[subsubsection]{Example}

\newtheorem{Rem}[subsubsection]{Remark}
\newtheorem{Claim}[subsubsection]{Claim}

\newcommand{\afs}{\mathbb{A}_f^{\Sigma}}
\newcommand{\gsc}{G^{\mathrm{sc}}}
\newcommand{\gder}{G^{\mathrm{der}}}
\newcommand{\gad}{G^{\mathrm{ad}}}

\newcommand{\spec}{\operatorname{Spec}}

\newcommand{\gal}{\operatorname{Gal}}

\newcommand\restr[2]{{
 \left.\kern-\nulldelimiterspace 
 #1 
 \vphantom{\big|} 
 \right|_{#2} 
 }}
\usepackage{float}
\mathtoolsset{showonlyrefs}
\usepackage{bbm}
\usepackage{graphicx}
\usepackage{enumitem}

\newcommand{\gr}{\operatorname{Gr}}

\newcommand{\affperfk}{\mathbf{Aff}_k^{\mathrm{perf}}}

\newcommand{\mE}{\mathcal{E}}
\newcommand{\smE}{\sigma^{\ast} \mE}
\newcommand{\mF}{\mathcal{F}}

\newcommand{\mG}{\mathcal{G}}

\newcommand{\grg}{\operatorname{Gr}_{K}}

\newcommand{\lpGk}{L^+ \mathcal{G}_{K}}
\newcommand{\lpGj}{L^+ \mathcal{G}_{J}}
\newcommand{\lpGe}{L^+ \mathcal{G}_{\emptyset}}
\newcommand{\lmG}{L^m \mathcal{G}_{K}}
\newcommand{\lmGe}{L^m \mathcal{G}_{\emptyset}}

\newcommand{\lG}{LG}

\newcommand{\twKK}{W_K \backslash \tilde{W}/W_K}

\newcommand{\mloce}{\operatorname{M}^{\mathrm{loc}}_{\emptyset, \{\mu\}}}
\newcommand{\mloc}{\operatorname{M}^{\mathrm{loc}}_{K, \{\mu\}}}
\newcommand{\mloceonered}{\operatorname{M}^{\mathrm{loc},1-\mathrm{red}}_{\emptyset, \{\mu\}}}
\newcommand{\mloconered}{\operatorname{M}^{\mathrm{loc},1-\mathrm{red}}_{K, \{\mu\}}}
\newcommand{\mloceinf}{\operatorname{M}^{\mathrm{loc},\infty}_{\emptyset, \{\mu\}}}
\newcommand{\mlocinf}{\operatorname{M}^{\mathrm{loc},\infty}_{K, \{\mu\}}}
\newcommand{\mlocjinf}{\operatorname{M}^{\mathrm{loc},\infty}_{J, \{\mu\}}}
\newcommand{\mloceone}{\ker \gamma \backslash \mloceinf}
\newcommand{\mlocn}{\operatorname{M}^{\mathrm{loc,n}}_{K, \{\mu\}}}
\newcommand{\hkg}{\operatorname{Hk}_{K}}

\newcommand{\sht}{\operatorname{Sht}_{G,K}}
\newcommand{\shtb}{\operatorname{Sht}_{G,K,[b]}}

\newcommand{\shtj}{\operatorname{Sht}_{G,J}}

\newcommand{\mgred}{(\overline{\mG}_K)^{\mathrm{red}}}

\newcommand{\mgrede}{(\overline{\mG}_{\emptyset})^{\mathrm{red}}}

\newcommand{\shg}{\operatorname{Sh}_{G,K,U^p}}
\newcommand{\shgsp}{\operatorname{Sh}_{G_V,\mathcal{P},M^p}}
\newcommand{\shegsp}{\operatorname{Sh}_{G_V,\mathcal{P}',M^p}}
\newcommand{\shginf}{\operatorname{Sh}_{G,K}}
\newcommand{\shgj}{\operatorname{Sh}_{G,J,U^p}}
\newcommand{\shgjinf}{\operatorname{Sh}_{G,J}}
\newcommand{\shgeinf}{\operatorname{Sh}_{G,\emptyset}}
\newcommand{\shgetinf}{\widehat{\operatorname{Sh}}_{G,\emptyset}}
\newcommand{\shgetinfb}{\widehat{\operatorname{Sh}}_{G,\emptyset,[b]}}

\newcommand{\shge}{\operatorname{Sh}_{G,\emptyset,U^p}}
\newcommand{\shget}{\widehat{\operatorname{Sh}}_{G,\emptyset,U^p}}
\newcommand{\tshget}{\widetilde{\widehat{\operatorname{Sh}}}_{G, \emptyset,U^p}}
\newcommand{\shgetb}{\widehat{\operatorname{Sh}}_{G,\emptyset,[b],U^p}}
\newcommand{\admu}{\operatorname{Adm}(\{\mu\})}
\newcommand{\shtemu}{\operatorname{Sht}_{G,\emptyset, \{\mu\}}}
\newcommand{\shtmu}{\operatorname{Sht}_{G,K, \{\mu\}}}
\newcommand{\shtjmu}{\operatorname{Sht}_{G,J, \{\mu\}}}
\newcommand{\shtgemu}{\operatorname{Sht}_{G,\emptyset, \{\mu\}}}
\newcommand{\shtgmu}{\operatorname{Sht}_{G,K, \{\mu\}}}
\newcommand{\shgb}{\operatorname{Sh}_{G,K,[b],U^p}}

\newcommand{\shgeb}{\operatorname{Sh}_{G,\emptyset,[b],U^p}}

\newcommand{\shtmub}{{\operatorname{Sht}_{G,K, \{\mu\},[b]}}}

\newcommand{\qbar}{\overline{\mathbb{Q}}}

\newcommand{\qpur}{\mathbb{Q}_p^{\mathrm{ur}}}
\newcommand{\zpur}{\mathbb{Z}_p^{\mathrm{ur}}}
\newcommand{\qp}{\mathbb{Q}_p}
\newcommand{\ql}{\mathbb{Q}_{\ell}}
\newcommand{\zp}{\mathbb{Z}_{p}}
\newcommand{\zlocp}{\mathbb{Z}_{(p)}}
\newcommand{\qpbar}{\overline{\mathbb{Q}}_p}
\newcommand{\zpbar}{\overline{\mathbb{Z}}_p}

\newcommand{\ovfp}{\overline{\mathbb{F}}_p}

\newcommand{\qpbr}{\breve{\mathbb{Q}}_p}
\newcommand{\zpbr}{\breve{\mathbb{Z}}_p}
\newcommand{\gsp}{\operatorname{GSp}}
\newcommand{\afp}{\mathbb{A}_{f}^{p}}
\newcommand{\af}{\mathbb{A}_{f}}
\newcommand{\admuk}{\operatorname{Adm}(\{\mu\})_K}
\newcommand{\kadmu}{{}^{K} \!\! \operatorname{Adm}(\{\mu\})}

\newcommand{\bgmu}{B(G, \{\mu\})}

\newcommand{\xmub}{X(\mu,b)_K}

\newcommand{\xmubj}{X(\mu,b)_J}
\newcommand{\xmube}{X(\mu,b)_{\emptyset}}

\newcommand{\fp}{\overline{\mathbb{F}}_{p}}
\newcommand{\scrs}{\ensuremath{\mathscr{S}}}
\newcommand{\scrS}{\ensuremath{\mathscr{S}}}

\usepackage[all,cmtip]{xy}
\usepackage[margin=1.5in]{geometry}
\makeatletter
\let\@wraptoccontribs\wraptoccontribs
\makeatother

\thanks{This work was supported by the Engineering and Physical Sciences Research Council [EP/L015234/1], The EPSRC Centre for Doctoral Training in Geometry and Number Theory (The London School of Geometry and Number Theory), University College London and King's College London.}
\author{Pol van Hoften}
\contrib[with an appendix by]{Rong Zhou}
\address{Stanford mathematics department, 450 Jane Stanford way, Building 380, Stanford, CA 94305, USA} \email{p.van.hoften@vu.nl}

\title{Mod $p$ points on Shimura varieties of parahoric level}
\begin{document}
\begin{abstract}
We study the $\ovfp$-points of the Kisin--Pappas integral models of Shimura varieties of Hodge type with parahoric level. We show that if the group is quasi-split, then every isogeny class contains the reduction of a CM point, proving a conjecture of Kisin--Madapusi-Pera--Shin. We furthermore show that the mod $p$ isogeny classes are of the form predicted by the Langlands--Rapoport conjecture (cf. Conjecture 9.2 of \cite{Rapoport}) if either the Shimura variety is proper or if the group at $p$ is unramified. The main ingredient in our work is a global argument that allows us to reduce the conjecture to the case of very special parahoric level. This case is dealt with in the appendix by Rong Zhou. As a corollary to our arguments, we determine the connected components of Ekedahl--Oort strata.
\end{abstract}

\maketitle
\setcounter{tocdepth}{2}
\tableofcontents

\section{Introduction and statement of results}
\subsection{Introduction}
In \cite{LanglandsZeta}, Langlands outlines a three-part approach to prove that the Hasse--Weil zeta functions of Shimura varieties are related to $L$-functions of automorphic forms. The second part is about describing the mod $p$ points of suitable integral models of Shimura varieties, which is the central topic of this article. 

A conjectural description of the mod $p$ points of integral models of Shimura varieties was first given by Langlands in \cite{LanglandsJugendtraum} and was later refined by Langlands--Rapoport and Rapoport \cite{LanglandsRapoport, Reimann, Rapoport}. Together with the test function conjecture of Haines--Kottwitz \cite{HainesConjecture}, which was recently proved by Haines--Richarz \cite{HainesRicharz}, this conjecture is the main geometrical input to the Langlands--Kottwitz method for Shimura varieties of parahoric level. To explain these conjectures, we first need to introduce some notation. \smallskip

Let $(G,X)$ be a Shimura datum of Hodge type, let $p$ be a prime number and let $U_p \subset G(\qp)$ be a parahoric subgroup. For sufficiently small compact open subgroups $U^p \subset G(\afp)$, there is a Shimura variety $\mathbf{Sh}_{U}(G,X)$ of level $U=U^pU_p$, which is a smooth quasi-projective variety defined over the reflex field $E$. For a prime $v | p$ of $E$ we let $\mathcal{O}_{E,(v)}$ be the localisation of the ring of integers $\mathcal{O}_E$ of $E$ at the prime ideal $v$. Then there should be a \emph{canonical} integral model $\mathscr{S}_{U}(G,X)$ over $\mathcal{O}_{E, (v)}$. When $U_p$ is hyperspecial, canonical integral models should be smooth and are unique if they satisfy a certain extension property (cf. \cite{MilneII}). Recent work \cite{Pappas,PappasRapoportShtukas} of Pappas and Pappas--Rapoport defines a notion of canonical integral models when $U_p$ is an arbitrary parahoric and proves that they are unique if they exist. 

Then there should be a bijection (see \cite[Section 5]{LanglandsRapoport} and \cite[Conjecture 9.2]{Rapoport})
\begin{align} \label{Eq:Uniformisation}
  \varprojlim_{U^p} \mathscr{S}_{U^p U_p}(G,X)(\ovfp) \simeq \coprod_{\phi} S(\phi), 
\end{align}
where
\begin{align}
  S(\phi)=I_{\phi}(\mathbb{Q}) \backslash X_p(\phi) \times X^p(\phi).
\end{align}
Let us elaborate: If we think of $\mathscr{S}_{U^p U_p}(G,X)$ as parametrizing "abelian varieties with $G$-structure", then the sets $S(\phi)$ should correspond to points in a single isogeny class of "abelian varieties with $G$-structure" over $\ovfp$. For a fixed point $x$ in such an isogeny class, the set $X_p(\phi)$ parametrises "abelian varieties with $G$-structure" with a fixed $p$-power isogeny to $x$, and the set $X^p(\phi)$ parametrises "abelian varieties with $G$-structure" with a fixed prime-to-$p$ isogeny to $x$. The isogeny class of $x$ is then given by the quotient of $ X_p(\phi) \times X^p(\phi)$ by the group $I_{\phi}(\mathbb{Q})$ of self quasi-isogenies of $x$. The set $X^p(\phi)$ is a $G(\afp)$-torsor and $X_p(\phi)$ is a subset of $G(\qpur)/\mathcal{G}(\zpur)$, where $\mathcal{G}/\zp$ is the parahoric group scheme with $\mathcal{G}(\zp)=U_p$. In fact, the set $X_p(\phi)$ is the set of $\ovfp$-points of an affine Deligne--Lusztig variety, see Section \ref{Sec:ADLVSet}.

In the unramified PEL case, \eqref{Eq:Uniformisation} corresponds to Rapoport--Zink uniformisation of isogeny classes (see \cite[Section 6]{RapoportZink}), with $X_p(\phi)$ corresponding to the set of $\ovfp$-points of a Rapoport--Zink space. This is why we will often refer to \eqref{Eq:Uniformisation} as \emph{uniformisation} of isogeny classes. Uniformisation of isogeny classes for Shimura varieties of Hodge type is often assumed in recent work in the area, see e.g. \cite{HamacherKim, Hesse,PappasRapoportShtukas}.

One also expects that \eqref{Eq:Uniformisation} is compatible with the action of $G(\afp)$ on both sides, and that the action of Frobenius on the left-hand side should correspond to the action of a certain operator $\Phi$ on the right-hand side, see \cite[Conjecture 9.2]{Rapoport}. If $G_{\mathbb{Q}_p}$ is quasi-split, then we moreover expect that each isogeny class contains the reduction of a special point, see \cite[Conjecture 1]{KMPS}. 
\subsection{Main results} \label{Sec:MainResultsNotation}
Let $(G,X)$ be a Shimura datum of Hodge type, let $p>2$ be a prime number. We will assume throughout this introduction that: The group $G_{\mathbb{Q}_p}$ is quasi-split and splits over a tamely ramified extension, the prime $p$ does not divide the order of $\pi_1(\gder)$ and $\pi_1(G)_{I_{p}}$ is torsion-free\footnote{For Shimura data of abelian type that are not of type $D^{\mathbb{H}}$ in the sense of \cite[Appendix B]{Milne}, one can always find an auxiliary Shimura datum of Hodge type where the last two conditions are satisfied, see \cite[Lemma 4.6.22]{KisinPappas}.}. Here $I_{p} \subset \gal(\qpbar/\qp)$ is the inertia group and $\pi_1(G)$ is the algebraic fundamental group of $G$, see \cite{Borovoi}. 

Let $U_p \subset G(\qp)$ be a parahoric subgroup, let $U^p \subset G(\afp)$ be a sufficiently small compact open subgroup and consider the Shimura variety $\mathbf{Sh}_{U}(G,X)$ of level $U=U^pU_p$. By \cite[Theorem 0.1]{KisinPappas}, this Shimura variety has an extension to a flat normal scheme $\mathscr{S}_{U}(G,X)$ over $\mathcal{O}_{E_,(v)}$, where $v | p$ is a prime of the reflex field $E$. Under our assumptions, these integral models are canonical in the sense of \cite[Definition 7.1.3]{Pappas}, see \cite[Theorem 1.4]{Pappas}. 
\begin{restatable}{mainThm}{CMTheorem} \label{Thm:CMLifting}
Let $(G,X)$ be a Shimura variety of Hodge type as above. Then each isogeny class of $\mathscr{S}_{U}(G,X)(\ovfp)$ contains a point $x$ which is the reduction of a special point on $\mathbf{Sh}_{U}(G,X)$. \end{restatable} 
This confirms \cite[Conjecture 1]{KMPS}. Theorem \ref{Thm:CMLifting} for very special parahoric subgroups $U_p$ is part 2 of Theorem \ref{Thm:VerySpecialUniformisationAndLifting} of the appendix by Rong Zhou.

Theorem \ref{Thm:CMLifting} was proved by Kisin when $U_p$ is a hyperspecial subgroup, see \cite{KisinPoints}, and proved by Zhou when $G_{\qp}$ is residually split in the sense of \cite[Definition 9.10.2]{KalethaPrasad}, see \cite{Z}. We remind the reader that split implies residually split implies quasi-split, and that residually split and unramified implies split. As in \cite{KisinPoints,Z}, such a lifting result is deduced from uniformisation of isogeny classes, which is our second main result. Part 1 of the next theorem is part 1 of Theorem \ref{Thm:VerySpecialUniformisationAndLifting} of the appendix.
\begin{restatable}{mainThm}{Uniformisation} \label{Thm:Uniformisation}
Let $(G,X)$ be as above and let $U_p$ denote a parahoric subgroup of $G(\qp)$.
\begin{enumerate}
  \item  If $U_p$ is very special, then each isogeny class of $\mathscr{S}_{U}(G,X)(\ovfp)$ has the form
\begin{align}
I_{\phi}(\mathbb{Q}) \backslash X_p(\phi) \times X^p(\phi)/U^p.
\end{align}
  \item If either $G_{\mathbb{Q}_p}$ splits over an unramified extension or if $\mathbf{Sh}_{U}(G,X)$ is proper, then the same conclusion holds for arbitrary parahoric subgroups $U_p$.
\end{enumerate}
\end{restatable}
 \smallskip 
As a consequence of part 2 of Theorem 2, we verify that the He--Rapoport axioms of \cite{He-Rapoport} hold for the Kisin--Pappas integral models. All but one of the axioms (Axiom 4(c)) were proved in earlier work of Zhou, see \cite{Z}.
\begin{mainThm} \label{Thm:HeRapoport}
Let $(G,X)$ be a Shimura datum of Hodge type as above. If either $G_{\mathbb{Q}_p}$ splits over an unramified extension or if $\mathbf{Sh}_{U}(G,X)$ is proper, then the He--Rapoport axioms of \cite[Section 3]{He-Rapoport} hold for the Kisin--Pappas integral models.
\end{mainThm}
Combining our proof of part 2 of Theorem \ref{Thm:Uniformisation} with the $\ell$-adic monodromy theorem of \cite{OrdinaryHO}, we obtain a computation of the set of irreducible components of the Ekedahl--Kottwitz--Oort--Rapoport (EKOR) strata defined by Shen--Yu--Zhang in \cite{ShenYuZhang}. We assume for simplicity that $\gad$ is simple over $\mathbb{Q}$, see Theorem \ref{Thm:IrredEKOR} for a more general statement.
\begin{mainThm} \label{Thm:EKOR}
Let $(G,X)$ be as above, and let $U_p$ denote a very special parahoric. Let $\mathscr{S}_{U, \ovfp}\{w\}$ be an EKOR stratum that is not contained in the smallest Newton stratum. If either $G_{\mathbb{Q}_p}$ splits over an unramified extension or if $\mathbf{Sh}_{U}(G,X)$ is proper, then the natural map
\begin{align}
  \mathscr{S}_{U, \ovfp}\{w\} \to \mathscr{S}_{U,\ovfp}(G,X)
\end{align}
induces a bijection on sets of connected components. 
\end{mainThm}
If $U_p$ is hyperspecial, then EKOR strata coincide with Ekedahl--Oort strata, and this theorem determines their connected components. Theorem \ref{Thm:EKOR} was proved by Ekedahl and van der Geer \cite{EkedahlvanderGeer} in the Siegel case. Theorem \ref{Thm:EKOR} is used in \cite{vHXiao} to determine the connected components of Igusa varieties and Newton strata.
\begin{Rem}
We have stated Theorem \ref{Thm:EKOR} only for very special parahoric subgroups because we don't understand the set of connected components of $\mathscr{S}_{U,\ovfp}(G,X)$ when $U_p$ is a general parahoric subgroup and $\mathbf{Sh}_{U}(G,X)$ is not proper. If the Shimura variety is proper, then a similar statement holds at arbitrary parahoric level. This can be deduced from the proof of Theorem \ref{Thm:IrredEKOR}.
\end{Rem}
\subsection{Overview of the proof}
Both \cite{KisinPoints} and \cite{Z} employ roughly the same strategy, which we will now briefly sketch: The integral models $\mathscr{S}_{U}(G,X)$ of Shimura varieties of Hodge type come equipped, by construction, with finite maps $\mathscr{S}_{U}(G,X) \to \mathscr{S}_{M}(\gsp, S^{\pm})$ to Siegel modular varieties. Given a point $x \in \mathscr{S}_{U}(G,X)(\ovfp)$, classical Dieudonné theory produces a map
\begin{align}
  X_p(\phi) \to \mathscr{S}_{M}(\gsp, S^{\pm})(\ovfp),
\end{align}
and the main difficulty is to show that it factors through $\mathscr{S}_{U}(G,X)$. A deformation theoretic argument shows that it suffices to prove this factorisation for one point on each connected component of $X_p(\phi)$\footnote{There is a perfect scheme whose set of $\ovfp$-points is naturally identified with $X_p(\phi)$, which gives a decomposition of $X_p(\phi)$ into connected components, see Lemma \ref{Lem:ADLVLemma}.}, and therefore we need to understand these connected components. In the hyperspecial case, this is done in \cite{CKV}, and in the parahoric case this is done in \cite{HZ}, under the assumption that $G_{\mathbb{Q}_p}$ is residually split. The main obstruction to extend the methods of \cite{Z} beyond the residually split case, is that we do not understand connected components of affine Deligne--Lusztig varieties of parahoric level for more general groups.\footnote{After a first version of our paper appeared, we learned of work of Nie \cite{NieConnected} which solves this problem for unramified groups. Recently there has been work of Gleason--Lim--Xu \cite{GleasonLimXu} and Gleason--Lourenço \cite{GleasonLourenco} which completely settles the problem of understanding connected components of affine Deligne--Lusztig varieties.} 

\subsubsection{} The geometry of affine Deligne--Lusztig varieties is simpler the larger the parahoric subgroup is. For unramified groups the geometry is simplest for hyperspecial subgroups, and for more general quasi-split groups the geometry is simplest for very special subgroups. This is why it is reasonable to try to prove Theorems \ref{Thm:Uniformisation} and \ref{Thm:CMLifting} for very special parahoric subgroups, using the above strategy.

In Appendix \ref{Sec:VerySpecialADLV}, Rong Zhou studies connected components of affine Deligne--Lusztig varieties for quasi-split groups and very special parahoric subgroups, generalising results of \cite{CKV} and \cite{Nie1} in the case of unramified groups and hyperspecial level. In particular, part 1 of Theorem \ref{Thm:Uniformisation} and Theorem \ref{Thm:CMLifting} in the case of a very special parahoric are proved there, see Theorem \ref{Thm:VerySpecialUniformisationAndLifting}. 

\subsubsection{} To prove uniformisation for a general parahoric subgroup, we use the fact that every parahoric subgroup contains an Iwahori subgroup, and that every Iwahori subgroup is contained in a very special parahoric subgroup if $G_{\qp}$ is quasi-split. Thus given the results of Appendix \ref{Sec:VerySpecialADLV}, we need to show that the validity of Theorems \ref{Thm:Uniformisation} and \ref{Thm:CMLifting} propagates "up" from very special parahoric subgroups to Iwahori subgroups, and propagates "down" from Iwahori subgroups to general parahoric subgroups. The latter is proved in \cite[Proposition 7.7]{Z}, and so we focus on the former.

Let $U_p$ denote a very special parahoric subgroup and let $U_p'$ denote an Iwahori subgroup contained in $U_p$, then by \cite[Section 7]{Z} there is a proper morphism of integral models $\mathscr{S}_{U'}(G,X) \to \mathscr{S}_{U}(G,X)$ and we let $\operatorname{Sh}_{U_p'} \to \operatorname{Sh}_{U_p}$ be the induced morphism on the perfections of their special fibers. There is a commutative diagram
\begin{equation} \label{Eq:KeyDiagram}
  \begin{tikzcd}
  \operatorname{Sh}_{G, U'} \arrow{r} \arrow{d} & \operatorname{Sht}_{G, \mu, U_p'} \arrow{d} \\
  \operatorname{Sh}_{G, U} \arrow{r} & \operatorname{Sht}_{G, \mu, U_p},
  \end{tikzcd}
\end{equation}
where $\operatorname{Sht}_{G, \mu, U_p}$ is the stack of parahoric $U_p$-shtukas of type $\mu$ introduced by Xiao--Zhu \cite{XiaoZhu} and Shen--Yu--Zhang \cite{ShenYuZhang} (see Section \ref{Sec:Shtukas}, \ref{Sec:RelPos}), with $\mu$ the inverse of the Hodge cocharacter induced by the Shimura datum $(G,X)$.

The horizontal morphisms in \eqref{Eq:KeyDiagram} are the Hodge type analogues of the morphism from the moduli space of abelian varieties to the moduli stack of quasi-polarised Dieudonn\'e modules. If $(G,X)=(\operatorname{GSp},S^{\pm})$, then this diagram is Cartesian. In general, it follows from `local uniformisation' of $\operatorname{Sht}_{G, \mu, U'}$, that isogeny classes in $\operatorname{Sh}_{G, U_p'}$ have the correct form if \eqref{Eq:KeyDiagram} is Cartesian, see Theorem \ref{Thm:LiftingUniformisation}. One of the main technical results of this paper, Theorem \ref{Thm:DiagramCartesian}, is that the diagram is Cartesian under the assumptions of part 2 of Theorem \ref{Thm:Uniformisation}, which proves a conjecture of He and Rapoport that we learned from Rong Zhou.   

\subsubsection{} We prove in Section \ref{Sec:Local}, see Proposition \ref{Prop:Representable}, that the morphism $\operatorname{Sht}_{G, \mu, U_p'} \to \operatorname{Sht}_{G, \mu, U_p}$ is representable in perfectly proper algebraic spaces, and we let $\operatorname{Sh}_{G, U', \star}$ be the fiber product of \eqref{Eq:KeyDiagram}. There is a map $\iota:\operatorname{Sh}_{G, U'} \to \operatorname{Sh}_{G, U', \star}$ given by the universal property of the fiber product and we prove that it is a closed immersion, see Proposition \ref{Prop:ClosedImmersion}. To prove the main theorem, it suffices to show that $\iota$ is an isomorphism. We do this by showing it is a closed immersion of equidimensional perfect algebraic spaces of the same dimension whose image intersects every irreducible component of
the target, which clearly must then be an isomorphism.

We first show that $\operatorname{Sh}_{G, U', \star}$ is equidimensional of the same dimension as $\operatorname{Sh}_{G, U'}$ and that it has a Kottwitz--Rapoport (KR) stratification with the expected properties. To do this, we build a local model diagram for $\operatorname{Sh}_{G, U', \star}$ in the world of perfect algebraic geometry, see Proposition \ref{Prop:perfectlysmooth}. This requires us to produce a version of the diagram in \eqref{Eq:KeyDiagram} for stacks of restricted shtukas, and to analyse the forgetful maps for these stacks. Another key ingredient is the fact, proved by Hoff, \cite{HoffThesis}, that the morphisms from $\operatorname{Sh}_{G, U}$ to these stacks of restricted shtukas are perfectly smooth. 

The next step is to study the irreducible components of $\operatorname{Sh}_{G, U', \star}$ and $\operatorname{Sh}_{G, U'}$.  In Section \ref{Sec:KRConnected}, see Proposition \ref{Prop:GlobalActionTransitiveKR}, we will show that each irreducible component of $\operatorname{Sh}_{G, U', \star}$ can be moved into $\operatorname{Sh}_{G, U'}$ using prime-to-$p$ Hecke operators. Since 
$\operatorname{Sh}_{G, U', \star}$ is stable under the prime-to-$p$ Hecke operators, we may conclude from this that $\iota:\operatorname{Sh}_{G, U'} \to \operatorname{Sh}_{G, U', \star}$ is an isomorphism.

To prove Proposition \ref{Prop:GlobalActionTransitiveKR}, we use the KR stratification of both $\operatorname{Sh}_{G, U'}$ and $\operatorname{Sh}_{G, U', \star}$ to reduce to analysing irreducible components in each KR stratum separately. Our proof then proceeds by degenerating to the zero-dimensional KR stratum, which we describe explicitly using Rapoport--Zink uniformisation of the basic locus. 

Our assumption that either $G_{\mathbb{Q}_p}$ splits over an unramified extension or that $\mathbf{Sh}_{U}(G,X)$ is proper will be used to prove that every irreducible component of the closure of a KR stratum in $\operatorname{Sh}_{G, U', \star}$ intersects the zero-dimensional KR stratum, see Lemma \ref{Lem:RaynaudArgument} and Proposition \ref{Prop:KRZero}. In the proper case, it is enough to prove that KR strata in $\operatorname{Sh}_{G, U', \star}$ are quasi-affine. In the unramified case, we use results of \cite{WedhornZiegler} and \cite{Andreatta} on the Ekedahl--Oort stratification and results of \cite{He1} on the geometry of forgetful maps. 

\subsection{Outline of the paper}
In Section \ref{Sec:Local} we will study forgetful maps for moduli stacks of local shtukas and moduli stacks of restricted local shtukas. We will also study Newton strata in moduli spaces of shtukas and describe them explicitly in terms of affine Deligne--Lusztig varieties. In Section \ref{Sec:Uniformisation} we study uniformisation of isogeny classes in Shimura varieties of Hodge type at parahoric level. We will deduce the existence of CM lifts at arbitrary parahoric level from the results of Appendix \ref{Sec:VerySpecialADLV}, and we will show that uniformisation for general parahoric subgroups is equivalent to a certain diagram being Cartesian. In Section \ref{Sec:Cartesian}, we prove that this diagram is Cartesian.
\section{Local shtukas} \label{Sec:Local}
We start this section by recalling some perfect algebraic geometry from \cite[Appendix A]{XiaoZhu} and defining a notion of weakly perfectly smooth morphisms of perfect algebraic stacks. 

In the rest of the section we will recall the moduli stacks of local shtukas with parahoric level of \cite{ShenYuZhang} and study the forgetful maps between them. We start by proving Proposition \ref{Prop:Representable}, which states that this forgetful map is representable and (perfectly) proper. We then study forgetful maps of restricted local shtukas and prove Proposition \ref{Prop:CartesianInfiniteToFinite}, which is an important technical result that will be used in Section \ref{Sec:Uniformisation} to prove equidimensionality of $\operatorname{Sh}_{G, U_p', \star}$.

In the second half, we discuss $\sigma$-conjugacy classes and the Newton stratification on moduli stacks of local shtukas. We end by discussing affine Deligne--Lusztig varieties and use them in Lemma \ref{Lem:LocalUniformisation} to describe Newton strata in moduli stacks of local shtukas. This latter result is used in Section \ref{Sec:Uniformisation} to lift uniformisation along forgetful maps.

\subsection{Some perfect algebraic geometry} \label{Sec:Dimensions} We will use the language of perfect algebraic geometry from \cite[Appendix A]{Zhu1}. Let $k$ be a perfect field and denote by $\affperfk$ the category of perfect $k$-algebras, on which we will consider both the \'etale and fpqc topologies. Perfect $k$-schemes define fpqc sheaves on $\affperfk$, and for $X$ a scheme over $k$ we will write $X^{\mathrm{perf}}$ for the (inverse) perfection of $X$, given by the inverse limit over the relative $k$-Frobenius of $X$. This inverse limit exists in the category of schemes, see \cite[Section 5]{Perfection}, and the natural map $X^{\mathrm{perf}} \to X$ is a universal homeomorphism.

Perfect algebraic spaces are defined to be sheaves $X$ on $\affperfk$ such that the diagonal $X \to X \times X$ is representable in perfect schemes, and such that $X$ admits an \'etale surjection from a scheme (cf. \cite[Definition 025Y]{stacks-project}). See \cite[Definition A.1.7]{XiaoZhu} for the definition of a pfp (perfectly of finite presentation) algebraic space. A perfect algebraic space is pfp if and only if it is isomorphic to the perfection of an algebraic space of finite presentation over $k$, see \cite[Proposition A.1.8]{XiaoZhu}. We will often write pfp algebraic space to mean pfp perfect algebraic space. A \emph{deperfection} of a pfp algebraic space $Y$ is a morphism $Y \to Y_0$ with $Y$ an algebraic space of finite presentation that induces an isomorphism $Y \xrightarrow{\sim} Y_0^{\mathrm{perf}}$.
\begin{Lem} \label{Lem:LimitLemma}
    If $X$ is a pfp algebraic space. Then for every directed set $I$ and any inverse system $\{T_i\}_{i \in I}$ of perfect qcqs $k$-schemes with affine transition maps $T_i \to T_J$, the natural map
    \begin{align}
        \operatorname{Hom}(\varprojlim_i T_i, X) \to \varinjlim_i \operatorname{Hom}(T_i, X)
    \end{align}
is a bijection.\footnote{Note that the inverse limit $T=\varprojlim_i T_i$ exists in the category of schemes by \cite[Tag 01YX]{stacks-project}. Moreover, this $T$ is a perfect scheme since perfection commutes with inverse limits (being an inverse limit).}
\end{Lem}
\begin{proof}
    Choose a deperfection $X \to X_0$ of $X$ using \cite[Proposition A.1.8]{XiaoZhu}. We may then apply \cite[Proposition 01ZC]{stacks-project} to deduce that the natural map
    \begin{align} \label{Eq:LimitMap}
        \operatorname{Hom}(\varprojlim_i T_i, X_0) \to \varinjlim_i \operatorname{Hom}(T_i, X_0),
    \end{align}
    is an isomorphism. We conclude by noting that \eqref{Eq:LimitMap} can be identified
    \begin{align}
        \operatorname{Hom}(\varprojlim_i T_i, X) \to \varinjlim_i \operatorname{Hom}(T_i, X)
    \end{align}
    since the $T_i$ are perfect and since $\varprojlim_i T_i$ is perfect.
\end{proof}

\subsubsection{} We will use the notion of perfectly proper morphisms of perfect algebraic spaces, see \cite[Definition A.18]{Zhu1}. A morphism $f:X \to Y$ of pfp algebraic spaces over $k$ is perfectly proper if and only if it is isomorphic to the perfection of a proper morphism of algebraic spaces of finite presentation over $k$, see \cite[Lemma A.19]{Zhu1}. We will often write perfectly proper algebraic space to mean a perfect algebraic space whose structure map to $\spec k$ is perfectly proper.

Recall that a morphism $f:X \to Y$ of perfect algebraic spaces is called \emph{perfectly smooth of relative dimension $d$ at $x$}, where $x \in X$, if there is an \'etale neighbourhood $U \to X$ of $x$ and $V \to Y$ of $f(x)$ such that $U \to X \to Y$ factors through a map $h:U \to V$ and such that $h$ factors as 
\begin{equation} \label{Eq:FactorisationPerfectlySmooth}
    \begin{tikzcd}
       & U \arrow{r} \arrow{dl}{h'} \arrow{d}{h} & X \arrow{d}{f} \\
           (\mathbb{A}^{d}_{k})^{\mathrm{perf}} \times V \arrow{r}{\operatorname{pr}}&  V \arrow{r} & Y,
    \end{tikzcd}
\end{equation}
where $h'$ is \'etale and where $\operatorname{pr}$ is the projection onto $V$. It is called \emph{perfectly smooth of relative dimension $d$} if it is perfectly smooth of relative dimension $d$ at all points $x \in X$. This property is preserved under base change, and the composition of a perfectly smooth morphism of relative dimension $d$ with a perfectly smooth morphism of relative dimension $d'$ is perfectly smooth of relative dimension $d+d'$. A morphism $X \to Y$ is called \emph{perfectly smooth} if it is perfectly smooth of \emph{some} dimension at every $x \in X$. This property is also preserved under base change and composition. 
\begin{Vb} \label{Vb:SmoothPerfSmooth}
If $f:X \to Y$ is a morphism of schemes over $k$ that is smooth of relative dimension $d$ at $x \in X$, then $f^{\mathrm{perf}}:X^{\mathrm{perf}} \to Y^{\mathrm{perf}}$ is perfectly smooth of relative dimension $d$ at $x$ by \cite[Lemma 054L]{stacks-project}. Indeed, the natural map $X^{\mathrm{perf}} \to X$ is a universal homeomorphism and thus identifies the \'etale sites of $X$ and $X^{\mathrm{perf}}$, see \cite[Theorem 05ZH]{stacks-project}
\end{Vb}
\begin{Vb} \label{Vb:GroupScheme}
    Let $G$ be a pfp group scheme over $\spec k$. Then $G \to \spec k$ arises as the perfection of a smooth group scheme over $k$ by \cite[Lemma A.26]{Zhu1}, and therefore $G \to \spec k$ is perfectly smooth by Example \ref{Vb:SmoothPerfSmooth}. This furthermore means that $G$-torsors for the \'etale topology are perfectly smooth morphisms, as the property of being perfectly smooth is clearly \'etale local on the target. 
\end{Vb}
The following lemma is a straightforward consequence of the definition.
\begin{Lem} \label{Lem:RelDimLocConstant2}
Let $f:X \to Y$ be a perfectly smooth morphism of perfect algebraic spaces. If $X$ is connected, then $f$ is perfectly smooth of relative dimension $d$ for some integer $d$.
\end{Lem}
We will later use the notion of normality for perfect algebraic spaces. Note that if an algebraic space $Y$ is normal then its perfection $Y^{\mathrm{perf}}$ is normal. Indeed, since normality is \'etale local, see \cite[Lemma 034F]{stacks-project}, this can be reduced to the affine case using the fact that Frobenius is affine, and then it follows from the fact that a filtered colimit of normal rings is normal.  
\begin{Lem} \label{Lem:NormalDeperfection}
    A normal pfp algebraic space $Y$ admits a normal deperfection.
\end{Lem}
\begin{proof}
Since $Y$ has finitely many irreducible components (because it is pfp), it follows from \cite[Lemma 0357]{stacks-project}, that $Y$ is a disjoint union of finitely many integral normal algebraic spaces. Thus we may assume that $Y$ is integral and normal. Choose a deperfection $Y \to Y_0$ with $Y_0$ reduced and observe that $Y_0$ is irreducible because $Y$ is, and thus $Y_0$ is integral. Let $\tilde{Y}_0 \to Y_0$ be the normalisation of $Y_0$ and note that it suffices to show that $\tilde{Y}_0$ has the same perfection as $Y_0$. Since normalisation commutes with \'etale base change, see \cite[Lemma 082F]{stacks-project}, we may assume that $Y_0=\spec A_0$ with $A_0$ an integral domain.

Then $\tilde{Y}_0$ corresponds to an $A_0$-algebra $B$, and $Y$ corresponds to an $A_0$-algebra $A$. Since the morphism $A_0 \to A$ is an injective integral morphism and $A$ is normal, it follows that $B$ is isomorphic to the integral closure of $A_0$ inside of $A$. But this implies that $B$ has the same perfection as $A_0$. 
\end{proof}
\begin{Lem} \label{Lem:NormalityIsLocal}
    Let $f:X \to Y$ be a perfectly smooth morphism of pfp algebraic spaces. If $Y$ is normal, then $X$ is normal. If $f$ is moreover surjective and $X$ is normal, then $Y$ is normal. 
\end{Lem}
\begin{proof}
Fix $x \in X$ with image $y \in Y$. By definition, we know that $f$ is perfectly smooth of relative dimension $d$ at $x$. Thus there are \'etale neighbourhoods $U \to X$ of $x$ and $V \to Y$ of $u$ such that $U \to X \to Y$ factors through a map $h:U \to V$ and such that $h$ factors as in equation \eqref{Eq:FactorisationPerfectlySmooth}. 

Assume that $Y$ is normal and choose a normal deperfection $Y \to Y_0$ of $Y$ using Lemma \ref{Lem:NormalDeperfection}. Then by topological invariance of the \'etale site, see \cite[Theorem 05ZH]{stacks-project}, there is a unique \'etale morphism $V_0 \to Y_0$ whose perfection recovers $V \to Y$. Similarly, there is a unique \'etale morphism $U_0 \to \mathbb{A}^{d}_{k} \times V_0$ whose perfection recovers $h'$. The induced map $U_0 \to Y_0$ is smooth because it is a composition of smooth maps. It follows from \cite[Lemma 034F]{stacks-project} that $U_0$ is normal since $Y_0$ is normal. Hence $U$ is normal and so $X$ is normal (in a neighborhood of $x$), by \cite[Lemma 034F]{stacks-project}. This argument works for arbitrary $x$ and thus proves the normality of $X$.

\smallskip We will now assume that $X$ is normal and show that $Y$ is normal in a neighbourhood of $y$ (if $f$ is surjective, this thus shows that normality of $X$ implies the normality of $Y$). Let $Y \to Y_0$ be any deperfection, then we will show that the normalisation $\tilde{Y}_0 \to Y_0$ is a universal homeomorphism. This implies that $\tilde{Y}_0 \to Y_0$ induces an isomorphism on perfections by \cite[Lemma 3.8]{BhattScholze}, and thus $Y$ is the perfection of a normal scheme and hence normal. 

Let $V_0 \to Y_0$ and $U_0 \to \mathbb{A}^{d}_{k} \times V_0$ be as above. Since normalisation commutes with smooth base change, see \cite[Lemma 082F]{stacks-project}, we find that there is a Cartesian diagram (where $\tilde{U}_0 \to U_0$ is the normalisation of $U_0$)
\begin{equation}
    \begin{tikzcd}
        \tilde{U}_0 \arrow{d} \arrow{r} & \tilde{Y}_0 \arrow{d} \\
        U_0 \arrow{r} & Y_0. 
    \end{tikzcd}
\end{equation}
Since $X$ is normal, we find that $U$ is normal by \cite[Lemma 034F]{stacks-project}. By \cite[Lemma 0BB4]{stacks-project}, there is a unique commutative diagram
\begin{equation}
    \begin{tikzcd}
        \tilde{U}_0 \arrow{dr} & & U \arrow{dl} \arrow{ll} \\
        & U_0
    \end{tikzcd}.
\end{equation}
This shows that $\tilde{U}_0 \to U_0$ is injective on $k$-points, and hence universally injective. Since $\tilde{U}_0 \to U_0$ is also surjective and closed, it follows that $\tilde{U}_0 \to U_0$ is a universal homeomorphism. It follows from \cite[Lemma 0CFX]{stacks-project} that $\tilde{Y}_0 \to Y_0$ is a universal homeomorphism, as desired.
\end{proof}

\subsubsection{} Let $f:X \to Y$ be a morphism of pfp algebraic spaces. We expect that if there is a perfectly smooth surjective map $g:Z \to X$ such that $f \circ g$ is perfectly smooth, then $f$ itself is perfectly smooth. However, we do not know how to prove this, hence we make the auxiliary definitions \ref{Def:WeaklyPerfectlySmooth} and \ref{Def:WeaklyPerfectlySmoothII}.

\begin{Def} \label{Def:WeaklyPerfectlySmooth}
A morphism $f:Y \to Z$ of perfect algebraic spaces is called \emph{weakly perfectly smooth of relative dimension $d$ at $y$} for $y \in Y$ if: There exists an open neighborhood $U$ of $y$ and a surjective map $g:X \to U$ that is perfectly smooth of relative dimension $e$, where $X$ is a perfect algebraic space, such that $f \circ g:X \to Z$ is perfectly smooth of relative dimension $e+d$. A morphism is called \emph{weakly perfectly smooth of relative dimension $d$} if it is weakly perfectly smooth of relative dimension $d$ at $y$ for all $y \in Y$. 
\end{Def}
This property is preserved under base change, and the composition of a weakly perfectly smooth morphism of relative dimension $d_1$ with a weakly perfectly smooth morphism of relative dimension $d_2$ is a weakly perfectly smooth morphism of relative dimension $d_1+d_2$. Indeed, suppose we are given morphisms $f_1:X \to Y$ and $f_2:Y \to Z$ together with surjections $g_1:X' \to X$ and $g_2:Y' \to Y$ that are perfectly smooth of relative dimension $e_1$ and $e_2$, respectively and such that $X' \to X \to Y$ and $Y' \to Y \to Z$ are perfectly smooth of relative dimension $d_1+e_1$ and $d_2+e_2$, respectively.  Then $X'':=X' \times_{Y} Y'$ is perfectly smooth over $X$ of relative dimension $e_1+e_2$, and $X'' \to X \to Z$ is perfectly smooth of relative dimension $e_2+d_2+e_1+d_1$ by writing it as the composition of $X'' \to Y'$ and $Y' \to Z$. \smallskip 

The following lemmas show that the integer $d$ is well-defined. 
\begin{Lem} \label{Lem:Dimension}
Let $f:X \to Y$ be a weakly perfectly smooth morphism of equidimensional pfp algebraic spaces such that the fibers of $f$ are equidimensional of dimension $d$. Then $\operatorname{Dim} X + d = \operatorname{Dim} Y$.
\end{Lem}
\begin{proof}
Since the dimension can be computed \'etale locally, we may assume that $X$ and $Y$ are equidimensional pfp schemes. We can moreover compute the dimensions of $X$ and $Y$ in terms of the Krull dimensions of their local rings at closed points since $X$ and $Y$ are pfp. So let $x \in X$ be a closed point with image $y \in Y$, then since $f$ is the perfection of a finite type morphism between Noetherian schemes, it follows from \cite[Lemma 00OM]{stacks-project} that 
\begin{align} \label{eq:Inequality}
    \operatorname{Dim} \mathcal{O}_{X,x} \le \operatorname{Dim} \mathcal{O}_{Y,y} + \operatorname{Dim} \mathcal{O}_{f^{-1}(y),x}.
\end{align}
Since $f$ is flat (perfectly smooth morphisms are clearly flat, and flatness can be checked after an fpqc cover (in particular a perfectly smooth cover)), it follows that going down holds for $\mathcal{O}_{X,x} \to \mathcal{O}_{Y,y}$, see \cite[Lemma 00HS]{stacks-project}. The inequality in \eqref{eq:Inequality} is then an equality by \cite[Lemma 00ON]{stacks-project} applied to the local rings of a choice of deperfection of $f$.
\end{proof}
The following lemma has a straightforward proof.
\begin{Lem} \label{Lem:WeaklyPerfectlysmoothRelDim}
If $f:Y \to Z$ is weakly perfectly smooth of relative dimension $d$ at $y \in Y$, then there is an open neighborhood $U$ of $y$ such that $U \cap f^{-1}(f(y))$ is equidimensional of dimension $d$. 
\end{Lem}
A morphism $f:Y \to Z$ is called \emph{weakly perfectly smooth} if there is a perfectly smooth surjection $g:X \to Y$ such that $f \circ g$ is perfectly smooth. The following lemma relates this to Definition \ref{Def:WeaklyPerfectlySmooth}, the proof is straightforward.
\begin{Lem} \label{Lem:WeaklyPerfectlySmooth2}
A morphism $f:Y \to Z$ is weakly perfectly smooth if and only if for all $y \in Y$ the morphism $f$ is weakly perfectly smooth of relative dimension $d_y$ at $y$, for some positive integer $d_y$ which is allowed to depend on $y$.
\end{Lem}

\begin{Lem} \label{Lem:RelDimLocConstant}
Let $f:Y \to Z$ be a weakly perfectly smooth morphism of perfect algebraic spaces. If $Y$ is connected, then $f$ is weakly perfectly smooth of relative dimension $d$ for some $d$.
\end{Lem}
\begin{proof}
It follows from Lemma \ref{Lem:WeaklyPerfectlySmooth2} that for $y \in Y$ there exists a positive integer $d_y$ such that $f$ is weakly perfectly smooth of relative dimension $d_y$ at $y$. Moreover, the same is true for all $u$ in an open neighborhood $U_y$ of $y$. 

Thus if $y,y' \in Y$ with positive integers $d_y,d_y'$ and open neighborhoods $U_y, U_{y'}$, then $U_y \cap U_{y'}$ is non-empty because $Y$ is connected. Therefore there is a point $u \in U_y \cap U_{y'}$ such that $f$ is weakly perfectly smooth of relative dimensions $d$ and $d'$ at $u$. By Lemma \ref{Lem:WeaklyPerfectlysmoothRelDim}, it follows that $d=d'$ and we conclude that $f$ is weakly perfectly smooth of relative dimension $d$.
\end{proof}
\begin{Lem} \label{Lem:WeaklyPerfectlySmoothNormal}
Let $f:Y \to Z$ be a weakly perfectly smooth surjective morphism of pfp algebraic spaces. Then $Z$ is normal if and only if $Y$ is normal. 
\end{Lem}
\begin{proof}
This can be deduced from Lemma \ref{Lem:NormalityIsLocal}.
\end{proof}

\subsubsection{} We follow \cite[Section 04XB]{stacks-project} to define certain properties of morphisms of prestacks on $\affperfk$ that are representable in morphisms of perfect algebraic spaces. For example, a morphism $f:X \to Y$ of prestacks that is representable in perfect algebraic spaces is called \emph{perfectly smooth} if it is representable in perfectly smooth morphisms of perfect algebraic spaces. In other words, if for every morphism $T \to Y$, where $T$ is a perfect algebraic space, the base change $X_T \to T$ is a perfectly smooth morphism of perfect algebraic spaces. 

A \emph{pfp algebraic stack} is a stack $Y$ on $\affperfk$ for the \'etale topology with diagonal representable in pfp algebraic spaces that admits a perfectly smooth surjective\footnote{This means per definition that $f$ is representable in perfectly smooth surjections of perfect algebraic spaces.} map $f:U \to Y$ from a pfp algebraic space. The main example that we will be interested in is the quotient stack\footnote{We always take quotient stacks in the \'etale topology unless otherwise specified.} $[X/G]$ of a pfp algebraic space $X$ by a pfp group scheme $G$. This is a pfp algebraic stack because $X \to [X/G]$ is perfectly smooth since $G$ is perfectly smooth over $\spec k$, see Example \ref{Vb:GroupScheme}. We will also need a notion of weak perfect smoothness for morphisms of pfp algebraic stacks that are not necessarily representable.
\begin{Def} \label{Def:WeaklyPerfectlySmoothII}
A morphism $f: Y \to Z$ of pfp algebraic stacks is called \emph{weakly perfectly smooth}, if there is a perfectly smooth surjective morphism $g: X \to Y$ from a pfp algebraic space $X$ such that the composition $f \circ g$ is perfectly smooth.
\end{Def}
As before, this property is preserved under base change and composition. If $f:Y \to Z$ is representable, then this is (per definition) equivalent to asking that $f:Y \to Z$ is representable in weakly perfectly smooth morphisms of perfect algebraic spaces. 
\begin{Vb} \label{Vb:Torsor}
Let $G$ be a pfp group scheme over $\spec k$, which is perfectly smooth over $\spec k$ by Example \ref{Vb:GroupScheme}. This implies that the natural map $\spec k \to \left[\spec k/G \right]$ is perfectly smooth and thus $\left[\spec k/G \right] \to \spec k$ is weakly perfectly smooth. 
\end{Vb}
\begin{Vb} \label{Vb:Gerbe}
Recall that an \'etale $G$-gerbe over a pfp algebraic stack $Y$ is a morphism $f:X \to Y$ of pfp algebraic stacks that is \'etale locally (on $Y$) of the form $Y \times \left[\spec k/G \right] \to Y$. Since $ \left[\spec k/G \right] \to \spec k$ is weakly perfectly smooth, it follows that $f:X \to Y$ is weakly perfectly smooth because this can be checked \'etale locally on $Y$.
\end{Vb}
\begin{Rem}
In \cite[Definition A.1.13]{XiaoZhu}, a morphism of pfp algebraic stacks satisfying the property in Definition \ref{Def:WeaklyPerfectlySmoothII} is called a perfectly smooth morphism. However, it is not clear to us why a morphism $f:Y \to Z$ of pfp algebraic spaces satisfying the property in Definition \ref{Def:WeaklyPerfectlySmoothII} is perfectly smooth (in the sense defined in the beginning of Section \ref{Sec:Dimensions}), rather than just weakly perfectly smooth. [This result should be true, but we were not able to find a proof].
\end{Rem}
\begin{Lem} \label{Lem:DimensionResult}
Suppose that $X$ is a pfp algebraic space that is equidimensional of dimension $d$ with an action of a pfp group scheme $G$, and let $Y$ be a pfp algebraic space together with a weakly perfectly smooth morphism
\begin{align}
  f:Y \to \left[X/G \right].
\end{align}
Then $Y$ is equidimensional if and only if $f$ is weakly perfectly smooth of relative dimension $n$, where $\operatorname{Dim} Y =d+n-\operatorname{Dim} G$. 
\end{Lem}
\begin{proof}
Consider the fiber product diagram
\begin{equation}
  \begin{tikzcd}
  \tilde{Y} \arrow{r}{\tilde{f}} \arrow{d} & X \arrow{d} \\
  Y \arrow{r}{f} & \left[X/G \right].
  \end{tikzcd}
\end{equation}
The lemma is now a straightforward consequence of Lemma \ref{Lem:WeaklyPerfectlysmoothRelDim}.
\end{proof}

\subsection{Affine flag varieties, moduli stacks of shtukas and forgetful maps}
\subsubsection{} \label{Sec:Parahorics}
Let $k=\ovfp$ and let $\zpbr=W(k)$ and $\qpbr=\zpbr[1/p]$, which come equipped with an automorphism $\sigma$ coming from the absolute Frobenius on $k$. Let $G$ be a connected reductive group over $\mathbb{Q}_p$ and let $B(G,\mathbb{Q}_p)$ (resp. $B(G,\qpbr)$) denote the (extended) Bruhat--Tits building of $G$ over $\mathbb{Q}_p$ (resp. $\qpbr$). For a non-empty bounded subset $\Xi \subset B(G,\qp)$ which is contained in an apartment, we let $G(\qp)_\Xi$ (resp. $G(\qpbr)_{\Xi}$) denote the subgroup of $G(\qp)$ (resp. $G(\qpbr)$) which fixes $\Xi$ pointwise. By the main result of \cite{BT2}, there exists a smooth affine group scheme $\tilde{\mathcal{G}}_{\Xi}$ over $\zp$ with generic fiber $G$ which is uniquely characterised by the property $\tilde{\mathcal{G}}_{\Xi}(\zpbr)=G(\qpbr)_{\Xi}$. We call such a group scheme the Bruhat--Tits stabiliser group scheme associated to $\Xi$. If $\Xi=\{x\}$ is a point, we write $G(\qp)_x$ (resp. $\mathcal{\tilde{G}}_x)$ for $G(\qp)_{\{x\}}$ (resp. $\tilde{\mathcal{G}}_{\{x\}}$). 
  
 For $\Xi\subset B(G,\qp)$ as above, we let $\mathcal{G}_{\Xi}$ denote the `connected stabiliser' (cf. \cite[\S4]{BT2}). We are mainly interested in the case that $\Xi$ is a point or an open facet $\mathfrak{f}$. In this case $\mathcal{G}_{\mathfrak{f}}$ (resp. $\mathcal{G}_x$) is the parahoric group scheme associated to $\mathfrak{f}$ (resp. $x$). 

 We may also consider the corresponding objects over $\qpbr$ and we use the same notation in this case. When it is understood which point of $B(G,\qp)$ or $B(G,\qpbr)$ we are referring to, we simply write $\tilde{\mathcal{G}}$ and $\mathcal{G}$ for the corresponding group schemes. 

 An important case that we need for applications is when $\tilde{\mathcal{G}}_x=\mathcal{G}_x$, i.e., when the parahoric is equal to the Bruhat--Tits stabiliser. When this happens, we necessarily have $\tilde{\mathcal{G}}_{\mathfrak{f}}=\tilde{\mathcal{G}}_x$, where $\mathfrak{f}$ is the facet containing $x$, and $x\in \mathfrak{f}$ is a point `in general position'. A parahoric group scheme $\mathcal{G}$ over $\zp$ (resp. $\zpbr$) is called a \emph{connected parahoric} if there exists $x\in B(G,\qp)$ (resp. $x\in B(G,\qpbr)$) such that $
\mathcal{G}=\mathcal{G}_x=\tilde{\mathcal{G}}_x$. 
 
 Let $\pi_1(G)$ be the algebraic fundamental group of $G \otimes \qpbar$, equipped with its action of $\gal(\qpbar/\qp)$ (see the introduction of \cite{Borovoi}), and let $I \subset \gal(\qpbar/\qp)$ be the inertia group.
\begin{Lem} \label{Lem:TorsionFreeParahoricConnected}
If $\pi(G)_I$ is torsion free, then $\tilde{\mathcal{G}}_x=\mathcal{G}_x$ for all $x$. In other words, all parahoric group schemes are connected parahoric group schemes. 
\end{Lem} 
\begin{proof}
This follows from \cite[Remark 11 of the appendix]{PappasRapoportAffineFlag}.
\end{proof}

\subsubsection{} \label{Sec:AffineWeylGroup}
Let $S \subset G_{\qpbr}$ be a maximal $\qpbr$-split torus defined over $\mathbb{Q}_p$ which exists by \cite[Axiom 4.1.27.UR2]{KalethaPrasad}, and let $T$ be its centraliser. Then $T$ is a maximal torus of $G$ because $G_{\qpbr}$ is quasi-split by a theorem of Steinberg, see \cite[Theorem 2.3.3]{KalethaPrasad}. Choose a $\sigma$-invariant alcove $\mathfrak{a}$ in the apartment of $B(G,\qpbr)$ associated to $S$. Let $N$ be the normaliser of $T$ in $G_{\qpbr}$. We define the \emph{relative Weyl group} as
\begin{align}
  W_0:=N(\qpbr)/T(\qpbr)
\end{align}
and the \emph{Iwahori--Weyl group} (or extended affine Weyl group) as
\begin{align}
  \tilde{W}:=N(\qpbr)/\mathcal{T}(\zpbr),
\end{align}
where $\mathcal{T}$ over $\zpbr$ is the connected Néron model of $T$. There is a short exact sequence (see \cite[Definition 7 of the Appendix]{PappasRapoportAffineFlag})
\begin{align}
  0 \to X_{\ast}(T)_I \to \tilde{W} \to W_0 \to 0,
\end{align}
where $I$ is the inertia group and $X_{\ast}(T)_I$ denotes the inertia coinvariants of the cocharacter lattice $X_{\ast}(T)$ of $T$. The map $X_{\ast}(T)_I \to \tilde{W}$ is denoted on elements by $\lambda \mapsto t^{\lambda}$. Let $\mathbb{S} \subset \tilde{W}$ denote the set of simple reflections in the walls of $\mathfrak{a}$ and let $\tilde{W}_a$ be the subgroup of $\tilde{W}$ generated by $\mathbb{S}$, which we will call the \emph{affine Weyl group}.

Parahoric subgroups $\mathcal{K}$ of $G(\qpbr)$ that contain the Iwahori subgroup corresponding to $\mathfrak{a}$ are called \emph{standard parahoric subgroups}; they correspond to subsets $K \subset \mathbb{S}$ such that the subgroup $W_K$ generated by $K$ is finite; we will call such subsets \emph{types}. This identification is Frobenius equivariant in the sense that $\sigma(\mathcal{K})$ corresponds to $\sigma(K)$. In particular, a subset $K \subset \mathbb{S}$ corresponds to a parahoric subgroup of $G$ if and only if $\sigma(K)=K$; note that our fixed Iwahori subgroup corresponds to $\emptyset \subset \mathbb{S}$. There are parahoric group schemes $\mathcal{G}_K$ over $\zpbr$ associated to types $K$ as above, and we have identifications $\sigma^{\ast} \mathcal{G}_K \simeq \mathcal{G}_{\sigma(K)}$. In particular, if $K$ is stable under $\sigma$, then $\mathcal{G}_K$ is defined over $\mathbb{Z}_p$. The maximal reductive quotient $\mgred$ of the special fiber $\overline{\mG}_K$ of $\mathcal{G}_K$ is a split reductive group over the residue field $k$ of $\qpbr$, and the image of $\mG_{\emptyset}$ in $\mgred$ is a Borel subgroup. The set of simple roots of $\mgred$ with respect to this Borel subgroup can be identified with $K$. The following lemma should be compared with \cite[Remark 4.2.14.b)]{KisinPappas}. 
\begin{Lem} \label{Lem:IwahoriConnected}
Let $J \subset K \subset \mathbb{S}$ and suppose that $\mG_{K}$ is a connected parahoric, then $\mG_{J}$ is a connected parahoric.
\end{Lem}
\begin{proof}
Let $x_K,x_J\in B(G,\qp)$ such that $\mG_{K}=\mathcal{G}_{x_K}$ and $\mG_{J}=\mathcal{G}_{x_J}$. We assume that $x_J$ and $x_K$ are in general position in their respective facets. Then we have $\mathcal{G}_{x_K}=\tilde{\mathcal{G}}_{x_K}$ since $\mathcal{G}$ is a connected parahoric, and we have $\tilde{\mathcal{G}}_{x_J}=\tilde{\mathcal{G}}_{\mathfrak{f}_J}$, where $\mathfrak{f}_J$ is the facet corresponding to $J$. 

Since $x_K$ lies in the closure of $\mathfrak{f}_J$ since $K \supset J$, it follows that $\tilde{\mathcal{G}}_{x_J}(\zpbr)\subset \tilde{\mathcal{G}}_{x_K}(\zpbr)={\mathcal{G}}_{x_K}(\zpbr)$. But ${\mathcal{G}}_{x_K}(\zpbr)$ is contained in the kernel of the Kottwitz map $\kappa:G(\qpbr)\rightarrow \pi_1(G)_I$. Therefore, we have $\tilde{\mathcal{G}}_{x_J}(\zpbr)\subset \ker(\kappa)$ and hence we deduce as in Lemma \ref{Lem:TorsionFreeParahoricConnected} that $\tilde{\mathcal{G}}_{x_J}={\mathcal{G}}_{x_J}$.
\end{proof}

\subsubsection{} \label{Sec:VerySpecialParahoric} A type $K \subset \mathbb{S}$ is called \emph{very special} if $W_K \subset \tilde{W}$ maps isomorphically onto $W_0$. Very special types correspond to very special vertices in $\mathfrak{a}$, see \cite[Lemma 1.3.42, Proposition 1.3.43]{KalethaPrasad}, where they are called extra special vertices. If $K$ is $\sigma$-stable, then the parahoric subgroup $\mathcal{G}_K(\zp)$ associated to a very special type is called a \emph{very special parahoric subgroup}. A fact that will be crucial for us is that there exists a $\sigma$-stable very special type $K$ if $G$ is quasi-split, see \cite[Proposition 10.2.1]{KalethaPrasad}. Thus if $G$ is quasi-split, then the standard Iwahori subgroup $\mathcal{G}_{\emptyset}(\zp)$ contains a very special parahoric subgroup.

\subsubsection{} There is a split short exact sequence (our choice of $\mathfrak{a}$ provides a splitting, see \cite[Lemma 14 of the appendix]{PappasRapoportAffineFlag})
\begin{align} \label{Eq:AffineWeylIwahoriWeyl}
  0 \to \tilde{W}_a \to \tilde{W} \to \pi_1(G)_I \to 0.
\end{align}
The affine Weyl group $\tilde{W}_a$ has the structure of a Coxeter group, and we will use this to define a Bruhat order (denoted by $\le$) and a notion of length on $\tilde{W}$, by splitting \eqref{Eq:AffineWeylIwahoriWeyl} and regarding $\pi_1(G)_I \subset \tilde{W}$ as the subset of length zero elements. We will write $\ell(w)$ for the length of an element of $\tilde{W}$. Similarly, we define a partial order $\le$ and a length function on $\twKK$ by taking minimal length representatives of double cosets. 
\subsubsection{} In this section we will recall some definitions from \cite{Zhu1, XiaoZhu, ShenYuZhang} and state some results. Let the notation be as in Sections \ref{Sec:Parahorics} and \ref{Sec:AffineWeylGroup}, so in particular $G$ denotes a connected reductive group over $\mathbb{Q}_p$. Let $\mG_{K}$ be a parahoric group scheme over $\zpbr$ corresponding to a $\sigma$-stable type $K \subset \mathbb{S}$. For an object $R$ of $\affperfk$ we set
\begin{align}
  D_R=\spec W(R), \qquad D_R^{\ast}=\spec W(R)[1/p],
\end{align}
where $W(R)$ denotes the ring of $p$-typical Witt vectors of $R$. We define group-valued functors on $\affperfk$ sending an object $R$ to
\begin{align}
  LG(R)&:=G(D_R^{\ast}) \\
   \lpGk(R)&:= \mathcal{G}_K(D_R) \\
   L^m \mathcal{G}_K(R)&:=\mathcal{G}_K\left(W(R)/p^m W(R) \right),
\end{align}
which we call the \emph{loop group}, respectively the \emph{positive loop group}, respectively the $m$-\emph{truncated loop group}. It follows from \cite[Section 1.1]{Zhu1} that $L^m \mathcal{G}_K$ and $\lpGk$ are representable in perfect schemes over $k$ and that $\lpGk = \varprojlim_m L^m \mathcal{G}_K$. Moreover, \cite[Proposition 1.1]{Zhu1} tells us that $LG$ is representable by an ind-(perfect scheme), which means that it is isomorphic to an inductive limit of perfect schemes along closed immersions. By \cite[Lemma 1.2.(i)]{Zhu1}, the natural map $\lpGk \to LG$ is a closed immersion. 

\subsubsection{} Fix an algebraic closure $\ovfp$ of $\fp$ and set $k=\ovfp$. Let $R$ be a perfect $k$-algebra and let $\mE$ and $\mF$ be $\mG_{K}$-torsors on $D_R$.\footnote{Here we mean torsor in the \'etale topology on $D_R=\spec W(R)$ in the usual way.} Recall from \cite[Section 3.1.3]{XiaoZhu} that a \emph{modification} $\beta: \mE \dashrightarrow \mF$ is an isomorphism of $G$-torsors
\begin{align}
  \beta: \restr{\mE}{D_R^{\ast}} \to \restr{\mF}{D_R^{\ast}}.
\end{align}
It follows from the proof of \cite[Lemma 1.3]{Zhu1} that there is an \'etale cover $\spec R' \to \spec R$ such that $\mathcal{E}$ is trivial after pullback along $\spec D_{R'} \to \spec D_{R}$. Therefore we can also think of $\mG_{K}$-torsors over $D_R$ as \'etale $L^+\mG_{K}$-torsors over $\spec R$.

We define the (partial) \emph{affine flag variety} $\grg$ to be the functor on $\affperfk$ sending $R$ to the set of isomorphism classes of modifications
\begin{align}
  \alpha:\mE \dashrightarrow \mE^0,
\end{align}
where $\mE$ is a $\mG_{K}$-torsor over $D_R$, and where $\mE^0$ is the trivial $\mG_{K}$-torsor over $D_R$. There is a natural action of $\lG$, thought of as the functor $$R \mapsto \operatorname{Aut}(\restr{\mE^0}{D_R^{\ast}})$$ on $\grg$, by postcomposing $\alpha$ with an automorphism of the restriction to $D_R^{\ast}$ of $\mE^0$, and the orbit of the $k$-point of $\grg$ given by the identity modification $\mE^0 \to \mE^0$ induces a map $O:LG \to \grg$. The map $O$ induces an identification (that we will implicitly use from now on)
\begin{align}
    \grg(k) \simeq G(\qpbr)/\mG_{K}(\zpbr).
\end{align}
It is a result of \cite{Zhu1,BhattScholze} that $\grg$ is representable by an inductive limit of perfections of projective $k$-schemes, with closed immersions as transition maps. In short, $\grg$ is ind-perfectly projective. We also define the \emph{Hecke stack} $\hkg$ to be the presheaf in groupoids on $\affperfk$ sending $R$ to the groupoid of modifications $\beta: \mE \dashrightarrow \mF$. The natural map $\grg \to \hkg$ is an $\lpGk$-torsor for the \'etale topology, where $\lpGk$ acts on $\grg$ via the closed immersion $\lpGk \subset \lG$. 
\subsubsection{} \label{Sec:Shtukas}
Recall from \cite[Definition 4.1.3]{ShenYuZhang} that a (local) $\mathcal{G}_K$-\emph{shtuka} over a perfect $k$-algebra $R$ is a pair $(\mE, \beta)$, where $\mE$ is a $\mG_{K}$-torsor over $D_R$ and where $\beta$ is a modification $\beta:\smE \dashrightarrow \mE$. Here $\sigma:D_R \to D_R$ denotes the Frobenius morphism induced from the absolute Frobenius on $R$, and we consider the restriction of $\sigma^{\ast} \mathcal{E}$ to $D_R^{\ast}$ as a $G$-torsor via the isomorphism $\sigma:\sigma^{\ast} G \to G$, coming from the fact that $G$ is defined over $\mathbb{Q}_p$. A morphism of shtukas $(\mE, \beta) \to (\mE', \beta')$ is an isomorphism $f:\mE \to \mE'$ of $\mG_{K}$-torsors such that the following diagram commutes
\begin{equation}
  \begin{tikzcd}
  \smE \arrow[r, dashed, "\beta"] \arrow{d}{\sigma^{\ast} f} & \mE \arrow{d}{f} \\
  \smE' \arrow[r, dashed, "\beta'"] & \mE'.
  \end{tikzcd}
\end{equation}
We will write $\sht(R)$ for the groupoid of $\mathcal{G}_K$-shtukas over $R$ and $\sht$ for the presheaf in groupoids on $\affperfk$ sending $R$ to $\sht(R)$. 

\begin{Vb}
   Our main examples of shtukas come from $p$-divisible groups. More precisely, for $\mathcal{G}_K=\operatorname{GL}_{n, \mathbb{Z}_p}$, a $\mathcal{G}_K$-shtuka over a perfect ring $R$ is a projective module $M$ of rank $n$ over $W(R)$ together with an isomorphism
\begin{align}
  \beta:\sigma^{\ast}M[1/p] \to M[1/p].
\end{align}
If the map $\beta$ satisfies $pM \subset \beta(\sigma^{\ast}M) \subset M$, then the pair $(M,\beta)$ is a (contravariant) Dieudonn\'e module. By a result of Gabber, see \cite{Lau}, there is a $p$-divisible group over $\spec R$ with contravariant Dieudonn\'e module $(M, \beta)$.
\end{Vb}

\subsubsection{} \label{fibersection}
For an inclusion of types $J \subset K$, there is a natural morphism of parahoric group schemes $\mathcal{G}_J \to \mathcal{G}_K$. The induced morphism on loop groups $L^+\mG_{J} \subset L^+\mG_{K}$ is a closed immersion, since this induced morphism commutes with the natural closed immersions of source and target to $LG$, see \cite[Lemma 1.2.(i)]{Zhu1}.\footnote{Here we are using the cancellation theorem for closed immersions, see e.g. \cite[Theorem 11.1.1]{VakilNotes}.}  

If $J$ and $K$ are $\sigma$-stable, then pushing out torsors along $L^+\mG_{J} \to L^+\mG_{K}$ induces a forgetful map
\begin{align}
  \shtj \to \sht.
\end{align}
In this section we will show that these forgetful maps are representable in perfectly proper algebraic spaces, which is an analogue of \cite[Proposition 8.7]{PappasRapoportAffineFlag}. 

Let $\mgred$ be the maximal reductive quotient of the $1$-truncated loop group $L^1 \mG_{K}=\overline{\mG}_K$ and let $H_J$ be the image of $\mG_{J}$ in $\mgred$; it is a standard parabolic subgroup of type $J \subset K$ (recall that $K$ can be identified with the set of simple roots of $\mgred$ with respect to the Borel $B$ that is the image of $\mG_{\emptyset} \to \mgred$). Recall that for a perfect group scheme $H$ we write $\mathbf{B} H$ for the classifying stack of $H$; in other words, $\mathbf{B}H$ is the groupoid valued functor that sends an object $R$ of $\affperfk$ to the groupoid of $H$ torsors (in the \'etale topology) over $\spec R$. There is a natural morphism $\spec k \to \mathbf{B} H$ corresponding to the trivial $H$-torsor over $\spec k$, which induces an isomorphism $\left[ \spec k /H \right] \to \mathbf{B}H$.
\begin{Lem} \label{representableproper}
The forgetful map $\mathbf{B} \lpGj \to \mathbf{B} \lpGk$ is a $\mgred/H_J$-fibration\footnote{This means that the basechange along $\spec R \to \mathbf{B} \lpGk$ for $R \in \affperfk$ is \'etale locally isomorphic to $\spec R \times \mgred/H_J$.} for the \'etale topology, in particular it is representable in perfectly proper algebraic spaces.
\end{Lem}
\begin{proof}
Let $R$ be a perfect $k$-algebra and let $X$ be an $\lpGk$ torsor over $\spec R$ represented by a map $\spec R \to \mathbf{B} \lpGk$. It follows from the definition of quotient stacks that both squares in the following diagram of stacks are Cartesian
\begin{equation}
  \begin{tikzcd}
  X \arrow{r} \arrow{d} & \spec k \arrow{d} \\
  \lbrack X/\lpGj \rbrack \arrow{r} \arrow{d} & \mathbf{B} \lpGj \arrow{d} \\ 
  \spec R \arrow{r} & \mathbf{B} \lpGk.
  \end{tikzcd}
\end{equation}
By \cite[Lemma 1.3]{Zhu1}, there is an \'etale cover $T \to \spec R$ such that $X_T$ is isomorphic to the trivial $\lpGk$ torsor over $T$, hence $[X/\lpGj]$ is \'etale locally isomorphic to $\spec R \times [\lpGk / \lpGj]$. Therefore it suffices to show that $[\lpGk / \lpGj]$ is representable in a perfectly proper scheme.\footnote{Note that the property of a morphism of pfp algebraic spaces being perfectly proper is \'etale local on the target. This follows from the fact that properness is \'etale local on the target, see \cite[Lemma 02L1]{stacks-project}, in combination with the topological invariance of the \'etale site and \cite[Lemma A.19]{Zhu1}.} We will argue as in the proof of \cite[Proposition 8.7]{PappasRapoportAffineFlag} that there is an isomorphism
\begin{align}
  [\lpGk / \lpGj] \simeq [\mgred/H_{J}]
\end{align}
and the latter is representable in a perfectly proper scheme because it is the perfection of a partial flag variety for $\mgred$. By Lemma \ref{Lem:GroupSchemesCartesian} below, the following 
commutative diagram of perfect group schemes is Cartesian. 
\begin{equation}
    \begin{tikzcd}
    \lpGj \arrow[r,twoheadrightarrow] \arrow[d, hook]& H_J \arrow[d, hook] \\
    \lpGk \arrow[r,twoheadrightarrow, "\varphi"] & \mgred.
    \end{tikzcd}
\end{equation}
It now follows formally that $ [\lpGk / \lpGj] \simeq [\mgred/H_J]$.
\end{proof}
\begin{Cor} \label{Cor:ForgetfulRepresentable}
The map $\shtj \to \sht$ is a $\mgred/H_J$-fibration for the \'etale topology, in particular it is representable in perfectly proper algebraic spaces.
\end{Cor}
\begin{proof}
This follows because the following diagram is Cartesian
\begin{equation} \label{Eq:ForgetfulLocalCartesianClassifying}
  \begin{tikzcd}
  \shtj \arrow{r} \arrow{d} & \sht \arrow{d} \\
  \mathbf{B} \lpGj \arrow{r} & \mathbf{B} \lpGk.
  \end{tikzcd}
\end{equation}
Indeed, this is a straightforward consequence of the definitions. [A $\mG_{J}$-shtuka is the same thing as a $\mG_{K}$-shtuka $(\mE, \beta)$ together with an $\lpGj$-torsor $\mE'$ and an isomorphism $\alpha:\mE' \times_{\lpGj} \lpGk \simeq \mE$, because the natural map $\mG_{J} \to \mG_{K}$ is an isomorphism over $\qp$.]
\end{proof}

\subsubsection{Relative position} \label{Sec:RelPos} It follows from the discussion in \cite[Section 3.6]{HZ} that there is an $\lpGk$-equivariant stratification
\begin{align}
  \grg=\bigcup_{w \in \twKK} \grg(w),
\end{align}
where each $\grg(w)$ is a locally closed subscheme of $\grg$, such that on $k$-points we recover the Bruhat--Tits decomposition
\begin{align}
    \grg(k)=G(\qpbr)/\mG_{K}(\zpbr) = \bigcup_{w \in \twKK} \mG_{K}(\zpbr) \dot{w} \mG_{K}(\zpbr)/\mG_{K}(\zpbr),
\end{align}
see \cite[Proposition 8 of the appendix]{PappasRapoportAffineFlag}. We deduce from this that we get a decomposition
\begin{align}
    \left[\lpGk \backslash \grg \right] =: \hkg = \bigcup_{w \in \twKK} \hkg(w),
\end{align}
where $\hkg(w) = \left[\lpGk \backslash \grg(w) \right]$. It follows moreover from \cite[Section 3.6]{HZ} that the closure of $\grg(w)$ is equal to 
\begin{align}
  \grg(\le \! w):=\bigcup_{w' \le w} \grg(w').
\end{align}
Furthermore, when $K=\emptyset$, the Schubert cell $\grg(w)$ is equidimensional of dimensional equal to the length $\ell(w)$ of $w$. This latter statement is proved for a certain Demazure resolution $D_w \to \grg(\le \! w)$ in\cite[Proposition 3.4]{HZ}, and follows for $\grg(\le \! w)$ (and hence $\grg(w)$) since the map $D_w \to \grg(\le \! w)$ is birational as explained in the proof of \cite[Proposition 3.7]{HZ}.
\subsubsection{} \label{Sec:AdmissibleSetLocalModels}
Let $\{\mu\}$ be a $G(\qpbar)$-conjugacy class of cocharacters of $G_{\overline{\mathbb{Q}}_p}$. Recall that we fixed a maximal torus $T$ of $G$ in Section \ref{Sec:AffineWeylGroup}. Choose a Borel $B$ of $G_{\qpbr}$ containing $T_{\qpbr}$ and let $\overline{\mu}$ be the image in $X_{\ast}(T)_I$ of a $B$-dominant representative of $\{\mu\}$. The set of $\{\mu\}$-admissible elements is defined as
\begin{align} \label{Eq:Admu}
  \admu=\{ w \in \tilde{W} \; : \; w \le t^{x(\overline{\mu})} \text{ for some } x \in W_0 \}.
\end{align}
There is a unique element $\tau=\tau_{\mu} \in \admu$ of length zero and in fact $\admu \subset \tilde{W}_a \tau$. For $K$ a $\sigma$-stable type, we define $\admuk$ as the image of $\admu$ under $\tilde{W} \to W_K \backslash \tilde{W} / W_K$. We write $\kadmu$ for $\admu \cap \! {}^{K} \tilde{W}$, where $\! {}^{K} \tilde{W} \subset \tilde{W}$ denotes the subset of elements that are of minimal length in their left $W_K$-coset. \smallskip

If $\{\mu\}$ is minuscule and $K$ is a $\sigma$-stable type, then we define the \emph{perfect local model} attached to $\{\mu\}$ and $K$ to be the perfectly projective closed subscheme of $\grg$ given by
\begin{align}
  \mloc:=\bigcup_{w \in \admuk} \grg(w).
\end{align}
This definition is motivated by the discussion in \cite[Section 2.1.7]{ShenYuZhang} and in particular \cite[Corollary 2.1.11]{ShenYuZhang}. It follows from the discussion in \cite[Section 2.1.7]{ShenYuZhang} that the scheme $\mloc$ is equidimensional of dimension $d=\langle 2 \rho, \overline{\mu} \rangle$, which is precisely the dimension of the flag variety for $G$ associated to $\mu$.

\subsubsection{} Let $\{\mu\}$ be a conjugacy class of cocharacters of $G_{\overline{\mathbb{Q}}_p}$ as above, and let $\admuk$ be the $\mu$-admissible set. Recall that the stack $\hkg$ is the moduli stack of modifications $\mE \dashrightarrow \mF$ of $\lpGk$-torsors. We define a map $\operatorname{Rel}:\sht \to \hkg$ sending $(\mE, \beta)$ to $\beta:\smE \dashrightarrow \mE$. For $w \in \twKK$, we have the locally closed substack $\hkg(w) \subset \hkg$ from Section \ref{Sec:RelPos}, and its pullback along $\operatorname{Rel}$ defines a locally closed substack
\begin{align}
  \sht(w) \subset \sht.
\end{align}
Following \cite[Definition 4.1.3]{ShenYuZhang}, we define the stack of shtukas of level $\mathcal{G}_K$ and type $\mu$ to be
\begin{align}
  \shtmu:=\bigcup_{w \in \admuk} \sht(w);
\end{align}
it is a closed substack of $\sht$ by the discussion in Section \ref{Sec:RelPos}. If $J \subset K$ is another $\sigma$-stable type, then the following diagram commutes by definition of $\admu_J$ and $\admuk$ (but it is generally not Cartesian)
\begin{equation} \label{Eq:Forgetfulmu}
  \begin{tikzcd}
  \shtjmu \arrow[r, hook] \arrow{d} & \shtj \arrow{d} \\
  \shtmu \arrow[r, hook] & \sht.
  \end{tikzcd}
\end{equation}
\begin{prop} \label{Prop:Representable}
  The forgetful morphism $\shtjmu \to \shtmu$ is representable in perfectly proper algebraic spaces. 
\end{prop}
\begin{proof}
We know that $\shtjmu \to \shtj$ is representable in perfectly proper algebraic spaces because it is a closed immersion, and the map $\shtj \to \sht$ is representable in perfectly proper algebraic spaces by Corollary \ref{Cor:ForgetfulRepresentable}. The composition is thus representable in perfectly proper algebraic spaces and factors over $\shtmu$, which proves the result.
\end{proof}

\subsection{Restricted local shtukas and forgetful maps} \label{Sec:RestrictedLocalShtukas}
We will recall some results from \cite[Section 4.2]{ShenYuZhang}. Fix a geometric conjugacy class of minuscule cocharacters $\{\mu\}$ of $G_{\overline{\mathbb{Q}}_p}$ for the rest of this section, and let $\admuk$ be the $\mu$-admissible set. Recall from \cite[Lemma 4.1.4]{ShenYuZhang} that $\sht$ has the following quotient description: Let $\sigma: \lpGk \to \lpGk$ be the relative Frobenius morphism and let $\lpGk$ act on $\lG$ via $h \cdot g = (h^{-1} g \sigma(h))$, we denote this action by $\operatorname{Ad}_{\sigma}$. With this notation, there is an isomorphism\footnote{Here we are taking quotient stacks in the \'etale topology. Note that this shows that $\sht$ is a stack in the \'etale topology. }
\begin{align} \label{Eq:quotientdescription}
  \sht \simeq \left[\frac{\lG}{\operatorname{Ad}_{\sigma} \lpGk} \right].
\end{align}
The map $\Lambda:\lG \to \sht$ constructed this way corresponds to a shtuka over $\lG$: It is the modification $\beta:\mE^0_{\lG} \simeq \sigma^{\ast} \mE^0_{\lG} \dashrightarrow \mE^0_{\lG}$ given by the tautological element in $\lG$. Moreover the map $\lG \to \sht$ is precisely the universal $L^+\mG_{K}$-torsor over $\sht$.

Consider the following commutative diagram
\begin{equation}
    \begin{tikzcd}
    LG \arrow{d}{\Lambda} \arrow{r}{O} & \grg \arrow{d} \\
    \sht \arrow{r}{\operatorname{Rel}} & \hkg.
    \end{tikzcd}
\end{equation}
There is a closed subscheme $\mlocinf \subset \lG$ defined to be the inverse image of $\mloc \subset \grg$ under $O:\lG \to \grg$. Since $\mloc \subset \grg$ is stable under the action of $\lpGk$, it follows that $\mlocinf \subset \lG$ is stable under the $\operatorname{Ad}_{\sigma}$-action of $\lpGk$. The discussion in the previous paragraph, along with the commutative diagram, tells us that there is a natural identification
\begin{align}
  \shtmu \simeq \left[\frac{\mlocinf}{\operatorname{Ad}_{\sigma} \lpGk} \right].
\end{align}
For $J \subset K$ a $\sigma$-stable subset, there is a closed immersion $\mlocjinf \subset \mlocinf$
which identifies
\begin{align} \label{Eq:IdentifiesPreimageLocalModel}
  \frac{\mlocinf}{L^+ \mG_{J}} \subset \gr_J
\end{align}
with the preimage of $\mloc$ under $\gr_J \to \gr_K$.

\subsubsection{} Let $\beta_K: \lpGk \to \mgred$ be the natural map, where $\mgred$ is the maximal reductive quotient of $\overline{\mG}_K=L^1 \mG_{K}$. Define $\mloconered:=\ker \beta_K \backslash \mlocinf$; it is a $\mgred$-torsor over $\mloc$. We then define
\begin{align}
  \shtmu^{(\infty,1)}:=\left[\frac{\mloconered}{\operatorname{Ad}_{\sigma} \lpGk} \right].
\end{align}
It follows from Lemma \ref{Lem:LimitLemma} that the twisted conjugation action of $\lpGk$ on $\mloconered$ factors through the action of $L^m \mG_{K}$ for $m \gg 0$. Indeed, this follows by applying the lemma to the inverse system $\{L^m \mG_K \times \mloconered\}_{m \in \mathbb{Z}}$ and the action map $\lpGk \times \mloconered \to \mloconered$ from the inverse limit to the pfp algebraic space $\mloconered$. For such $m$ we define the stack of $(m)$-restricted shtukas of type $\{\mu\}$ by
\begin{align}
  \shtmu^{(m,1)}:=\left[\frac{\mloconered}{\operatorname{Ad}_{\sigma} L^m \mG_{K}} \right].
\end{align}
Note that there are natural morphisms
\begin{align} \label{Eq:ShtukasToRestrictedShtukas}
  \shtmu=\left[\frac{\mlocinf}{\operatorname{Ad}_{\sigma} \lpGk} \right] \to \left[\frac{\mloconered}{\operatorname{Ad}_{\sigma} \lpGk} \right] \to \left[\frac{\mloconered}{\operatorname{Ad}_{\sigma} L^m \mG_{K}} \right]=\shtmu^{(m,1)}
\end{align}
induced by the natural map $\mlocinf \to \mloconered$ and the natural map $\lpGk \to \lmG$.
\begin{Rem} \label{Rem:ShtukaLocalModelII}
There is a `local model diagram'
\begin{equation}
  \begin{tikzcd}
  & \mloconered \arrow{dr} \arrow{dl} \\
  \shtmu^{(m,1)} & & \mloc.
  \end{tikzcd}
\end{equation}
The left-hand map is an $L^m \mG_{K}$-torsor while the right-hand map is a $\mgred$-torsor. In particular, the stack $\shtmu^{(m,1)}$ is an equidimensional pfp algebraic stack.\footnote{A quotient stack $[X/G]$, where $G$ is a pfp group scheme over $k$ and $X$ is a pfp algebraic space, is defined to be equidimensional if $X$ is equidimensional. The dimension of $[X/G]$ is defined to be $\operatorname{Dim} X - \operatorname{Dim G}$; this is well-defined in view of Lemma \ref{Lem:DimensionResult}.} Indeed $\mloc$ is pfp and equidimensional and since the right-hand map is perfectly smooth of relative dimension $\operatorname{Dim} \mgred$ we find that $\mloconered$ is pfp and equidimensional by Lemma \ref{Lem:Dimension}. 
\end{Rem}

\subsubsection{} \label{Sec:LocalShtukasComparison} The goal of this section is to compare $\shtmu^{(m,1)}$ and $\shtemu^{(m',1)}$.  Unfortunately, there is no natural map between them when $K \not=\emptyset$. However, we will be able to construct a correspondence between them instead, and study its properties, see Proposition \ref{Prop:CartesianInfiniteToFinite}, Lemma \ref{Lem:OneMoreWeaklyPerfectlySmooth} and Section \ref{Sec:GeneralRestrictedShtukas}.

Consider the closed immersion $\lpGe \subset \lpGk$, which induces a closed immersion $B \subset \mgred$, where $B$ is the image of $L^1 \mG_{\emptyset}$ in $\mgred$; let $\gamma:\lpGe \to B$ be the natural surjection. By Lemma \ref{Lem:LimitLemma}, we can choose $m \gg 0$ such that the action $\operatorname{Ad}_{\sigma} \lpGk$ on $\ker \gamma \backslash \mlocinf$ factors through $\lmG$. As in equation \eqref{Eq:ShtukasToRestrictedShtukas}, the natural maps
\begin{align}
    \mlocinf &\to \ker \gamma \backslash \mlocinf \\
    \lpGk &\to \lmG
\end{align}
induce a natural map
\begin{align}
    \shtmu=\left[\frac{\mlocinf}{\operatorname{Ad}_{\sigma} \lpGk} \right] \to \left[ \frac{ \ker \gamma \backslash \mlocinf}{ \operatorname{Ad}_{\sigma} \lmG} \right].
\end{align}
For $m$ as above, let $H_m$ be the image of $\lpGe$ in $\lmG$. Since $\mloceinf \subset \mlocinf$, it follows that the action of $\lpGe \subset \lpGk$ on $\ker \gamma \backslash \mloceinf$ factors through $H_m$. Therefore there is a natural map
\begin{align}
    \shtemu=\left[\frac{\mloceinf}{\operatorname{Ad}_{\sigma} \lpGe} \right] \to \left[ \frac{ \ker \gamma \backslash \mloceinf}{ \operatorname{Ad}_{\sigma} H_m} \right]
\end{align}
induced by $\mloceinf \to \ker \gamma \backslash \mloceinf$ and $\lpGe \to H_m$.
\begin{Prop} \label{Prop:CartesianInfiniteToFinite}
If $m$ is an integer such that the action $\operatorname{Ad}_{\sigma} \lpGk$ on $\ker \gamma \backslash \mlocinf$ factors through $\lmG$, then the diagram
\begin{equation}
\begin{tikzcd}
    \shtemu \arrow{r} \arrow{d} & \left[ \frac{ \ker \gamma \backslash \mloceinf}{ \operatorname{Ad}_{\sigma} H_m} \right] \arrow{d} \\
    \shtmu \arrow{r} & \left[ \frac{ \ker \gamma \backslash \mlocinf}{ \operatorname{Ad}_{\sigma} \lmG} \right],
\end{tikzcd}    
\end{equation}
where the right vertical map is induced by the closed immersions $\mloceinf \xhookrightarrow{} \mlocinf$ and $H_m \xhookrightarrow{} \lmG$, is Cartesian. 
\end{Prop}
We start by proving a lemma.
\begin{Lem} \label{Lem:GroupSchemesCartesian}
Both squares in the following diagram of perfect group schemes are Cartesian.
\begin{equation} \label{Eq:GroupSchemesCartesian}
  \begin{tikzcd}
  \lpGe \arrow[d, hook] \arrow[r, twoheadrightarrow] & H_m \arrow{r} \arrow[d, hook] & B \arrow[d, hook] \\
  \lpGk \arrow[r, twoheadrightarrow] & \lmG \arrow{r} & \mgred.
  \end{tikzcd}
\end{equation}
\end{Lem}
\begin{proof}
We first check that the outer square is Cartesian: It is enough to check this on $k'$-points for all algebraically closed fields $k'$ because $\lpGe \to \lpGk$ is a closed immersion by \cite[Lemma 1.2.(i)]{Zhu1} and perfect schemes are reduced. The result on the level of $k'$-points is \cite[Theorem 4.6.33]{BT2}.

The left square is Cartesian by definition of $H_m$ and it therefore follows from general properties of Cartesian squares that the right square is also Cartesian. 
\end{proof}
\begin{Lem} \label{Lem:DimensionStacks}
The stacks 
\begin{align}
    \left[ \frac{ \ker \gamma \backslash \mloceinf}{ \operatorname{Ad}_{\sigma} H_m} \right] \text{ and } \left[ \frac{ \ker \gamma \backslash \mlocinf}{ \operatorname{Ad}_{\sigma} \lmG} \right]
\end{align}
are equidimensional of the same dimension.
\end{Lem}
\begin{proof}
To compute the dimensions we note that it follows from the right Cartesian square in Lemma \ref{Lem:GroupSchemesCartesian} that
\begin{align}
    \operatorname{Dim} H_m = \operatorname{Dim} \lmG - (\operatorname{Dim} \mgred - \operatorname{Dim} B)
\end{align}
and thus it suffices to show that
\begin{align}
    \operatorname{Dim} \left(\ker \gamma \backslash \mloceinf\right) = \operatorname{Dim} \left(\ker \gamma \backslash \mlocinf\right) - \operatorname{Dim} \mgred + \operatorname{Dim} B.
\end{align}
The map $\ker \gamma \backslash \mloceinf \to \mloce$ is a $B$-torsor by construction and $\ker \gamma \backslash \mlocinf \to \mloc$ is a $\mgred$-torsor by construction, see \eqref{Eq:GroupSchemesCartesian}. Therefore the equality above is equivalent to the equality
\begin{align}
    \operatorname{Dim} \mloc = \operatorname{Dim} \mloce,
\end{align}
which is true, see Section \ref{Sec:AdmissibleSetLocalModels}.
\end{proof}
\begin{proof}[Proof of Proposition \ref{Prop:CartesianInfiniteToFinite}]
Consider the following diagram, where the maps are defined as in \eqref{Eq:ShtukasToRestrictedShtukas}
\begin{equation} \label{Eq:BigDiagramOne}
  \begin{tikzcd}
  \shtemu \arrow[r, equals] \arrow{d} &\left[\frac{\mloceinf}{\operatorname{Ad}_{\sigma} \lpGe} \right] \arrow{d} \arrow{r} & \left[ \frac{ \ker \gamma \backslash \mloceinf}{ \operatorname{Ad}_{\sigma} \lpGe} \right] \arrow{d} \arrow{r} & \left[ \frac{ \ker \gamma \backslash \mloceinf}{ \operatorname{Ad}_{\sigma} H_m} \right] \arrow{d} \\
  \shtmu \arrow[r, equals] &\left[\frac{\mlocinf}{\operatorname{Ad}_{\sigma} \lpGk} \right] \arrow{r} & \left[\frac{\ker \gamma \backslash \mlocinf}{\operatorname{Ad}_{\sigma} \lpGk} \right] \arrow{r} & \left[ \frac{ \ker \gamma \backslash \mlocinf}{ \operatorname{Ad}_{\sigma} \lmG} \right].
  \end{tikzcd}
\end{equation}
It follows from Lemma \ref{Lem:GroupSchemesCartesian} that $P^m:=\operatorname{Ker} \left(\lpGk \to \lmG \right)$ is contained in $\lpGe$ and that $P^m$ is also equal to the kernel of $\lpGe \to H_m$. 

We deduce that the right horizontal maps are both $P^m$ gerbes. Therefore the map from the top-left term of the right-most square to the fiber product is a morphism of $P^m$-gerbes, and thus an isomorphism. Similarly, the middle horizontal maps are both $\ker \gamma$-torsors. Therefore the map from the top-left term of the middle square to the fiber product is a morphism of $\ker \gamma$-torsors and thus an isomorphism. We deduce that the outer square of the diagram is Cartesian.
\end{proof}
\subsubsection{} \label{Sec:LocalShtukasToWeirdLocalShtukas}
In this section we record two more lemmas.
\begin{Lem} \label{Lem:Cofinal}
For each integer $m' \ge 1$ there is an integer $m \gg m'$ such that there is an inclusion 
$\ker \left(\lpGe \to H_{m}\right) \subset \ker \left(\lpGe \to L^{m'}\mG_{\emptyset} \right)$ of closed subschemes of $\lpGe$.
\end{Lem}
\begin{proof}
Fix $m'$. Recall that 
\begin{align}
    \lpGk &\simeq \varprojlim_m \lmG \\
    \lpGe &\simeq \varprojlim_m \lmGe,
\end{align}
and the first of these equalities moreover implies that $\lpGe \simeq \varprojlim_m H_m$. The lemma now follows from Lemma \ref{Lem:LimitLemma}.
\end{proof}
It follows from Lemma \ref{Lem:Cofinal} that for each $m'$ there is an $m \gg m'$ such that the natural map $\lpGe \to L^{m'} \mG_{\emptyset}$ factors through the natural map $\lpGe \to H_m$ via a surjection $H_m \to L^{m'} \mG_{\emptyset}$. Note moreover that $\lpGe \to L^1 \mG_{\emptyset}=\overline{\mG}_{\emptyset} \to \mgrede$ factors through $\lpGe \to L^1 \mG_{\emptyset} \to B$ because the maximal reductive quotient of $\overline{\mG}_{\emptyset}$ is isomorphic to the maximal reductive quotient of $B$. Thus there is a natural map $\ker \gamma \to \ker (\lpGe \to \mgrede)$ which induces a map $\mloceone \to \mloceonered$. Recall moreover that for $m \gg 0$ the action $\operatorname{Ad}_{\sigma} \lpGe$ on $\mloceonered$ factors through an action of $L^{m'}\mG_{\emptyset}$.
\begin{Lem} \label{Lem:OneMoreWeaklyPerfectlySmooth}
Let $m' \gg 0$ be a positive integer and let $m \gg m'$ satisfy the conclusion of Lemma \ref{Lem:Cofinal}. Then the map (induced by $\mloceone \to \mloceonered$ and $H_m \to L^{m'}\mG_{\emptyset}$)
\begin{align}
    \left[\frac{\mloceone}{\operatorname{Ad}_{\sigma} H_m} \right] \to \left[\frac{\mloceonered}{\operatorname{Ad}_{\sigma} L^{m'}\mG_{\emptyset}} \right] = \shtemu^{(m',1)}
\end{align}
is weakly perfectly smooth.
\end{Lem}
\begin{proof}
The natural map 
\begin{align}
  \left[\frac{\mloceone}{\operatorname{Ad}_{\sigma} H_m} \right] \to \left[\frac{\mloceonered}{\operatorname{Ad}_{\sigma} H_m} \right] = \shtemu^{(m',1)}
\end{align}
is a torsor for $\ker(B \to \mgrede)$ and hence weakly perfectly smooth. The natural map 
\begin{align}
  \left[\frac{\mloceonered}{\operatorname{Ad}_{\sigma} H_m} \right] \to \left[\frac{\mloceonered}{\operatorname{Ad}_{\sigma} L^{m'}\mG_{\emptyset}} \right]
\end{align}
is a gerbe for $\ker(H_m \to L^{m'}\mG_{\emptyset})$ and is thus weakly perfectly smooth. It follows that the composition is weakly perfectly smooth, and the lemma is proved.
\end{proof}
\subsubsection{} \label{Sec:GeneralRestrictedShtukas}
The stack $\left[ \frac{ \ker \gamma \backslash \mlocinf}{ \operatorname{Ad}_{\sigma} \lmG} \right]$ is not a stack of restricted shtukas in the sense of Shen--Yu--Zhang \cite{ShenYuZhang}. However, it is closely related to the more general stacks of restricted shtukas introduced in \cite[Section 5.3]{XiaoZhu}. We define for $n \ge 2$ the quotient
\begin{align}
  \mlocn:=\ker \left( \lpGk \to L^n \mG_{K} \right) \backslash \mlocinf.
\end{align}
Then by Lemma \ref{Lem:LimitLemma}, for $m \gg n$ the action $\operatorname{Ad}_{\sigma} \lpGk$ on $\mlocn$ will factor through $\lmG$ and we define 
\begin{align}
  \shtmu^{(m,n),\mathrm{loc}}:=\left[\frac{\mlocn}{\operatorname{Ad}_{\sigma} \lmG} \right].
\end{align}
We have added the `loc' in the superscript and the condition that $n \ge 2$ so that these are not confused with the previously introduced stacks of restricted shtukas (since the notation is \emph{not} compatible). \smallskip

The proof of Lemma \ref{Lem:Cofinal} shows that for $n \gg 0$ we have an inclusion $\ker \left( \lpGk \to L^n \mG_{K} \right) \subset \ker \gamma$ and thus a natural map
\begin{align}
    \mlocn & \to \ker \gamma \backslash \mlocinf.
\end{align}
This induces a morphism (for $m \gg n$ as before)
\begin{align}
  \shtmu^{(m,n), \mathrm{loc}} = \left[\frac{\mlocn}{\operatorname{Ad}_{\sigma} \lmG} \right] \to \left[\frac{\ker \gamma \backslash \mlocinf}{\operatorname{Ad}_{\sigma} \lmG} \right],
\end{align}
which is a torsor for the image of $\ker \gamma$ in $L^n \mG_K$, and thus perfectly smooth.

\subsubsection{The EKOR stratification} \label{Sec:EKOR}
Recall that $\kadmu$ is the intersection of $\admu$ with $^{K} \tilde{W}$, where $^{K} \tilde{W} \subset \tilde{W}$ denotes the subset of elements that are of minimal length in their left $W_K$-coset. By \cite[Lemma 4.2.4]{ShenYuZhang}, the underlying topological space of $\shtmu^{(m,1)}$ is isomorphic to $\kadmu$ equipped with the partial order topology (for the partial order $\preceq$ on $\kadmu$ introduced in \cite[page 3123]{ShenYuZhang}). They use this to define locally closed substacks $\shtmu^{(m,1)}\{w\}$ for $w \in \kadmu$ such that the locally closed substack
\begin{align}
  \shtmu^{(m,1)}\{\preceq w\} := \bigcup_{w' \preceq w} \shtmu^{(m,1)}\{w'\}.
\end{align}
is closed. This allows us to define the Ekedahl--Kottwitz--Oort--Rapoport (EKOR) stratification on any stack mapping to $\shtmu^{(m,1)}$; for example on $\shtmu$ via \eqref{Eq:ShtukasToRestrictedShtukas} and later on Shimura varieties of Hodge type. Note that if $K=\emptyset$ then the EKOR stratification agrees with the Kottwitz--Rapoport (KR) stratification from Section \ref{Sec:AdmissibleSetLocalModels} and $\preceq$ agrees with $\le$. This follows from \cite[Section 1.3.2]{ShenYuZhang} and the discussion preceding \cite[Proposition 4.2.5]{ShenYuZhang}.

\subsection{Affine Deligne--Lusztig varieties} \label{Sec:SigmaConjugacyClasses}
Recall from \cite[Section 2.3]{RapoportRichartz} that there is a partial order on the set $B(G)$ of $\sigma$-conjugacy classes in $G(\qpbr)$. Let $\{\mu\}$ be a $G(\qpbar)$-conjugacy class of cocharacters of $G_{\qpbar}$ and let $B(G, \{\mu\}) \subset B(G)$ be the set of \emph{neutral acceptable} $\sigma$-conjugacy classes with respect to $\{\mu\}$, see \cite[Definition 2.5]{RapoportViehmann}. 
\subsubsection{}
Let $\mE$ be an $\lG$-torsor over $k'$, with $k'$ an algebraically closed field of characteristic $p$, and let $\beta:\smE \to \mE$ be an isomorphism where $\sigma$ is the absolute Frobenius. After choosing a trivialisation of $\mE$, we see that $\beta$ can be represented by an element $b_{\beta} \in G(W(k')[1/p])$ well-defined up to $\sigma$-conjugacy. Since the set of $\sigma$-conjugacy classes in $G(W(k')[1/p])$ does not depend on the choice of algebraically closed field $k'$, it thus gives us an element $[b_{\beta}] \in B(G)$.

Let $R$ be a perfect $k$-algebra, let $\mE$ be an $\lG$-torsor over $R$ and let $\beta:\smE \to \mE$ be an isomorphism. If $x \in \spec R$ and $K$ is an algebraic closure of the residue field $k(x)$, then we will write $[b_{\beta}(x)] \in B(G)$ for the $\sigma$-conjugacy class of the pullback of $(\mathcal{E}, \beta)$ along $\spec K \to \spec R$. Then for $[b] \in B(G)$, the subset (using the partial order introduced above)
\begin{align}
  (\spec R)_{\le [b]}:=\{ x \in \spec R \: : \: [b_{\beta}(x)] \le [b] \}
\end{align}
is closed in $\spec R$ by \cite[Theorem 3.6.(ii)]{RapoportRichartz} and
\begin{align}
  (\spec R)_{[b]}:=\{ x \in \spec R \: : \: [b_{\beta}(x)] = [b] \}
\end{align}
is locally closed. 

\subsubsection{} Given $\spec R \to \sht$ corresponding to a $\mathcal{G}_K$-shtuka $(\mathcal{E}', \beta)$, we can set $\mathcal{E}$ to be the pushout of $\mathcal{E}'$ along $\lpGk \to \lG$ to obtain a pair $(\mathcal{E}, \beta)$ as above. Then we may form the locally closed subsets $(\spec R)_{[b]} \subset \spec R$ as above. This allows us to define a stratification
\begin{align}
  \sht:=\bigcup_{[b] \in B(G)} \shtb,
\end{align}
where $\shtb$ denotes the locally closed substack of $\sht$ whose $R$-points are given by the full subgroupoid
\begin{align}
    \shtb(R) \subset \sht(R)
\end{align}
of maps $\spec R \to \sht$ such that $(\spec R)_{[b]}=\spec R$. We will write $\shtmub$ for the intersection (fiber product over $\sht$) of $\shtmu$ and $\shtb$; we will see in Corollary \ref{Cor:NonemptyNewtonStrata} that this is non-empty if and only if $[b] \in B(G, \{\mu\})$.
\subsubsection{} \label{Sec:ADLVSet}
Let $K \subset \mathbb{S}$ be a $\sigma$-stable type and let $b \in G(\qpbr)$. Then we define the affine Deligne--Lusztig set
\begin{align}
  \xmub=\{g \in G(\qpbr)/\mathcal{G}_K(\zpbr) \; | \; g^{-1} b \sigma(g) \in \bigcup_{w \in \admuk} \mathcal{G}_K(\zpbr) \dot{w} \mathcal{G}_K(\zpbr)/ \mathcal{G}_K(\zpbr)\}.
\end{align}
Let $J_b$ be the algebraic group over $\mathbb{Q}_p$ whose $R$-points are given by
\begin{align}
  J_b(R)=\{g \in G(\qpbr \otimes_{\mathbb{Q}_p} R) \: | \; g^{-1} b \sigma (g) = b\}.
\end{align}
Then $J_b(\qp) \subset G(\qpbr)$ acts on $\xmub$ via left multiplication. By \cite[Theorem 1.1]{He2}, the set $\xmub$ is non-empty if and only if $[b] \in B(G, \{\mu\})$. Moreover \cite[Theorem 1.1]{He2} says that for $J \subset K$ another $\sigma$-stable type, the natural projection $G(\qpbr)/\mathcal{G}_{J}(\zpbr) \to G(\qpbr)/\mathcal{G}_{K}(\zpbr)$ induces a $J_b(\mathbb{Q}_p)$-equivariant surjection
\begin{align}
  \xmubj \to \xmub.
\end{align}
We will soon see that $\xmub$ can be identified with the set of $k$-points of a perfect scheme over $k$, which we will also denote by $\xmub$.
\subsubsection{} 
Let $K$ be a $\sigma$-stable type, let $b \in G(\qpbr)$ and consider the functor $\xmub'$ on $\affperfk$ sending $R$ to the set of isomorphism classes of commutative diagrams of modifications of $\mathcal{G}_K$-torsors on $D_R$
\begin{equation} \label{Eq:ADLVMOD}
  \begin{tikzcd}
  \smE_1 \arrow[r, dashrightarrow, "\beta_1"] \arrow[d, dashrightarrow, "\sigma^{\ast}\beta_0"] & \mE_1 \arrow[d, dashrightarrow, "\beta_0"] \\
  \smE^0 \arrow[r, dashrightarrow, "b"] & \mE^0,
  \end{tikzcd}
\end{equation}
such that $\beta_1:\smE_1 \dashrightarrow \mE_1$, considered as an element of $\hkg(R)$, lies in $\left(\bigcup_{w \in \admuk} \hkg(w)\right)(R)$. Here $b$ is the modification of the trivial $G_K$-torsor $\smE^0 \simeq \mE^0$ given by multiplication by $b$. We will sometimes refer to $\beta_0$ as a \emph{quasi-isogeny} of shtukas from $(\mE_1, \beta_1) \to (\mE^0, b)$. 
\begin{Lem} \label{Lem:ADLVLemma}
The morphism $\xmub' \to \grg$ that sends a diagram as in \eqref{Eq:ADLVMOD} to $\beta_0:\mE_1 \dashrightarrow \mE^0$, is a closed immersion. Moreover it identifies
\begin{align}
  \xmub'(k) \subset \grg(k) = G(\qpbr)/\mathcal{G}_K(\zpbr)
\end{align}
with the affine Deligne--Lusztig set $\xmub$ from Section \ref{Sec:ADLVSet}.
\end{Lem}
\begin{proof}
Consider the functor $X(b)$ sending $R$ to the set of isomorphism classes of commutative diagrams of modifications of $\mathcal{G}_K$-torsors on $D_R$
\begin{equation} \label{Eq:ADLVMOD2}
  \begin{tikzcd}
  \smE_1 \arrow[r, dashrightarrow, "\beta_1"] \arrow[d, dashrightarrow, "\sigma^{\ast}\beta_0"] & \mE_1 \arrow[d, dashrightarrow, "\beta_0"] \\
  \smE^0 \arrow[r, dashrightarrow, "b"] & \mE^0
  \end{tikzcd}
\end{equation}
as before, but now \emph{without} the condition that $\beta_1 \in \left(\bigcup_{w \in \admuk} \hkg(w)\right)(R)$. As before, \cite[the discussion after Remark 3.5]{HZ} tells us that $\xmub'$ is a closed subfunctor of $X(b)$, and the lemma would follow if we could show that the map
\begin{align}
  f:X(b) \to \grg
\end{align}
sending a diagram as in \eqref{Eq:ADLVMOD2} to $\beta_0:\mE_1 \dashrightarrow \mE^0$ is an isomorphism. The map $f$ is an isomorphism because the map $g:\grg \to X(b)$ sending $\beta_0:\mE_1 \dashrightarrow \mE^0$ to the diagram
\begin{equation}
  \begin{tikzcd}
  \smE_1 \arrow[r, dashrightarrow, "\beta_1"] \arrow[d, dashrightarrow, "\sigma^{\ast}\beta_0"] & \mE_1 \arrow[d, dashrightarrow, "\beta_0"] \\
  \smE^0 \arrow[r, dashrightarrow, "b"] & \mE^0.
  \end{tikzcd}
\end{equation}
with $\beta_1=\beta_0^{-1} b \sigma^{\ast} \beta_0$ is an inverse to $f$. We see that $\xmub'(k)$ is cut out from $X(b)(k)= G(\qpbr)/\mathcal{G}_K(\zpbr)$ by the condition that $\beta_1 \in \bigcup_{w \in \admuk} \hkg(w)(k)$, in other words, that
\begin{align}
  \beta_0^{-1} b \sigma^{\ast} \beta_0 \in \bigcup_{w \in \admu} \mathcal{G}_K(\zpbr) \dot{w} \mathcal{G}_K(\zpbr)/\mathcal{G}_K(\zpbr).
\end{align}
This is precisely the condition cutting out $\xmub \subset G(\qpbr)/\mathcal{G}_K(\zpbr)$, and so we are done.
\end{proof}
From now on, we will write $\xmub$ for $\xmub'$ by abuse of notation. It follows from \cite[Lemma 1.1]{HamacherViehmann} and \cite[Corollary 2.5.3]{ZhouZhu} that $\xmub$ is actually a perfect scheme that is perfectly locally of finite type. 

If $b'$ is $\sigma$-conjugate to $b$, that is if $b' = g^{-1} b \sigma(g)$ with $g \in G(\qpbr)$, then $\xmub \simeq X(\mu,b')_K$ via the map
\begin{equation}
  \begin{tikzcd}
  \smE_1 \arrow[r, dashrightarrow, "\beta_1"] \arrow[d, dashrightarrow, "\sigma^{\ast}\beta_0"] & \mE_1 \arrow[d, dashrightarrow, "\beta_0"] \\
  \smE^0 \arrow[r, dashrightarrow, "b"] & \mE^0.
  \end{tikzcd} \mapsto 
  \begin{tikzcd}
  \smE_1 \arrow[r, dashrightarrow, "\beta_1"] \arrow[d, dashrightarrow, "\sigma(g)\sigma^{\ast}\beta_0"] & \mE_1 \arrow[d, dashrightarrow, "g \beta_0"] \\
  \smE^0 \arrow[r, dashrightarrow, "b'"] & \mE^0.
  \end{tikzcd} 
\end{equation}
We note that this map is nothing more than the action of $g \in \lG(k)$ on $\xmub \subset \grg$ via the natural left action of $\lG$ on $\grg$. For $b'=b$ this induces an action of the closed subgroup $F_b \subset \lG$ on $\xmub$, where $F_b$ is defined as the subfunctor of $\lG$ sending a perfect $\mathbb{F}_p$-algebra $R$ to the group
\begin{align}
  F_b(R)=\{g \in \lG(R) \; | \; g^{-1} b \sigma(g) = b \}.
\end{align}
The $k$-points of $F_b$ are in bijection with $J_b(\mathbb{Q}_p)$, where $J_b$ is the algebraic group over $\mathbb{Q}_p$ introduced in Section \ref{Sec:ADLVSet}. Recall the notion of a pro-\'etale cover of a scheme, see \cite[Definition 1.2]{BhattScholzeII}.
\begin{Lem} \label{Lem:LocalUniformisation}
Consider the morphism $\Theta_b:\xmub \to \shtmub$, which sends a diagram as in \eqref{Eq:ADLVMOD} to $(\mE_1, \beta_1)$. This morphism is $F_b$-invariant for the trivial action on the target, and induces an isomorphism of groupoids
\begin{align}
  \shtmub \simeq \left[F_b \backslash \xmub \right],
\end{align}
where the quotient stack is taken in the pro-\'etale topology. Moreover $F_b$ is isomorphic to the locally profinite group scheme $\underline{J_b(\mathbb{Q}_p)}$ associated to the topological group $J_b(\mathbb{Q}_p)$.\footnote{For a topological group $B$ we define $\underline{B}$ as the functor on $\affperfk$ sending $R$ to the group of continuous functions $|\spec R| \to B$, where $|\spec R|$ is the underlying topological space of $\spec R$. When $B$ is profinite this is representable in an affine group scheme, and thus when $B$ is locally profinite it is representable in a group scheme.}
\end{Lem}
\begin{proof}
The morphism $\Theta_b$ is $F_b$-invariant, since the action of $F_b$ on $\xmub$ doesn't change $(\mE_1, \beta_1)$. For a scheme $T \mapsto \shtmub$, the set $\xmub(T)$ is the set of quasi-isogenies from $(\mE_1, \beta_1)$ to $(\mE^0_T, b_T)$, which is either empty or has a simply transitive action of the group $F_b(T)$ of self quasi-isogenies of $(\mE^0_T, b_T)$. In other words, we have shown that $\Theta_b$ is a pseudo-torsor for $F_b$. By \cite[Theorem I.2.1]{FarguesScholze}, for any $\mG_{K}$-shtuka $(\mE_1, \beta_1) \in \shtb(T)$, the pseudo-torsor of quasi-isogenies to $(\mE^0_T, b_T)$ has a section pro-\'etale locally on $T$. Thus we find that map $\Theta_b$ is a pro-\'etale torsor for $F_b$. In other words, there is an isomorphism
\begin{align} \label{Eq:StacksFPQCTrivial}
  \shtmub \simeq \left[F_b \backslash \xmub \right].
\end{align}
It also follows from \cite[Theorem I.2.1]{FarguesScholze} that $F_b$ is isomorphic to the locally profinite group scheme $\underline{J_b(\mathbb{Q}_p)}$ associated to $J_b(\mathbb{Q}_p)$.
\end{proof}
\begin{Cor} \label{Cor:NonemptyNewtonStrata}
The stack $\shtmub$ is non-empty if and only if $[b] \in \bgmu$.
\end{Cor}
\begin{proof}
This is a direct consequence of Lemma \ref{Lem:LocalUniformisation} in combination with the analogous result for $\xmub(\ovfp)$, which is \cite[Theorem 1.1]{He2}.
\end{proof}
%

	\newcommand{\Z}{\ensuremath{\mathbb{Z}}}
	\newcommand{\uZ}{\ensuremath{\underline{\mathbb{Z}}}}
	\newcommand{\D}{\ensuremath{\mathbb{D}}}
	\newcommand{\uQ}{\ensuremath{\underline{\mathbb{Q}}}}
	\newcommand{\A}{\ensuremath{\mathbb{A}}}
	\newcommand{\real}{\ensuremath{\mathbb{R}}}
	\newcommand{\rhat}{\ensuremath{\hat{\mathscr{R}}}}
	\newcommand{\F}{\ensuremath{\mathbb{F}}}
	\newcommand{\Fpbar}{\ensuremath{\overline{\mathbb{F}}_p}}
	\newcommand{\aff}{\ensuremath{X(\mu,b)}}
	\newcommand{\der}{\ensuremath{\mathrm{der}}}
	\newcommand{\ab}{\ensuremath{\mathrm{ab}}}

	\newcommand{\Sh}{\ensuremath{\mathrm{Sh}}}
	\newcommand{\Q}{\ensuremath{\mathbb{Q}}}
	\newcommand{\R}{\ensuremath{\mathbb{R}}}
	\newcommand{\xbar}{\ensuremath{{\overline{x}}}}
	\newcommand{\Khat}{\ensuremath{\widehat{K}}}
	\newcommand{\Ok}{\ensuremath{\mathcal{O}}}
	\renewcommand{\O}{\ensuremath{\mathcal{O}}}
	
	\newcommand{\G}{\ensuremath{\mathcal{G}}}
	\newcommand{\GG}{\ensuremath{\mathbb{G}}}
	\newcommand{\uG}{\ensuremath{\underline{G}}}
	\newcommand{\ucalG}{\ensuremath{\underline{\mathcal{G}}}}
	\newcommand{\pdiv}{\ensuremath{\mathscr{G}}}
	\newcommand{\Spec}{\ensuremath{{\mathrm{Spec}}}}
	\newcommand{\Spf}{\ensuremath{\mathrm{Spf}}}
	\newcommand{\Res}{\ensuremath{\mathrm{Res}}}
	\newcommand{\GL}{\ensuremath{\mathrm{GL}}}
	\newcommand{\bGL}{\ensuremath{\mathbf{GL}}}
	\newcommand{\GSp}{\ensuremath{\mathrm{GSp}}}
	\newcommand{\bGSp}{\ensuremath{\mathbf{GSp}}}
	\newcommand{\SL}{\ensuremath{\mathrm{SL}}}
	\newcommand{\SU}{\ensuremath{\mathrm{SU}}}
	\newcommand{\cGL}{\ensuremath{\mathcal{GL}}}
	\newcommand{\ucGL}{\ensuremath{\underline{\mathcal{GL}}}}
	\newcommand{\Gr}{\ensuremath{\mathrm{Gr}}}
	\newcommand{\Fl}{\ensuremath{\mathrm{Fl}}}
	
	\newcommand{\set}{\ensuremath{s_{\alpha,\mathrm{\acute{e}t}}}}
	\newcommand{\s}{\ensuremath{\tilde{s}}}
	\newcommand{\kbar}{\ensuremath{\overline{k}}}
	\newcommand{\kot}{\ensuremath{\tilde{\kappa}_G}}
	\newcommand{\p}{\ensuremath{\mathfrak{p}}}
	
	\newcommand{\loc}{\ensuremath{M^{\mathrm{loc}}_{\G}}}
	\newcommand{\lochat}{\ensuremath{\hat{M}^{\mathrm{loc}}_G}}
	\newcommand{\E}{\ensuremath{\mathbf{E}}}
	
	\newcommand{\hW}{\ensuremath{\widehat{W}}}
	\newcommand{\Adm}{\ensuremath{\mathrm{Adm}}}
	
	\newcommand{\fka}{\ensuremath{\mathfrak{a}}}
	\newcommand{\fkb}{\ensuremath{\mathfrak{b}}}
	\newcommand{\fkc}{\ensuremath{\mathfrak{c}}}
	\newcommand{\fkd}{\ensuremath{\mathfrak{d}}}
	\newcommand{\fke}{\ensuremath{\mathfrak{e}}}
	\newcommand{\fkf}{\ensuremath{\mathfrak{f}}}
	\newcommand{\fkg}{\ensuremath{\mathfrak{g}}}
	\newcommand{\fkh}{\ensuremath{\mathfrak{h}}}
	\newcommand{\fki}{\ensuremath{\mathfrak{i}}}
	\newcommand{\fkj}{\ensuremath{\mathfrak{j}}}
	\newcommand{\fkk}{\ensuremath{\mathfrak{k}}}
	\newcommand{\fkl}{\ensuremath{\mathfrak{l}}}
	\newcommand{\fkm}{\ensuremath{\mathfrak{m}}}
	\newcommand{\fkn}{\ensuremath{\mathfrak{n}}}
	\newcommand{\fko}{\ensuremath{\mathfrak{o}}}
	\newcommand{\fkp}{\ensuremath{\mathfrak{p}}}
	\newcommand{\fkq}{\ensuremath{\mathfrak{q}}}
	\newcommand{\fkr}{\ensuremath{\mathfrak{r}}}
	\newcommand{\fks}{\ensuremath{\mathfrak{s}}}
	\newcommand{\fkt}{\ensuremath{\mathfrak{t}}}
	\newcommand{\fku}{\ensuremath{\mathfrak{u}}}
	\newcommand{\fkv}{\ensuremath{\mathfrak{v}}}
	\newcommand{\fkw}{\ensuremath{\mathfrak{w}}}
	\newcommand{\fkx}{\ensuremath{\mathfrak{x}}}
	\newcommand{\fky}{\ensuremath{\mathfrak{y}}}
	\newcommand{\fkz}{\ensuremath{\mathfrak{z}}}
	
	\newcommand{\fkA}{\ensuremath{\mathfrak{A}}}
	\newcommand{\fkB}{\ensuremath{\mathfrak{B}}}
	\newcommand{\fkC}{\ensuremath{\mathfrak{C}}}
	\newcommand{\fkD}{\ensuremath{\mathfrak{D}}}
	\newcommand{\fkE}{\ensuremath{\mathfrak{E}}}
	\newcommand{\fkF}{\ensuremath{\mathfrak{F}}}
	\newcommand{\fkG}{\ensuremath{\mathfrak{G}}}
	\newcommand{\fkH}{\ensuremath{\mathfrak{H}}}
	\newcommand{\fkI}{\ensuremath{\mathfrak{I}}}
	\newcommand{\fkJ}{\ensuremath{\mathfrak{J}}}
	\newcommand{\fkK}{\ensuremath{\mathfrak{K}}}
	\newcommand{\fkL}{\ensuremath{\mathfrak{L}}}
	\newcommand{\fkM}{\ensuremath{\mathfrak{M}}}
	\newcommand{\fkN}{\ensuremath{\mathfrak{N}}}
	\newcommand{\fkO}{\ensuremath{\mathfrak{O}}}
	\newcommand{\fkP}{\ensuremath{\mathfrak{P}}}
	\newcommand{\fkQ}{\ensuremath{\mathfrak{Q}}}
	\newcommand{\fkR}{\ensuremath{\mathfrak{R}}}
	\newcommand{\fkS}{\ensuremath{\mathfrak{S}}}
	\newcommand{\fkT}{\ensuremath{\mathfrak{T}}}
	\newcommand{\fkU}{\ensuremath{\mathfrak{U}}}
	\newcommand{\fkV}{\ensuremath{\mathfrak{V}}}
	\newcommand{\fkW}{\ensuremath{\mathfrak{W}}}
	\newcommand{\fkX}{\ensuremath{\mathfrak{X}}}
	\newcommand{\fkY}{\ensuremath{\mathfrak{Y}}}
	\newcommand{\fkZ}{\ensuremath{\mathfrak{Z}}}
	
	\newcommand{\bbA}{\ensuremath{\mathbb{A}}}
	\newcommand{\bbB}{\ensuremath{\mathbb{B}}}
	\newcommand{\bbC}{\ensuremath{\mathbb{C}}}
	\newcommand{\bbD}{\ensuremath{\mathbb{D}}}
	\newcommand{\bbE}{\ensuremath{\mathbb{E}}}
	\newcommand{\bbF}{\ensuremath{\mathbb{F}}}
	\newcommand{\bbG}{\ensuremath{\mathbb{G}}}
	\newcommand{\bbH}{\ensuremath{\mathbb{H}}}
	\newcommand{\bbI}{\ensuremath{\mathbb{I}}}
	\newcommand{\bbJ}{\ensuremath{\mathbb{J}}}
	\newcommand{\bbK}{\ensuremath{\mathbb{K}}}
	\newcommand{\bbL}{\ensuremath{\mathbb{L}}}
	\newcommand{\bbM}{\ensuremath{\mathbb{M}}}
	\newcommand{\bbN}{\ensuremath{\mathbb{N}}}
	\newcommand{\bbO}{\ensuremath{\mathbb{O}}}
	\newcommand{\bbP}{\ensuremath{\mathbb{P}}}
	\newcommand{\bbQ}{\ensuremath{\mathbb{Q}}}
	\newcommand{\bbR}{\ensuremath{\mathbb{R}}}
	\newcommand{\bbS}{\ensuremath{\mathbb{S}}}
	\newcommand{\bbT}{\ensuremath{\mathbb{T}}}
	\newcommand{\bbU}{\ensuremath{\mathbb{U}}}
	\newcommand{\bbV}{\ensuremath{\mathbb{V}}}
	\newcommand{\bbW}{\ensuremath{\mathbb{W}}}
	\newcommand{\bbX}{\ensuremath{\mathbb{X}}}
	\newcommand{\bbY}{\ensuremath{\mathbb{Y}}}
	\newcommand{\bbZ}{\ensuremath{\mathbb{Z}}}
	
	\newcommand{\bfA}{\ensuremath{\mathbf{A}}}
	\newcommand{\bfB}{\ensuremath{\mathbf{B}}}
	\newcommand{\bfC}{\ensuremath{\mathbf{C}}}
	\newcommand{\bfD}{\ensuremath{\mathbf{D}}}
	\newcommand{\bfE}{\ensuremath{\mathbf{E}}}
	\newcommand{\bfF}{\ensuremath{\mathbf{F}}}
	\newcommand{\bfG}{\ensuremath{\mathbf{G}}}
	\newcommand{\bfH}{\ensuremath{\mathbf{H}}}
	\newcommand{\bfI}{\ensuremath{\mathbf{I}}}
	\newcommand{\bfJ}{\ensuremath{\mathbf{J}}}
	\newcommand{\bfK}{\ensuremath{\mathbf{K}}}
	\newcommand{\bfL}{\ensuremath{\mathbf{L}}}
	\newcommand{\bfM}{\ensuremath{\mathbf{M}}}
	\newcommand{\bfN}{\ensuremath{\mathbf{N}}}
	\newcommand{\bfO}{\ensuremath{\mathbf{O}}}
	\newcommand{\bfP}{\ensuremath{\mathbf{P}}}
	\newcommand{\bfQ}{\ensuremath{\mathbf{Q}}}
	\newcommand{\bfR}{\ensuremath{\mathbf{R}}}
	\newcommand{\bfS}{\ensuremath{\mathbf{S}}}
	\newcommand{\bfT}{\ensuremath{\mathbf{T}}}
	\newcommand{\bfU}{\ensuremath{\mathbf{U}}}
	\newcommand{\bfV}{\ensuremath{\mathbf{V}}}
	\newcommand{\bfW}{\ensuremath{\mathbf{W}}}
	\newcommand{\bfX}{\ensuremath{\mathbf{X}}}
	\newcommand{\bfY}{\ensuremath{\mathbf{Y}}}
	\newcommand{\bfZ}{\ensuremath{\mathbf{Z}}}
	
	\newcommand{\rmA}{\ensuremath{\mathrm{A}}}
	\newcommand{\rmB}{\ensuremath{\mathrm{B}}}
	\newcommand{\rmC}{\ensuremath{\mathrm{C}}}
	\newcommand{\rmD}{\ensuremath{\mathrm{D}}}
	\newcommand{\rmE}{\ensuremath{\mathrm{E}}}
	\newcommand{\rmF}{\ensuremath{\mathrm{F}}}
	\newcommand{\rmG}{\ensuremath{\mathrm{G}}}
	\newcommand{\rmH}{\ensuremath{\mathrm{H}}}
	\newcommand{\rmI}{\ensuremath{\mathrm{I}}}
	\newcommand{\rmJ}{\ensuremath{\mathrm{J}}}
	\newcommand{\rmK}{\ensuremath{\mathrm{K}}}
	\newcommand{\rmL}{\ensuremath{\mathrm{L}}}
	\newcommand{\rmM}{\ensuremath{\mathrm{M}}}
	\newcommand{\rmN}{\ensuremath{\mathrm{N}}}
	\newcommand{\rmO}{\ensuremath{\mathrm{O}}}
	\newcommand{\rmP}{\ensuremath{\mathrm{P}}}
	\newcommand{\rmQ}{\ensuremath{\mathrm{Q}}}
	\newcommand{\rmR}{\ensuremath{\mathrm{R}}}
	\newcommand{\rmS}{\ensuremath{\mathrm{S}}}
	\newcommand{\rmT}{\ensuremath{\mathrm{T}}}
	\newcommand{\rmU}{\ensuremath{\mathrm{U}}}
	\newcommand{\rmV}{\ensuremath{\mathrm{V}}}
	\newcommand{\rmW}{\ensuremath{\mathrm{W}}}
	\newcommand{\rmX}{\ensuremath{\mathrm{X}}}
	\newcommand{\rmY}{\ensuremath{\mathrm{Y}}}
	\newcommand{\rmZ}{\ensuremath{\mathrm{Z}}}
	
	\newcommand{\scrA}{\ensuremath{\mathscr{A}}}
	\newcommand{\scrB}{\ensuremath{\mathscr{B}}}
	\newcommand{\scrC}{\ensuremath{\mathscr{C}}}
	\newcommand{\scrD}{\ensuremath{\mathscr{D}}}
	\newcommand{\scrE}{\ensuremath{\mathscr{E}}}
	\newcommand{\scrF}{\ensuremath{\mathscr{F}}}
	\newcommand{\scrG}{\ensuremath{\mathscr{G}}}
	\newcommand{\scrH}{\ensuremath{\mathscr{H}}}
	\newcommand{\scrI}{\ensuremath{\mathscr{I}}}
	\newcommand{\scrJ}{\ensuremath{\mathscr{L}}}
	\newcommand{\scrK}{\ensuremath{\mathscr{K}}}
	\newcommand{\scrL}{\ensuremath{\mathscr{L}}}
	\newcommand{\scrM}{\ensuremath{\mathscr{M}}}
	\newcommand{\scrN}{\ensuremath{\mathscr{N}}}
	\newcommand{\scrO}{\ensuremath{\mathscr{O}}}
	\newcommand{\scrP}{\ensuremath{\mathscr{P}}}
	\newcommand{\scrQ}{\ensuremath{\mathscr{Q}}}
	\newcommand{\scrR}{\ensuremath{\mathscr{R}}}

	\newcommand{\scrT}{\ensuremath{\mathscr{T}}}
	\newcommand{\scrU}{\ensuremath{\mathscr{U}}}
	\newcommand{\scrV}{\ensuremath{\mathscr{V}}}
	\newcommand{\scrW}{\ensuremath{\mathscr{W}}}
	\newcommand{\scrX}{\ensuremath{\mathscr{X}}}
	\newcommand{\scrY}{\ensuremath{\mathscr{Y}}}
	\newcommand{\scrZ}{\ensuremath{\mathscr{Z}}}
	
	\newcommand{\calA}{\ensuremath{\mathcal{A}}}
	\newcommand{\calB}{\ensuremath{\mathcal{B}}}
	\newcommand{\calC}{\ensuremath{\mathcal{C}}}
	\newcommand{\calD}{\ensuremath{\mathcal{D}}}
	\newcommand{\calE}{\ensuremath{\mathcal{E}}}
	\newcommand{\calF}{\ensuremath{\mathcal{F}}}
	\newcommand{\calG}{\ensuremath{\mathcal{G}}}
	\newcommand{\calH}{\ensuremath{\mathcal{H}}}
	\newcommand{\calI}{\ensuremath{\mathcal{I}}}
	\newcommand{\calJ}{\ensuremath{\mathcal{J}}}
	\newcommand{\calK}{\ensuremath{\mathcal{K}}}
	\newcommand{\calL}{\ensuremath{\mathcal{L}}}
	\newcommand{\calM}{\ensuremath{\mathcal{M}}}
	\newcommand{\calN}{\ensuremath{\mathcal{N}}}
	\newcommand{\calO}{\ensuremath{\mathcal{O}}}
	\newcommand{\calP}{\ensuremath{\mathcal{P}}}
	\newcommand{\calQ}{\ensuremath{\mathcal{Q}}}
	\newcommand{\calR}{\ensuremath{\mathcal{R}}}
	\newcommand{\calS}{\ensuremath{\mathcal{S}}}
	\newcommand{\calT}{\ensuremath{\mathcal{T}}}
	\newcommand{\calU}{\ensuremath{\mathcal{U}}}
	\newcommand{\calV}{\ensuremath{\mathcal{V}}}
	\newcommand{\calW}{\ensuremath{\mathcal{W}}}
	\newcommand{\calX}{\ensuremath{\mathcal{X}}}
	\newcommand{\calY}{\ensuremath{\mathcal{Y}}}
	\newcommand{\calZ}{\ensuremath{\mathcal{Z}}}
	
	\newcommand{\brA}{\ensuremath{\breve{A}}}
	\newcommand{\brB}{\ensuremath{\breve{B}}}
	\newcommand{\brC}{\ensuremath{\breve{C}}}
	\newcommand{\brD}{\ensuremath{\breve{D}}}
	\newcommand{\brE}{\ensuremath{\breve{E}}}
	\newcommand{\brF}{\ensuremath{\breve{F}}}
	\newcommand{\brG}{\ensuremath{\breve{G}}}
	\newcommand{\brH}{\ensuremath{\breve{H}}}
	\newcommand{\brI}{\ensuremath{\breve{I}}}
	\newcommand{\brJ}{\ensuremath{\breve{J}}}
	\newcommand{\brK}{\ensuremath{\breve{K}}}
	\newcommand{\brL}{\ensuremath{\breve{L}}}
	\newcommand{\brM}{\ensuremath{\breve{M}}}
	\newcommand{\brN}{\ensuremath{\breve{N}}}
	\newcommand{\brO}{\ensuremath{\breve{O}}}
	\newcommand{\brP}{\ensuremath{\breve{P}}}
	\newcommand{\brQ}{\ensuremath{\breve{\mathbb{Q}}_p}}
	\newcommand{\brR}{\ensuremath{\breve{R}}}
	\newcommand{\brS}{\ensuremath{\breve{S}}}
	\newcommand{\brT}{\ensuremath{\breve{T}}}
	\newcommand{\brU}{\ensuremath{\breve{U}}}
	\newcommand{\brV}{\ensuremath{\breve{V}}}
	\newcommand{\brW}{\ensuremath{\breve{W}}}
	\newcommand{\brX}{\ensuremath{\breve{X}}}
	\newcommand{\brY}{\ensuremath{\breve{Y}}}
	\newcommand{\brZ}{\ensuremath{\breve{\mathbb{Z}}_p}}

	\newcommand{\OKE}{\ensuremath{\mathcal{O}_{\mathbf{E}_\lambda}}}
	\newcommand{\kl}{\ensuremath{k_\lambda}}
	\newcommand{\m}{\ensuremath{\mathfrak{m}}}
	\newcommand{\n}{\ensuremath{\mathfrak{n}}}
	\newcommand{\et}{\ensuremath{\mathrm{\acute{e}t}}}
	\newcommand{\alHom}{\ensuremath{\mbox{alHom}}}

	\newcommand{\Mod}{\ensuremath{\mathrm{Mod}_{\mathfrak{S}}^\varphi}}

	\newcommand{\Fr}{\ensuremath{\mathrm{Fr}}}
	\newcommand{\Cl}{\ensuremath{\mathrm{Cl}}}
	\newcommand{\sa}{\ensuremath{s_{\alpha,0}}}
	
	\newcommand{\po}{\ar@{}[dr]|{\text{\pigpenfont R}}}
	\newcommand{\pb}{\ar@{}[dr]|{\text{\pigpenfont J}}}

\section{Uniformisation of isogeny classes} \label{Sec:Uniformisation}
In this section we will recall the construction of the Kisin--Pappas integral models of Shimura varieties of Hodge type with parahoric level structure, and recall the construction of Hamacher--Kim of shtukas on the perfections of their special fibers. We also discuss the change-of-parahoric maps constructed by Zhou in \cite{Z}, and show that the shtukas of Hamacher--Kim are compatible with these maps using results of \cite{PappasRapoportShtukas}.

We then recall the results from Appendix \ref{Sec:VerySpecialADLV} about the existence of CM lifts for Shimura varieties with very special parahoric level, and use that to deduce the existence of CM lifts for arbitrary parahorics. Next, we study how uniformisation `lifts' along the change-of-parahoric maps. Concretely, we will show that uniformisation of isogeny classes at Iwahori level follows from uniformisation at very special level if a certain diagram of stacks on $\affperfk$ is Cartesian.

\subsection{Integral models of Shimura varieties} \label{Sec:IntegralModels}
We recall the construction of the integral models of Shimura varieties of Hodge type in \cite{KisinPappas}. Let $(G,X)$ be a Shimura datum with reflex field $E$ and let $\{\mu_h\}$ be the $G(\mathbb{C})$-conjugacy class of cocharacters of $G_{\mathbb{C}}$ defined in \cite[Section 6]{Z}. Let $\A_f$ denote the ring of finite adeles and $\afp$ the subring of $\A_f$ with trivial $p$-component. Let $U_p\subset G(\Q_p)$ and $U^p\subset G(\A_f^p)$ be compact open subgroups, write $U=U^pU_p$. Then for $U^p$ sufficiently small $$G(\Q)\backslash X\times G(\A_f)/U$$
has the structure of an algebraic variety over $\mathbb{C}$, which has a canonical model $\mathbf{Sh}_{U}(G,X)$ over the reflex field $E$ of $(G,X)$. We will also consider the projective limits (which exists by \cite[Tag 01YX]{stacks-project} since the transition maps are finite \'etale and the schemes are qcqs)
\begin{align}
  \mathbf{Sh}_{U_p}(G,X)&:=\varprojlim_{U^p}\mathbf{Sh}_{U^pU_p}(G,X) \\
  \mathbf{Sh}(G,X)&:=\varprojlim_{U}\mathbf{Sh}_{U}(G,X).
\end{align}
\subsubsection{} Let $V$ be a vector space over $\Q$ of dimension $2g$ equipped with a perfect alternating bilinear form $\psi$. For a $\mathbb{Q}$-algebra $R$, we write $V_R=V\otimes_{\Q}R$. Let $G_V$ denote the corresponding group of symplectic similitudes and let $\mathcal{H}_V$ denote the set of homomorphisms $h:\mathbb{S}\rightarrow G_{V,\R}$ corresponding to the Siegel upper and lower half space, where $\mathbb{S}:=\operatorname{Res}_{\mathbb{C}/\mathbb{R}} \mathbb{G}_m$ is the Deligne torus.

For the rest of this section, we fix an embedding of Shimura data $\iota:(G,X)\rightarrow (G_V, \mathcal{H}_V)$. We sometimes write $G$ for $G_{\Q_p}$ when there is no risk of confusion. We will also assume for the rest of this section that the following conditions hold
\begin{equation}\text{$G$ splits over a tamely ramified extension of $\Q_p$ and $p\nmid|\pi_1(\gder)|$}.
\end{equation}
Let $\G$ be a connected parahoric subgroup of $G$, that is, $\G=\G_x=\G_x^{\circ}$ for some $x\in B(G,\Q_p)$, see Section \ref{Sec:Parahorics}. We will follow the notation of Section \ref{Sec:Local} to write $\G=\mG_{K}$ for some $\sigma$-stable type $K \subset \mathbb{S}$. By \cite[2.3.15]{KisinPappas}, after replacing $\iota$ by another symplectic embedding, there is a closed immersion $\mG_{K} \rightarrow \mathcal{P}$, where $\mathcal{P}$ is a parahoric group scheme of $G_V$ corresponding to the stabiliser of a lattice $V_{\Z_p}\subset V$. Upon scaling $V_{\Z_p}$, we may assume $V_{\Z_p}^\vee\subset V_{\Z_p}$. This induces a closed immersion (see \cite[Proposition 2.3.7]{KisinPappas}) of local models $$M^{\mathrm{loc}}_{\mG_{K},X}\rightarrow M^{\mathrm{loc}}_{\mathcal{P},\mathcal{H}_V}\otimes\mathcal{O}_{E,v}$$
for every place $v$ of $E$ above $p$. Here the local models are as introduced in \cite[Section 2.1]{KisinPappas}.

\subsubsection{} Let $U_V^p \subset G_V(\afp)$ be a sufficiently small compact open subgroup. Let $V_{\Z_{(p)}}=V_{\Z_p}\cap V$ and write $G_{\Z_{(p)}}$ for the Zariski closure of $G$ in $GL(V_{\Z_{(p)}})$, then $G_{\Z_{(p)}}\otimes_{\Z_{(p)}}\Z_p\cong \mG_{K}$. The choice of $V_{\Z_{(p)}}$ gives rise to a compact open subgroup $U_{V,p} \subset G_V(\qp)$ which gives the Shimura variety $\mathbf{Sh}_{U_V}(G_V,\mathcal{H}_V)$ of level $U_V=U_V^p U_{V,p}$ an interpretation as a moduli space of (weakly polarised) abelian varieties up to prime-to-$p$ isogeny, and hence an integral model $\mathscr{S}_{U_V}(G_V,\mathcal{H}_V)$ over $\Z_{(p)}$, which is described in \cite[Section 6.3]{Z}. 

\subsubsection{} For the rest of this paper, we fix an algebraic closure $\overline{\Q}$ of $E$, and for each place $v$ of $\Q$ an algebraic closure $\overline{\Q}_v$ together with an embedding $\overline{\Q}\rightarrow \overline{\Q}_v$. Using these embeddings, we get a $G(\qpbar)$-conjugacy class of cocharacter $\{\mu_h\}$ induced from the Hodge cocharacter associated to $X$. 

By \cite[Lemma 2.1.2]{KisinModels}, we can choose $U_V^p$ such that $\iota$ induces a closed immersion
$$\mathbf{Sh}_{U}(G,X)\hookrightarrow \mathbf{Sh}_{U_V}(G_V,\mathcal{H}_V)_{E}$$ defined over $E$. The choice of embedding $E\rightarrow \overline{\Q}_p$ determines a place $v$ of $E$. Write $\Ok_{E,{(v)}}$ for the localisation of $\Ok_{E}$ at $v$, let $E_v$ be the completion of $E$ at $v$ and $\mathcal{O}_{E,v}$ the ring of integers of $E_v$. We assume the residue field has $q=p^r$ elements and as before $k$ will denote an algebraic closure of $\mathbb{F}_q$. We define $\mathscr{S}_U(G,X)^-$ to be the Zariski closure of $\mathbf{Sh}_{U}(G,X)$ inside $\mathscr{S}_{U_V}(G_V,\mathcal{H}_V)\otimes_{\Z_{(p)}}\mathcal{O}_{E,(v)}$, and $\mathscr{S}_{U}(G,X)$ to be its normalisation. By construction, for $U^p_1\subset U^p_2$ compact open subgroups of $G(\A_f^p)$, there are finite \'etale transition maps $\mathscr{S}_{U^p_1 U_p}(G,X)\rightarrow \mathscr{S}_{U^p_2 U_p}(G,X)$ and we write $\mathscr{S}_{U_p}(G,X):=\varprojlim_{U^p}\mathscr{S}_{U^pU_p}(G,X)$. Under these assumptions we have the following result:
\begin{Thm}[\cite{KisinPappas} Theorem 4.2.2, Theorem 4.2.7] \label{Thm:KisinPappasLocalModel} The $\mathcal{O}_{E,(v)}$ scheme $\mathscr{S}_{U_p}(G,X)$ is a flat $G(\A_f^p)$-equivariant extension of $\mathbf{Sh}_{U_p}(G,X)$. Moreover $\mathscr{S}_{U}(G,X)_{\mathcal{O}_{E,v}}$ fits in a local model diagram
\[\xymatrix{ &\widetilde{\mathscr{S}}_{U}(G,X)_{\mathcal{O}_{E,v}}\ar[dr]^\pi\ar[dl]^q&\\
\mathscr{S}_{U}(G,X)_{\mathcal{O}_{E,v}} & &M^{\mathrm{loc}}_{\mG_{K},X}},\]
where $q$ is a $\mG_{K}$-torsor and $\pi$ is smooth of relative dimension $\operatorname{Dim} G$.
\end{Thm}
Note that the main result of \cite{Pappas} tells us that the integral model $\mathscr{S}_{U}(G,X)$ does not depend on the choice of Hodge embedding. 

\subsubsection{} \label{Sec:TensorsIntegral} By \cite[1.3.2]{KisinModels}, the subgroup $G_{\Z_{(p)}}$ is the stabiliser of a collection of tensors $s_\alpha\in V_{\Z_{(p)}}^\otimes$ for $\alpha \in \mathscr{A}$. Let $h:\mathcal{A}\rightarrow \mathscr{S}_{U}(G,X)$ denote the pullback of the universal abelian variety on $\mathscr{S}_{U_V}(G_V,\mathcal{H}_V)$ and let $V_B:=R^1h_{\mathrm{an}*}\Z_{(p)}$, where $h_{\mathrm{an}}$ is the map of complex analytic spaces associated to $h$. We also let $\mathcal{V}=R^1h_*\Omega^\bullet$ be the relative de Rham cohomology of $\mathcal{A}$. Using the de Rham isomorphism, the tensors $s_\alpha$ give rise to a collection of Hodge cycles $s_{\alpha,dR}\in \mathcal{V}_\mathbb{C}^\otimes$, where $\mathcal{V}_\mathbb{C}$ is the complex analytic vector bundle associated to $\mathcal{V}$. By \cite[\S 2.2]{KisinModels}, these tensors are defined over $E$, and in fact over $\mathcal{O}_{E,(v)}$ by \cite[Proposition 4.2.6]{KisinPappas}.

Similarly, for a finite prime $\ell \neq p$, we let $\mathcal{V}_{\ell}=R^1h_{\acute{e}t*}\Q_\ell$ and $\mathcal{V}_p=R^1h_{\eta,\acute{e}t*}\Z_p$ where $h_\eta$ is the generic fibre of $h$. Using the \'etale-Betti comparison isomorphism, we obtain tensors $s_{\alpha,\ell}\in \mathcal{V}^\otimes_\ell$ and $s_{\alpha,p}\in\mathcal{V}_p^\otimes$. For $*=B, dR,\ell$ and $x\in \mathscr{S}_{U^pU_p}(G,X)(T)$ for some $\mathcal{O}_{E,(v)}$-scheme $T$, we write $s_{\alpha,*,x}$ for the pullback of $s_{\alpha,*}$ to $T$ via $x$. Similarly, the image of $x$ under $\scrS_{U^pU_p}(G,X) \to \mathscr{S}_{U_V}(G_V,\mathcal{H}_V)\otimes_{\Z_{(p)}}\mathcal{O}_{E,(v)}$ gives us a weakly polarised abelian variety up to prime-to-$p$ isogeny $(\mathcal{A}_x,\lambda)$. Over $\scrs_{U_p}(G,X)$, there is a canonical isomorphism of pro-\'etale $\afp$-local systems
\begin{align} \label{Eq:BasisAdelicTateModule}
    \eta: \mathcal{V}^p\xrightarrow{\sim} V \otimes \afp
\end{align}
which takes $t_{\alpha, \afp}$ to $t_{\alpha} \otimes 1$ for all $\alpha \in \mathscr{A}$, see \cite[Lemma 5.1.9]{KisinShinZhu}.

\subsubsection{} \label{Sec:HamacherKimShtukas} Recall that $\ovfp$ is an algebraic closure of $\F_q$ and $\qpbr=W(\ovfp)[1/p]$. Let $\overline{x}\in\mathscr{S}_{U}(G,X)(\ovfp)$ and $x\in\mathscr{S}_{U}(G,X)(\Ok_L)$ a point lifting $\overline{x}$, where $L/\qpbr$ is a finite extension. Let us write $\varphi$ for the Frobenius on $\qpbr$ and $\zpbr$.

Let $\pdiv_x$ denote the $p$-divisible group associated to $\mathcal{A}_x$ and $\pdiv_{x,0}$ its special fiber. Then $T_p\pdiv_x^\vee$ is identified with $H^1_{\acute{e}t}(\mathcal{A}_x,\Z_p)$ and we obtain $\Gamma_K$-invariant tensors $s_{\alpha,\acute{e}t,x}\in T_p\pdiv^{\vee\otimes}$ whose stabiliser can be identified with $\mG_{K}$. Let $\mathbb{D}_x:=\D(\pdiv_{x,0})$ be the contravariant Dieudonn\'e module associated to the $p$-divisible group $\pdiv_{x,0}$. We may apply the constructions of \cite[Section 3]{Z} to obtain $\varphi$-invariant tensors $s_{\alpha,0,x} \in \mathbb{D}_x$, whose stabiliser group can be identified with $\mG_{K} \otimes_{\Z_p} \zpbr$. 

This means that we can upgrade the Dieudonné module of $A_x$ to a $\mathcal{G}_K$-shtuka over $\ovfp$, and this gives a map (see \cite[Proof of Axiom 4 in Section 8]{Z})
\begin{align} \label{Eq:ShtukaMapDef}
  \mathscr{S}_{U}(G,X)(\ovfp) \to \shtmu(\ovfp),
\end{align}
where $\{\mu\}=\{\sigma(\mu_h^{-1})\}$. It is a result of Hamacher--Kim (\cite[Proposition 1]{HamacherKim}, see \cite[Poposition 4.4.1]{ShenYuZhang}) that there is a morphism $\shg \to \shtmu$ inducing \eqref{Eq:ShtukaMapDef} on $\ovfp$-points, where $\shg$ is the perfection of the basechange to $k$ of $\mathscr{S}_{U}(G,X)$.\footnote{The subscript $K$ is used to signify the choice of $\sigma$-stable type $K \subset \mathbb{S}$ corresponding to the parahoric subgroup $U_p$.} It follows from \cite[Theorem 1.3.4]{PappasRapoportShtukas}\footnote{Pappas and Rapoport construct a local shtuka bounded by $\{\mu\}$ over the diamond associated to the formal completion of $\mathscr{S}_{U}(G,X)$ and prove uniqueness and functoriality for this object. By \cite[Example 2.4.9]{PappasRapoportShtukas}, this induces a $\mathcal{G}_K$-shtuka over the perfect special fiber of $\mathscr{S}_{U}(G,X)$, which is of type $\{\mu\}$ by the discussion in \cite[Section 2.4.3]{PappasRapoportShtukas}, cf. \cite[Lemma 3.1.7]{DvHKZIgusaStacks}. The $\mathcal{G}_K$-shtuka constructed this way moreover agrees with the one constructed by \cite{ShenYuZhang}, this is spelled out in \cite[Section 5.3]{DvHKZIgusaStacks}. \label{FN:PR}} that this morphism does not depend on the choice of Hodge embedding and moreover can be upgraded to a $G(\afp)$-equivariant morphism
\begin{align}
  \shginf:=\varprojlim_{U^p} \shg \to \shtmu,
\end{align}
where $G(\afp)$ acts trivially on $\shtmu$.

It follows from \cite[the discussion after Theorem 4.4.3]{ShenYuZhang} that the perfection of the special fiber of $M_{\mG_{K},X}^{\mathrm{loc}}$ can be identified with the closed subscheme $\mloc$ of the affine flag variety for $\mG_{K}$ introduced in Section \ref{Sec:AdmissibleSetLocalModels}. Under this isomorphism, the right action of $\lpGk$ on $\mloc$, which factors through $\overline{\mG}_K$, is identified with the $\overline{\mG}_K$ action\footnote{Here we are writing $\overline{\mG}_K$ for the perfection of the special fiber of $\mG_K$, by abuse of notation.} on the perfection of $M_{G,K,\mu}^{\mathrm{loc}}$. Thus the local model diagram of Theorem \ref{Thm:KisinPappasLocalModel} gives us a (perfectly smooth) morphism
\begin{align}
  \lambda_K:\shg \to \lbrack \mloc/\overline{\mG}_K \rbrack.
\end{align}

\subsubsection{} \label{Sec:RestrictedShtukasShimuraVarieties} Fix $n \ge 2$ and choose $m \gg 0$ such that the action $\operatorname{Ad}_{\sigma} \lpGk$ on $\mlocn$ factors through $\lmG$ and such that $m$ satisfies the assumptions of Proposition \ref{Prop:CartesianInfiniteToFinite}. Let $\shtmu^{(m,1)}$ be the stack of $(m)$-restricted shtukas of type $\{\mu\}$ from Section \ref{Sec:RestrictedLocalShtukas} and also consider the stack $\shtmu^{(m,n), \mathrm{loc}}$ from Section \ref{Sec:GeneralRestrictedShtukas}. If we compose the morphism $\shg \to \shtmu$ constructed above with the natural map $\shtmu \to \shtmu^{(m,1)}$ we obtain a morphism 
\begin{align}
  \shg \to \shtmu^{(m,1)}.
\end{align}
By \cite[Theorem 4.4.3]{ShenYuZhang}, the perfectly smooth map $\lambda_K:\shg \to [\mloc/\overline{\mG}_K]$ induced from the local model diagram fits in a commutative diagram
\begin{equation}
\begin{tikzcd}
  \shg \arrow{r} \arrow{dr}{\lambda_K} & \shtmu^{(m,1)}\arrow{d} \\
  & \lbrack \mloc/\overline{\mG}_K \rbrack,
\end{tikzcd}
\end{equation}
where $\shtmu^{(m,1)} \to \lbrack \mloc/\overline{\mG}_K \rbrack$ comes from the diagram in Remark \ref{Rem:ShtukaLocalModelII}. The map $\shg \to \shtmu^{(m,1)}$ is perfectly smooth by \cite[Corollary 2.57]{HoffThesis}. Recall moreover that there is a natural map $\shtmu \to \shtmu^{(m,n), \mathrm{loc}}$ which induces a map $\shg \to \shtmu^{(m,n), \mathrm{loc}}$, this is also  perfectly smooth by \cite[Corollary 2.57]{HoffThesis}.
\begin{Rem}
    The perfect smoothness results discussed above are also claimed in \cite[Theorem 4.4.3]{ShenYuZhang} and \cite[Proposition 7.2.4]{XiaoZhu} (the latter in the hyperspecial case). It has been pointed out to us by Manuel Hoff and Xinwen Zhu that the proof of \cite[Proposition 7.2.4]{XiaoZhu} is not correct as written; the square in \cite[top of page 113]{XiaoZhu} does not commute. The same error seems to be present in the proof of \cite[Theorem 4.4.3]{ShenYuZhang}.
\end{Rem}

We can use the perfectly smooth map $\shg \to \shtmu^{(m,1)}$ to define the EKOR stratification on $\shg$, see Section \ref{Sec:EKOR}. In particular for $w \in \kadmu$ we will write $\shg\{w\}$ for the locally closed EKOR stratum of $\shg$. Since $\shg \to \shtmu^{(m,1)}$ is perfectly smooth and thus open, we find that the closure of $\shg\{w\}$ is given by
\begin{align}
    \shg\{\preceq w \}:=\bigcup_{w' \preceq w} \shg\{w'\}
\end{align}
because the closure relations hold on $\shtmu^{(m,1)}$.

\subsubsection{Isogeny classes} \label{Sec:IsogenyClasses} Let $x\in \shg(\ovfp)$, then attached to $x$ is an abelian variety $\calA_x$ over $\ovfp$. We write $\bbD_x$ for the contravariant Dieudonn\'e module associated to the $p$-divisible group $\scrG_x$ of $\calA_x$; then $\bbD_x$ is equipped with a corresponding set of tensors $s_{\alpha,0,x}$, see Section \ref{Sec:HamacherKimShtukas}. Similarly, for $\ell \neq p$, the rational $\ell$-adic Tate module $V_\ell\calA_x$ is equipped with tensors $s_{\alpha,\ell,x}$.

Two points $x,x'\in \shginf(\ovfp)$ are said to lie in the same isogeny class if there exists a quasi-isogeny $\calA_x\rightarrow\calA_{x'}$ such that the induced maps $\mathbb{D}_x[1/p] \rightarrow \mathbb{D}_{x'}[1/p]$ and $V_{\ell}\calA_x \rightarrow V_{\ell}\calA_{x'}\otimes $ sends $s_{\alpha,0,x'}$ to $s_{\alpha,0,x}$ and $s_{\alpha,\ell,x}$ to $s_{\alpha,\ell,x'}$ for all $\ell\neq p$. We write $\mathscr{I}_x \subset \shginf(\ovfp)$ for the isogeny class of $x$. 

For $x\in \shginf(\ovfp)$ we let $I_x$ denote the reductive $\bbQ$-group associated to $x$ as in \cite[Section 9.2]{Z}; it is a subgroup of the algebraic group of self quasi-isogenies of the abelian variety $A_x$. It comes equipped with a natural map $I_{x,\afp} \to G_{\afp}$ coming from the tautological basis of the prime-to-$p$ adelic Tate-module of $A_x$ given by \eqref{Eq:BasisAdelicTateModule}. If we choose an isomorphism $\alpha:\mathbb{D}_x \simeq V_{\mathbb{Z}_p} \otimes_{\zp} \zpbr$ sending $s_{\alpha, 0,x}$ to $s_{\alpha} \otimes 1$, under which the Frobenius of $\mathbb{D}_x$ corresponds to $b \in G(\qpbr)$, then there is also an induced map $I_{x, \qp} \to J_b$. Note that an isomorphism $\alpha$ as above always exists, by \cite[Section 5.6]{Z}.

\subsubsection{Change of parahoric} \label{Sec:ChangeOfParahoric}
Now let $J \subset K$ be another $\sigma$-stable type, let $\mG_{J}(\zp)=:U_p' \subset U_p$ and let $U'=U^pU_p'$. Note that $\mG_{J}$ is a connected parahoric since $\mG_{K}$ is, see Lemma \ref{Lem:IwahoriConnected}. We will use $\shgj$ to denote the perfection of the special fiber of $\mathscr{S}_{U'}(G,X)$. By \cite[Theorem 7.1]{Z}, there is a (necessarily unique) proper morphism $\pi_{J,K}:\mathscr{S}_{U'}(G,X) \to \mathscr{S}_{U}(G,X)$ which induces the obvious forgetful morphism on generic fibers and induces a $G(\afp)$-equivariant map
\begin{align}
  \varprojlim_{U^p} \mathscr{S}_{U'}(G,X) \to \varprojlim_{U^p} \mathscr{S}_{U}(G,X).
\end{align}
We now recall some aspects of the construction of the forgetful map from \cite[Section 7.2]{Z}, which we will need to compare isogeny classes in the source and target. There are facets $\mathfrak{f}, \mathfrak{f}'$ of the extended Bruhat--Tits building $B(G,\mathbb{Q}_p)$ of $G_{\qp}$ such that $U_p$ is the stabiliser of $\mathfrak{f}$ and such that $U_p'$ is the stabiliser of $\mathfrak{f}'$. Fix a choice of embedding $\theta:B(G,\mathbb{Q}_p) \to B(G_{V},\mathbb{Q}_p)$ as in \cite[Section 1.2]{KisinPappas}, compatible with  $G \to G_V$. Choose facets $\mathfrak{g}, \mathfrak{g}'$ containing $\theta(\mathfrak{f})$ and $\theta(\mathfrak{f}')$ respectively, and we let $M_p \subset G_V(\qp)$ be the stabiliser of $\mathfrak{g}$ and $M_p' \subset G_V(\qp)$ be the stabiliser of $\mathfrak{g}'$. As in \cite[Section 8.1]{Z}, the facets $\mathfrak{g}, \mathfrak{g}'$ correspond to lattice chains $\mathcal{L}$ and $\mathcal{L}'$ in $V_{\qp}$ respectively, with $\mathcal{L}'$ a refinement\footnote{This means that $\mathcal{L}$ and $\mathcal{L}'$ are chains of lattices in $V_{\qp}$ and that every lattice in $\mathcal{L}$ is also contained in $\mathcal{L}'$.} of $\mathcal{L}$; note that \cite{Z} writes $\mathcal{L}'$ for what we call $\mathcal{L}$ and vice versa. 

Then for sufficiently small $M^p \subset G_V(\afp)$ there are moduli-theoretic integral models $\mathscr{S}_{M^pM_p'}(G_V, \mathcal{H}_V)$ and $\mathscr{S}_{M^pM_p}(G_V, \mathcal{H}_V)$ over $\mathbb{Z}_{(p)}$. The former is a moduli space of $\mathcal{L}'$-chains of (weakly polarised) abelian varieties up to prime-to-$p$ isogeny with $M^p$ level structure, as explained in \cite[Proof of Axiom 1 in Section 8]{Z}, and the latter is a moduli space of $\mathcal{L}$-chains of (weakly polarised) abelian varieties up to prime-to-$p$ isogeny with $M^p$ level structure. There is a natural proper forgetful map 
\begin{align}
    \pi_{\mathcal{L}',\mathcal{L}}:\mathscr{S}_{M^pM_p'}(G_V, \mathcal{H}_V) \to \mathscr{S}_{M^pM_p}(G_V, \mathcal{H}_V)
\end{align}
which sends an $\mathcal{L}'$-chain of abelian varieties to the underlying $\mathcal{L}$-chain of abelian varieties. \smallskip 

Taking the direct sum of the lattices in the lattice chain $\mathcal{L}$ (resp. $\mathcal{L}'$) we get a symplectic space $V_{\mathcal{L}}$ (resp. $V_{\mathcal{L}'}$) and a lattice $V_{\mathcal{L},\zp}$ (resp. $V_{\mathcal{L}',\zp}$) in $V_{\mathcal{L},\qp}$ (resp. $V_{\mathcal{L}',\qp}$). Let us denote the stabiliser of $V_{\mathcal{L},\zp}$ (resp. $V_{\mathcal{L}',\zp}$) in $G_{V_{\mathcal{L}}}(\qp)$ (resp. $G_{V_{\mathcal{L}'}}(\qp)$) by  $J_p$ (resp. $J_p'$).

Then there are Hodge embeddings $(G_V,\mathcal{H}_V) \to (G_{V_{\mathcal{L}}},\mathcal{H}_{V_{\mathcal{L}}})$ and $(G_V, \mathcal{H}_V) \to (G_{V_{\mathcal{L}'}} \mathcal{H}_{V_{\mathcal{L}'}})$, which take $M_p$ to $J_p$ and $M_p'$ to $J_p'$ respectively. These induce finite maps
\begin{align}
  \mathscr{S}_{M^pM_p'}(G_V, \mathcal{H}_V) \to \mathscr{S}_{J^{'p}J_p'}(G_{V_{\mathcal{L}'}}, \mathcal{H}_{V_{\mathcal{L}'}}), \qquad \mathscr{S}_{M^pM_p}(G_V, \mathcal{H}_V) \to \mathscr{S}_{J^pJ_p}(G_{V_{\mathcal{L}}},\mathcal{H}_{V_{\mathcal{L}}})
\end{align}
for some $J^p \subset G_{V_{\mathcal{L}}}(\afp)$ and $J^{'p} \subset G_{V_{\mathcal{L}'}}(\afp)$ sufficiently small, which take an $\mathcal{L}'$-set (resp. $\mathcal{L}$-set) of abelian varieties to the product of all the abelian varieties in the $\mathcal{L}'$-set (resp. the $\mathcal{L}$-set), equipped with the product polarisation and level structure. It is explained in \cite[Equation 8.1 of Section 8]{Z} that our forgetful maps fit in a commutative diagram where all the horizontal maps are finite
\begin{equation} \label{Eq:ParahoricForgetfulGandSiegel}
  \begin{tikzcd}
   \mathscr{S}_{U'}(G,X) \arrow{d}{\pi_{J,K}} \arrow{r} & \mathscr{S}_{M^pM_p'}(G_V, \mathcal{H}_V)\otimes \mathcal{O}_{E,(v)} \arrow{d}{\pi_{\mathcal{L}', \mathcal{L}}} \arrow{r} &  \mathscr{S}_{J^{'p} J_p'}(G_{V_{\mathcal{L}'}} \mathcal{H}_{V_{\mathcal{L}'}}) \otimes \mathcal{O}_{E,(v)} \\
   \mathscr{S}_{U}(G,X) \arrow{r} & \mathscr{S}_{M^pM_p}(G_V, \mathcal{H}_V)\otimes \mathcal{O}_{E,(v)} \arrow{r} & \mathscr{S}_{J^{p} J_p}(G_{V_{\mathcal{L}}},\mathcal{H}_{V_{\mathcal{L}}}) \otimes \mathcal{O}_{E,(v)}.
  \end{tikzcd}
\end{equation}
\subsubsection{Change of parahoric and isogeny classes}
We set $\shgjinf:=\varprojlim_{U^p} \shgj$ and we let $\pi:\shgjinf \to \shginf$ denote the $G(\afp)$-equivariant map induced by $\pi_{J,K}$. We now define\footnote{This definition depends on the choice of Hodge embedding, at least a priori.} isogeny classes in $\shgjinf(\ovfp)$ using the Hodge embedding $(G,X) \to (G_{V_{\mathcal{L}'}} \mathcal{H}_{V_{\mathcal{L}'}})$, as in Section \ref{Sec:IsogenyClasses}. Similarly, we define isogeny classes in $\shginf(\ovfp)$ using the Hodge embedding $(G,X) \to (G_{V_{\mathcal{L}}},\mathcal{H}_{V_{\mathcal{L}}})$. For this, we choose tensors $s_{\beta} \in V_{\mathcal{L}', (p)}^{\otimes}$ cutting out $\mathcal{G}_{J, \zlocp}$ and tensors $s_{\alpha} \in V_{\mathcal{L}, (p)}^{\otimes}$ cutting out $\mathcal{G}_{K, \zlocp}$. 

By \cite[Proposition 7.7]{Z}, the forgetful map is compatible with isogeny classes in the sense that for $x \in \shgjinf(\ovfp)$ we have $\pi \left( \mathscr{I}_x \right) \subset \mathscr{I}_{\pi(x)}$. We will need the following (straightforward) refinement. 
\begin{Prop} \label{Prop:ForgetfulIsogenyClass}
Let $z,y \in \shgjinf(\ovfp)$ with the same image $x \in \shginf(\ovfp)$. Then $z$ and $y$ lie in the same isogeny class. In particular $\mathscr{I}_z=\pi^{-1}(\mathscr{I}_x)$.
\end{Prop}
\begin{proof}
The points $z,y$ correspond to $\mathcal{L}'$-sets of abelian varieties
\begin{align}
    ((A_1,\lambda_1,\eta_1) &\to (A_2,\lambda_2,\eta_2) \to \cdots \to (A_r,\lambda_r,\eta_r)) \\
    ((B_1,\lambda_1,\eta_1) &\to (B_2,\lambda_2,\eta_2) \to \cdots \to (B_r,\lambda_r,\eta_r))
\end{align}
such that the induced chains
\begin{align}
    ((A_{i_1},\lambda_{i_1},\eta_{i_1}) &\to (A_{i_2},\lambda_{i_2},\eta_{i_2}) \to \cdots \to (A_{i_s},\lambda_{i_s},\eta_{i_s})) \\
    ((B_{i_1},\lambda_{i_1},\eta_{i_1}) &\to (B_{i_2},\lambda_{i_2},\eta_{i_2}) \to \cdots \to (B_{i_s},\lambda_{i_s},\eta_{i_s}))
\end{align}
are isomorphic. There are unique quasi-isogenies $A_i \to B_i$ for all $i=1, \cdots, r$ that extend the given isomorphisms $A_{s_j} \simeq B_{s_j}$ for $j=1, \cdots, s$, and we would like to argue that the resulting quasi-isogeny
\begin{align}
    g:\prod_{i=1}^r A_i \to \prod_{i=1}^rB_i
\end{align}
is tensor preserving. By the discussion in Section 5.6 of \cite{Z} we can choose isomorphisms
\begin{align} \label{Eq:BasisI}
    \left(\bigoplus_{i=1}^r \Lambda_{i} \right) \otimes \zpbr \simeq \bigoplus_{i=1}^r \mathbb{D}(A_i[p^{\infty}]) \\
    \left(\bigoplus_{i=1}^r \Lambda_{i} \right) \otimes \zpbr \simeq \bigoplus_{i=1}^r \mathbb{D}(B_i[p^{\infty}])
\end{align}
taking $t_{\beta} \otimes 1$ to $t_{\beta,0,z}$ and $t_{\beta,0,y}$ respectively. There are induced isomorphisms
\begin{align} \label{Eq:BasisII}
    \left(\bigoplus_{i=1}^s \Lambda_{s_{i}} \right) \otimes \zpbr \simeq \bigoplus_{i=1}^s \mathbb{D}(A_i[p^{\infty}]) \\
    \left(\bigoplus_{i=1}^s \Lambda_{s_{i}} \right) \otimes \zpbr \simeq \bigoplus_{i=1}^s \mathbb{D}(B_i[p^{\infty}]).
\end{align}
taking $s_{\alpha} \otimes 1$ to $s_{\alpha,0,x}$. The isomorphism $f:\prod_{i=1}^s A_{s_i} \to \prod_{i=1}^s B_{s_i}$ coming from the equality $\pi(z)=x=\pi(y)$ is clearly tensor preserving. If we use the bases from \eqref{Eq:BasisII} then this means that the induced automorphism
\begin{align}
    \mathbb{D}(f) \in \operatorname{GL}\left(\bigoplus_{i=1}^s \Lambda_{s_{i,p}} \right)
\end{align}
lies in $G_{\mathbb{Z}_{p}}(\zpbr)$. This observation in combination with the following commutative diagram
\begin{equation}
\begin{tikzcd}
    G_{\mathbb{Z}_{p}}' \arrow{r} \arrow{d} & \operatorname{GL}\left(\bigoplus_{i=1}^r \Lambda_{i,p} \arrow{d} \right)\\
    G_{\mathbb{Z}_{p}} \arrow{r} & \operatorname{GL}\left(\bigoplus_{i=1}^s \Lambda_{s_{i,p}} \right)
\end{tikzcd}    
\end{equation}
shows that the automorphism induced by $g$ lands in $G_{\mathbb{Z}_{p}}(\zpbr)$. But this means that $\mathbb{D}(g) \in G_{\mathbb{Z}_{p}}'(\qpbr)=G_{\mathbb{Z}_{p}}(\qpbr)$ and therefore $g$ is tensor preserving. A similar argument shows that $g$ preserves the tensors for $\ell \not=p$. 
\end{proof}
We will also need the following lemma.
\begin{Lem} \label{Lem:NotCartesianYet}
The following diagram commutes:
\begin{equation} \label{Eq:NotCartesian}
  \begin{tikzcd}
  \shgj \arrow{r} \arrow{d} & \shtjmu \arrow{d} \\
  \shg \arrow{r} & \shtmu.
  \end{tikzcd}
\end{equation}
\end{Lem}
\begin{proof}
This is a consequence of \cite[Corollary 4.3.2]{PappasRapoportShtukas}, see Footnote \ref{FN:PR}.
\end{proof}

\subsection{CM Lifts}
In this section, we will prove a corollary of Theorem \ref{Thm:VerySpecialUniformisationAndLifting}, which is a slight generalisation of Theorem \ref{Thm:CMLifting}. 

\subsubsection{} Recall that a special point datum for $(G,X)$ is a triple $(T,h,i)$ where $(T,h)$ is a Shimura datum with $T$ a torus, and where $i:(T,h) \to (G,X)$ is an embedding of Shimura data such that $i(T)$ is a maximal torus. Associated to a special point datum $\mathfrak{s}=(T,h,i)$ is a $\qpbar$-point $x_{\mathfrak{s}}$ of $\mathbf{Sh}(G,X)$, see \cite[Section 5.7.1]{KisinShinZhu}. It is explained in loc. cit. that for any parahoric $U_p \subset G(\qp)$ its projection to $\mathbf{Sh}_{U_p}(G,X)(\qpbar)$ extends to a $\zpbar$ point of $\mathscr{S}_{U_p}(G,X)$. A \emph{special point} of $\mathbf{Sh}(G,X)(\qpbar)$ is a point that lies in the $G(\af)$-orbit of $x_{\mathfrak{s}}$ for some special point datum $\mathfrak{s}$. For any choice of parahoric $U_p \subset G(\qp)$, the projection of a special point to $\mathbf{Sh}_{U_p}(G,X)(\qpbar)$ extends (uniquely) to a $\zpbar$ point of $\mathscr{S}_{U_p}(G,X)$. The mod $p$ reductions of these extensions define points of $\shginf(\ovfp)$ that we call \emph{reductions of special points}.

\subsubsection{} Recall the notion of a very special parahoric from Section \ref{Sec:VerySpecialParahoric}. 
\begin{Cor} \label{Cor:Uniformisation}
Let $U_p' \subset G(\qp)$ be an arbitrary connected parahoric and suppose that there is a connected Iwahori subgroup $U_p''$ contained in $U_p'$ and a connected very special parahoric subgroup $U_p$ containing $U_p''$. Then each isogeny class of $\scrS_{U_p'}(G,X)(\ovfp)$ contains a point $x$ which is the reduction of a special point.  
\end{Cor}
\begin{proof}
Choose a connected Iwahori subgroup $U_p'' \subset U_p'$ and a connected very special parahoric subgroup $U_p \supset U_p''$ as in the assumptions of the theorem. We first prove the theorem for $\scrS_{U_p''}(G,X)$.

Let $z \in \scrS_{U_p''}(G,X)(\ovfp)$ and let $x$ be its image in $\scrS_{U_p}(G,X)(\ovfp)$. Then the isogeny class $\mathscr{I}_x$ contains the reduction of a special point $P \in \scrS_{U_p}(G,X)(\qpbar)$ by Theorem \ref{Thm:VerySpecialUniformisationAndLifting}. By definition, any lift $P'' \in \scrS_{U_p''}(G,X)(\qpbar)$ is also special. Thus we find that the inverse image of $\mathscr{I}_x$ under 
\begin{align}
  \scrS_{U_p''}(G,X)(\ovfp) \to \scrS_{U_p}(G,X)(\ovfp)
\end{align}
contains the reduction of a special point. But by Proposition \ref{Prop:ForgetfulIsogenyClass}, this inverse image is equal to $\mathscr{I}_z$ and so every isogeny class in $\scrS_{U_p''}(G,X)(\ovfp)$ contains the reduction of a special point. A similar argument shows that every isogeny class in $\scrS_{U_p'}(G,X)(\ovfp)$ contains the reduction of a special point.
\end{proof}

\subsection{Lifting uniformisation} \label{Sec:LiftingUniformisation}
From now on we let $K \subset \mathbb{S}$ be a $\sigma$-stable type corresponding to a connected very special parahoric. We let $U_p=\mG_{K}(\zp)$ and $U_p'=\mG_{\emptyset}(\zp)$; note that $U_p'$ is a connected parahoric subgroup by Lemma \ref{Lem:IwahoriConnected}. In this case, the commutative diagram from Lemma \ref{Lem:NotCartesianYet} is
\begin{equation} \label{Eq:NotYetCartesianIwahori}
  \begin{tikzcd}
  \shge \arrow{r} \arrow{d} & \shtemu \arrow{d} \\
  \shg \arrow{r} & \shtmu.
  \end{tikzcd}
\end{equation}
The goal of this section is to prove the following result. Let $x \in \shgeinf(\ovfp)$ and choose an isomorphism $\mathbb{D}_x \simeq V_{\mathbb{Z}_p} \otimes_{\zp} \zpbr$ sending $s_{\alpha, 0,x}$ to $s_{\alpha} \otimes 1$. Let $b \in G(\qpbr)$ be the element corresponding to the Frobenius of $\mathbb{D}_x$ under this isomorphism.
\begin{Thm} \label{Thm:LiftingUniformisation}
If for every sufficiently small compact open subgroup $U^p$ the diagram \eqref{Eq:NotYetCartesianIwahori} is Cartesian, then for $z\in \scrS_{U_p'}(G,X)(\ovfp)$ with associated element $b\in G(\brQ)$ there is a $G(\afp)$-equivariant bijection
			\begin{align}
			    I_z(\bbQ) \backslash \xmube(\ovfp) \times G(\bbA_f^p) \to \mathscr{I}_z.
			\end{align}
\end{Thm}
\subsubsection{} Let $z \in \scrS_{U_p'}(G,X)(\ovfp)$ with image $x \in \scrS_{U_p}(G,X)(\ovfp)$ and let $b \in G(\brQ)$ be as in the statement of Theorem \ref{Thm:VerySpecialUniformisationAndLifting}. Then Theorem \ref{Thm:VerySpecialUniformisationAndLifting} gives us a map of sets
\begin{align}
  G(\afp) \times \xmub(\ovfp) \to \mathscr{I}_x
\end{align}
and Lemma \ref{Lem:LocalUniformisation} gives us a map of stacks $\Theta_b:\xmub \to \shtmub$.
\begin{Lem} \label{Lem:CompatibilityUniformisationJointStratification}
The following diagram of groupoids commutes
\begin{equation}
  \begin{tikzcd}
  G(\afp) \times \xmub(\ovfp) \arrow[r, "pr_2"] \arrow{d} & \xmub(\ovfp) \arrow{d}{\Theta_b} \\
  \mathscr{I}_x \arrow{r} & \shtmub(\ovfp).
  \end{tikzcd}
\end{equation}
\end{Lem}
\begin{proof}
This follows from the compatibility of the uniformisation map with the `joint stratification' $\shg(\ovfp) \to \shtmu(\ovfp)$, as discussed in the proof of Axiom 4(b) in Section 8 of \cite{Z}.
\end{proof}
We have the following corollary of Lemma \ref{Lem:CompatibilityUniformisationJointStratification}.
\begin{Cor} \label{Cor:InverseImageUniformisation}
Let $\mathscr{L}_x$ be the inverse image in $\shgeinf(\ovfp)$ of $\mathscr{I}_x \subset \shginf(\ovfp)$. If the assumption of Theorem \ref{Thm:LiftingUniformisation} holds, then there is a $G(\afp)$-equivariant bijection
\begin{align}
  \mathscr{L}_x \simeq I_x(\bbQ) \backslash \xmube(\ovfp) \times G(\bbA_f^p).
\end{align}
\end{Cor}
\begin{proof}
Taking the inverse limit over $U^p$ of the Cartesian diagram of \eqref{Eq:NotYetCartesianIwahori}, we get the following $G(\afp)$-equivariant Cartesian diagram of groupoids
\begin{equation} \label{Eq:LiftingUniformisationCartesian}
\begin{tikzcd}
  \mathscr{L}_x \arrow{r} \arrow{d} & \shtemu(\ovfp) \arrow{d} \\
  \mathscr{I}_x \arrow{r} & \shtmu(\ovfp).
\end{tikzcd}  
\end{equation}
Theorem \ref{Thm:VerySpecialUniformisationAndLifting} gives us a bijection $I_x(\bbQ) \backslash \xmub(\ovfp) \times G(\bbA_f^p) \to \mathscr{I}_x$. Lemma \ref{Lem:LocalUniformisation} 
and Lemma \ref{Lem:CompatibilityUniformisationJointStratification} tell us that we can identify \eqref{Eq:LiftingUniformisationCartesian} with
\begin{equation} \label{Eq:LiftingUniformisationCartesian2}
\begin{tikzcd}
  \mathscr{L}_x \arrow{r} \arrow{d} & \left[\frac{\xmube(\ovfp)}{J_b(\mathbb{Q}_p)} \right] \arrow{d} \\
  I_x(\bbQ) \backslash \xmub(\ovfp) \times G(\bbA_f^p) \arrow{r} & \left[\frac{\xmub(\ovfp)}{J_b(\mathbb{Q}_p)} \right],
\end{tikzcd}  
\end{equation}
such that the bottom map is induced by the projection map $\xmub(\ovfp) \times G(\afp) \to \xmub(\ovfp)$ and the right vertical map is induced by the natural map $\xmube(\ovfp) \to \xmub(\ovfp)$. But now it is clear that there is a $G(\afp)$-equivariant bijection
\begin{align}
  \mathscr{L}_x \simeq I_x(\bbQ) \backslash \xmube(\ovfp) \times G(\bbA_f^p).
\end{align}
\end{proof}
\begin{proof}[Proof of Theorem \ref{Thm:LiftingUniformisation}]
The theorem is a direct consequence of Corollary \ref{Cor:InverseImageUniformisation}, which proves uniformisation for $\mathscr{L}_x$, and Proposition \ref{Prop:ForgetfulIsogenyClass}, which proves that $\mathscr{L}_x=\mathscr{I}_z$. 
\end{proof}

\subsection{Uniformisation and connected components} \label{Sec:UniformisationAndPiNaught} Define $G(\mathbb{Q})_{+}=G(\mathbb{Q}) \cap G(\mathbb{R})_{+}$, with $G(\mathbb{R})_{+}$ the inverse image of the identity component (in the real topology) of $\gad(\mathbb{R})$ under the natural map $G(\mathbb{R}) \to \gad(\mathbb{R})$. Let $\rho:\gsc \to \gder$ be the simply-connected cover of the derived subgroup of $G$; we will sometimes conflate groups like $\gsc(\Q)$ and $\gsc(\afp)$ with their images under $\rho$ by abuse of notation. Consider the profinite topological space
\begin{align}
  \pi(G):=\varprojlim_{U^p} G(\mathbb{Q})_{+} \backslash G(\mathbb{A}_f) / U^p U_p.
\end{align}
We have $\rho(\gsc(\Q)) \subset G(\mathbb{Q})_{+}$ since $\gsc(\R)$ is connected, and strong approximation for $\gsc$ away from $\infty$, see \cite[Theorem 7.12]{PR}, tells us that the closure of $\rho(\gsc(\Q))$ in $G(\af)$ contains $\rho(\gsc(\af))$. Moreover, the subset $G(\mathbb{Q})_{+} \rho(\gsc(\af))$ is closed in $G(\af)$ since $(G,X)$ is of Hodge type, see \cite[Section 2.0.15]{Deligne}. This means that $G(\mathbb{Q})_{+} \backslash G(\af)/ \rho(\gsc(\af))$ is Hausdorff. Thus the natural action of $U_p$ on it has compact stabilisers and compact orbits, since $U_p$ is compact. We can therefore deduce from \cite[Lemma 4.20]{Milne} that the natural map 
\begin{align}
  G(\mathbb{Q})_{+} \backslash \frac{G(\mathbb{A}_f)}{\rho(\gsc(\af))}/U_p \to \pi(G)
\end{align}
is a homeomorphism. We see that $\pi(G)$ is an abelian group since $\tfrac{G(\mathbb{A}_f)}{\rho(\gsc(\af))}$ is.
\subsubsection{} \label{Sec:StrongApproximation} By Lemma \ref{Lem:Fundamentalgroup} below, we may make the identification
\begin{align} \label{Eq:IdentificationPiNaught}
  \pi(G) = G(\mathbb{Q})_{+} \backslash \left(\pi_1(G)_I^{\sigma} \times \frac{G(\afp)}{\rho(\gsc(\afp))}\right). 
\end{align}
In particular there is a natural surjective group homomorphism $\pi_1(G)_I^{\sigma} \times G(\afp) \to \pi(G)$. 
\begin{Lem} \label{Lem:Fundamentalgroup}
Let $\mathbb{G}$ be a connected reductive group over $\qp$ that splits over a tamely ramified extension and let $\mathcal{G}$ be a parahoric group scheme for $\mathbb{G}$. Then there is a natural isomorphism
\begin{align}
  \frac{\mathbb{G}(\mathbb{Q}_p)}{\rho(\mathbb{G}^{\mathrm{sc}}(\mathbb{Q}_p)) \cdot \mathcal{G}(\zp)} \simeq \pi_1(\mathbb{G})_I^{\sigma}.
\end{align}
\end{Lem}
\begin{proof}
Recall that we have the surjective Kottwitz homomorphism $\tilde{k}_\mathbb{G}:\mathbb{G}(\qpbr) \to \pi_1(\mathbb{G})_I$ with kernel given by $\rho(\mathbb{G}^{\mathrm{sc}}(\qpbr)) \cdot \mathcal{T}(\zpbr) = \rho(\mathbb{G}^{\mathrm{sc}}(\qpbr)) \cdot \mathcal{G}(\zpbr)$ (see \cite[Lemma 17 of the appendix]{PappasRapoportAffineFlag}), where $\mathcal{T}$ is the connected Néron model of a standard torus $T$ of $\mathbb{G}$. Recall moreover that $\tilde{k}_{\mathbb{G}}$ restricts to a surjective map $\tilde{\kappa}_{\mathbb{G},0}:\mathbb{G}(\qp) \to \pi_1(\mathbb{G})_I^{\sigma}$ by \cite[Section 7.7]{Kottwitz2}. Thus when $\mathbb{G}=T$ is a torus, we have a short exact sequence
\begin{align}
    0 \to \mathcal{T}(\zpbr) \to T(\qpbr) \to \pi_1(G)_I \to 0,
\end{align}
that remains exact upon taking $\sigma$-invariants, proving the lemma for tori. If $\mathbb{G}^{\mathrm{der}}$ is simply connected, then there is a canonical identification $\pi_1(\mathbb{G})=\pi_1(\mathbb{G}^{\mathrm{ab}})$, where $\mathbb{G}^{\mathrm{ab}}$ is the maximal abelian quotient of $\mathbb{G}$. We can consider the morphism of short exact sequences
\begin{equation}
    \begin{tikzcd}
    1 \arrow{r} & \mathbb{G}^{\mathrm{der}}(\qp) \arrow{d} \arrow{r} & \mathbb{G}(\qp) \arrow{r} \arrow{d} & \mathbb{G}^{\mathrm{ab}}(\qp) \arrow{d}  \arrow{r} & 1 \\
    & 1 \arrow{r} & \pi_1(\mathbb{G})_I^{\sigma} \arrow{r} & \pi_1(\mathbb{G}^{\mathrm{ab}})_I^{\sigma} \arrow{r} & 1.
    \end{tikzcd}
\end{equation}
The lemma now follows from the well known fact (see e.g.\cite[Proposition 2.6.2]{vHXiao}) that the image of $\mathcal{G}(\zp)$ in $\mathbb{G}^{\mathrm{ab}}(\qp)$ is equal to $\mathcal{D}(\zp)$, where $\mathcal{D}$ is the connected Néron model of $\mathbb{G}^{\mathrm{ab}}$. 

For general $\mathbb{G}$, choose a $z$-extension $1 \to Z \to \tilde{\mathbb{G}} \to \mathbb{G} \to 1$ in the sense of \cite[Section 11.4]{KalethaPrasad}. Then it follows from \cite[Proposition 11.5.3]{KalethaPrasad} that $\ker \tilde{\kappa}_{\tilde{\mathbb{G}},0} \to \tilde{\kappa}_{\mathbb{G},0}$ is surjective. Choosing a parahoric model $\tilde{\mathcal{G}}$ of $\tilde{\mathbb{G}}$ together with a morphism $\tilde{\mathcal{G}} \to \mathcal{G}$, see \cite[Section 1.1.3]{KisinPappas}, we see that it suffices to show that $\tilde{\mathcal{G}}(\zp) \to \mathcal{G}(\zp)$ is surjective. For this, we note that by \cite[Proposition 1.1.4]{KisinPappas} there is a short exact sequence of group schemes over $\zp$ (here we use the tameness assumption)
\begin{align}
    1 \to \mathcal{Z} \to \tilde{\mathcal{G}} \to \mathcal{G} \to 1,
\end{align}
where $\mathcal{Z}$ has smooth connected special fiber. The surjectivity of $\tilde{\mathcal{G}}(\zp) \to \mathcal{G}(\zp)$ now follows from Lang's lemma.
\end{proof}

\subsubsection{} Define (cf. \cite[Section 2.1.3]{Deligne}) 
\begin{align} \label{Eq:ConnectedComponentsGeneric}
  \pi(G,X):=\varprojlim_{U^p} \pi_0(\mathbf{Sh}_{U}(G,X)_{\overline{\mathbb{C}}})= \varprojlim_{U^p} G(\mathbb{Q}) \backslash \left( \pi_0(X) \times G(\af) /U^p U_p\right).
\end{align}
This is a quotient (where now the inverse limit runs over all compact open subgroups $U \subset G(\af)$)
\begin{align}
  \varprojlim_{U} G(\mathbb{Q}) \backslash \left( \pi_0(X) \times G(\af) /U \right),
\end{align}
on which $G(\af)$ acts through the abelian group $G(\af)/\rho(\gsc(\af))$, again by strong approximation for $\gsc$ away from infinity. By the discussion above, this induces an action of $G(\afp) \times \pi_1(G)_I^{\sigma}$ on $\pi(G,X)$ which makes it into a torsor for $\pi(G)$, see \cite[Section 5.5.4]{KisinShinZhu}.
\smallskip 

Recall that $U_p$ is a very special parahoric, which implies that the integral model $\mathscr{S}_{U^pU_p}(G,X)$ has normal special fiber, see \cite[Collary 4.6.26]{KisinPappas}. Then \cite[Corollary 4.1.11]{MP} tells us that for all choices of $U^p$, for each finite extension $F$ of the reflex field $E$ and any place $w$ of $F$ extending $v$, the natural maps
\begin{align} \label{Eq:ConnectedComponentsSpecial}
  \pi_0(\mathbf{Sh}_{U^pU_p}(G,X) \otimes_E F) \rightarrow \pi_0(\mathscr{S}_{U^pU_p}(G,X) \otimes_{\mathcal{O}_{E,(v)}} \mathcal{O}_{F,(w)}) \leftarrow \pi_0(\shg \otimes k(w))
\end{align}
are isomorphisms. Thus there is a natural $G(\afp)$-equivariant isomorphism $\pi_0(\shginf) \to \pi(G,X)$, which turns $\pi_0(\shginf)$ into a torsor for $\pi(G)$ and equips it with an action of $G(\afp) \times \pi_1(G)_I^{\sigma}$.

\subsubsection{} As before $\mathcal{G}_K$ denotes a connected very special parahoric group scheme. Let $x \in \shginf(\ovfp)$ and $b\in G(\brQ)$ the associated element that is well-defined up to $\mG_{K}(\zpbr)$-conjugacy. The Kottwitz homomorphism induces a natural map of perfect schemes
\begin{align}
  \kappa: \xmub \to \grg \to \pi_0(\grg) \simeq \pi_1(G)_I,
\end{align}
with image $c_{[b], \mu} + \pi_1(G)_I^{\sigma} \subset \pi_1(G)_I$, see \cite[Lemma 6.1]{HZ}. As in \cite[Section 6.7]{Z} we have $1 \in \xmub(\ovfp)$, which implies that the coset $c_{[b], \mu} + \pi_1(G)_I^{\sigma}$ contains $1$ and is thus equal to $\pi_1(G)_I^{\sigma}$. In particular, the map $\kappa$ takes values in $\pi_1(G)_I^{\sigma}$. Theorem \ref{Thm:VerySpecialUniformisationAndLifting} gives us a $G(\afp)$-equivariant map of sets 
\begin{align}
  i_x:\xmub(\ovfp) \times G(\afp) \to \shginf(\ovfp),
\end{align}
sending $(1,1)$ to $x$. 
\begin{Prop} \label{Prop:UniformisationAndConnectedComponents}
Consider the composition $\xmub(\ovfp) \times G(\afp) \to \shginf(\ovfp) \to \pi_0(\shginf) =\pi(G,X)$, and let $\underline{x}$ be the image of $x$ in $\pi(G,X)$. Then the image of $(y,g^p)$ in $\pi(G,X)$ is given by
\begin{align}
  (\kappa(y),g^p) \cdot \underline{x},
\end{align}
where $\cdot$ denotes the natural action of $\pi_1(G)_I^{\sigma} \times G(\afp)$ on $\pi(G,X)$ constructed above.
\end{Prop}
\begin{proof}
By the $G(\afp)$-equivariance of the map $i_x$, it suffices to prove the theorem for $g^p=1$ or for the map $\xmub(\ovfp) \to \shginf(\ovfp)$. The map $\xmub(\ovfp) \to \shginf(\ovfp)$ upgrades to a map of perfect schemes $\xmub \to \shginf$ by the proof of \cite[Proposition 5.2.2]{HeZhouZhu}. Therefore the image of $y \in \xmub(\ovfp)$ in $\pi(G,X)$ only depends on the connected component that $y$ lies in. Thus the result is true for a union of connected components $X(\mu,b)_K^{\circ}$ of $\xmub$. Moreover, the result is clearly true for $y=1$. \medskip

Now we follow the proof of Proposition \ref{Prop:UniformisationVerySpecial} and freely use the notation from that proof: Let $M\subset G_{\bbQ_p}$ be the standard Levi subgroup given by the centraliser of the Newton cocharacter $\overline\nu_b$. By Theorem \ref{thm: conn ADLV}, there exists $\lambda\in I_{\mu,b,M}$ and an element $$g\in X(\mu,b)^\circ_K\cap X^M(\lambda,b)_M.$$ We may then replace $x$ by $i_x(g)$ to assume that $b\in M(\brQ)$ and furthermore that $b=\dot\tau_\lambda$ where $\tau_\lambda\in\Omega_M$ corresponds to $\kappa_M(b)\in \pi_1(M)_I$. \smallskip 

Arguing as in the proof of \ref{Prop:UniformisationVerySpecial}, we can find a finite extension $L$ of $\qpbr$ and choose an $(M,\mu_y)$-adapted lifting $\tilde{\scrG}/\calO_{L}$ of $\scrG_x$ (cf. \cite[Definition 4.6]{Z}) which corresponds to a point $\tilde{x}\in\scrS_{U_p}(G,X)(\calO_{L})$. The construction in \cite[Proposition 5.14]{Z} gives us a map $$\iota:M(\bbQ_p)/\calM(\bbZ_p)\rightarrow X^M(\lambda,b)_{K_M},$$
whose composition with $X^M(\lambda,b)_{K_M} \to \xmub$ fits into the following commutative diagram
\begin{equation}
  \begin{tikzcd}
  M(\qp) \arrow{dr} \arrow{rr} & & \xmub(\ovfp) \arrow{dl}{\kappa} \\
  & \pi_1(G)_I^{\sigma},
  \end{tikzcd}
\end{equation}
where the left diagonal map is the composition $M(\qp) \to G(\qp) \to \pi_1(G)_I^{\sigma}$. Choose a lift of $\tilde{x}$ to a point $z \in \mathbf{Sh}(G,X)(C)$, where $C$ is an algebraic closure of $\qpbr$. Then by construction the map $\iota$ fits into the following diagram (compare with the diagram in \cite[Corollary 1.4.12]{KisinPoints})
\begin{equation}
  \begin{tikzcd}
  M(\mathbb{Q}_p) \arrow{r} \arrow{d} & \mathscr{S}_{U_p}(G,X)(\mathcal{O}_{C}) \arrow{d} \\
  \xmub(\ovfp) \arrow{r} & \shginf(\ovfp).
  \end{tikzcd}
\end{equation}
Here the top horizontal map is given by the (Hecke) action of $M(\mathbb{Q}_p) \subset G(\qp)$ on $z \in \mathbf{Sh}(G,X)(C)$ followed by projection to back to $\mathscr{S}_{U_p}(C)$, extending to $\scrS_{U_p}(G,X)(\mathcal{O}_C)$ and reducing mod $p$. 

We see that elements $g \in \xmub(\ovfp)$ in the image of $M(\qp) \to \xmub(\ovfp)$ satisfy the conclusion of the proposition. Moreover, this means that the result holds for all points $g \in \xmub(\ovfp)$ lying in a connected component of $\xmub$ intersecting the image of the map $M(\qp) \to \xmub(\ovfp)$. But the map $$M(\bbQ_p)/\calM(\bbZ_p)\rightarrow \pi_0(X^M(\lambda,b)_{K_M})$$ is surjective by \cite[Proposition 5.19]{Z} and moreover
\begin{align}
  \pi_0(X^M(\lambda,b)_{K_M}) \to \pi_0(\xmub)
\end{align}
is surjective by Theorem \ref{thm: conn ADLV}. Thus every connected component of $\xmub$ contains a point in the image of $M(\qp) \to \xmub(\ovfp)$, and so we are done.
\end{proof}
\begin{Cor} \label{Cor:BasicNonEmptiness}
Let $\tau \in \admu$ be the unique element of length zero. Then
\begin{align}
  \shge(\tau) \to \pi_0(\shg)
\end{align}
is surjective.
\end{Cor}
\begin{proof}
It suffices to prove this for the analogous map $\shgeinf(\tau) \to \pi_0(\shginf)$. Since $\shgeinf(\tau)$ is contained in the basic locus we can use \cite[Proposition 6.5(i)]{Z} to produce for $x \in \shgeinf(\tau)$ a uniformisation map
\begin{align}
    \xmube(\ovfp) \times G(\afp) \to \shgeinf 
\end{align}
which, as in \cite[proof of Axiom 5]{Z}, restricts to a map
\begin{align}
    i_x:\xmube(\tau)(\ovfp) \times G(\afp) \to \shgeinf(\tau)(\ovfp).
\end{align}
Moreover, the following diagram commutes (by construction, see \cite[Proposition 7.8]{Z})
\begin{equation}
    \begin{tikzcd}
    \xmube(\tau)(\ovfp) \times G(\afp) \arrow{d} \arrow{r}{i_x} & \shgeinf(\tau)(\ovfp) \arrow{d} \\
    \xmub(\ovfp) \times G(\afp) \arrow{r}{i_{z}} & \shginf(\ovfp),
    \end{tikzcd}
\end{equation}
where $z$ is the image in $\shginf$ of $x$. Since $\xmube(\tau)(\ovfp) \subset \xmube(\ovfp)$ is $J_b(\qp)$-stable, it follows that its image in $\xmub(\ovfp)$ is $J_b(\qp)$-stable. Thus its image via $\kappa$ in $\pi_1(G)_I^{\sigma}$ is $J_b(\qp)$-stable. Now since $b$ is basic there is a $J_b(\qp)$-equivariant isomorphism $\pi_1(G)_I^{\sigma} = \pi_1(J_b)_I^{\sigma}$, and therefore by Lemma \ref{Lem:Fundamentalgroup} we see that $\xmube(\tau)(\ovfp)$ surjects onto $\pi_1(G)_I^{\sigma}$. The result now follows from Proposition \ref{Prop:UniformisationAndConnectedComponents} and the fact that $\pi_1(G)_I^{\sigma} \times G(\afp)$ acts transitively on $\pi_0(\shginf)$.
\end{proof}
\begin{Cor} \label{Cor:NonEmptiness}
For $w \in \admu$, the map
\begin{align}
  \shge(w) \to \pi_0(\shg)
\end{align}
is surjective. 
\end{Cor}
\begin{proof}
By \cite[Theorem 4.1]{He-Rapoport}, this follows from Corollary \ref{Cor:BasicNonEmptiness}.
\end{proof}

\section{The Cartesian diagram} \label{Sec:Cartesian}
Let the notation be as in Section \ref{Sec:Uniformisation}, in particular $\mG_{K}$ is a connected very special parahoric group scheme. Define a sheaf $\shget$ via the following fiber product diagram
\begin{equation} \label{Eq:CartesianDiagramDef}
  \begin{tikzcd}
  \shget \arrow{r} \arrow{d} & \shtemu \arrow{d} \\
  \shg \arrow{r} & \shtmu.
  \end{tikzcd}
\end{equation}
In particular, $\shget$ is $\operatorname{Sh}_{G,U',\star}$ from the introduction. Proposition \ref{Prop:Representable} tells us that $\shget$ is (representable by) a perfect algebraic space which is perfectly proper over $\shg$. The universal property of the fiber product gives us a morphism $\iota:\shge \to \shget$ and the goal of this section is to show that $\iota$ is an isomorphism, under some hypotheses. 

In Section \ref{Sec:ClosedImmersion} we will show that $\iota$ is a closed immersion. In Section \ref{Sec:PerfectlySmooth} we will show that $\shget$ is equidimensional of the same dimension as $\shg$. In Section \ref{Sec:KRConnected} we will show that each irreducible component of $\shget$ can be moved into $\shge$ using prime-to-$p$ Hecke operators. We prove this by degenerating to the zero-dimensional KR stratum, which we describe explicitly using Rapoport--Zink uniformisation of the basic locus.

\subsection{The natural map is a closed immersion} \label{Sec:ClosedImmersion}  Because the morphism $\shginf \to \shtmu$ is $G(\afp)$-equivariant, see \cite[Theorem 1.3.4]{PappasRapoportShtukas}, we can form $\shget$ for every choice of prime-to-$p$ level subgroup $U^p$. Then there is an induced action of $G(\afp)$ on $\shgetinf:=\varprojlim_{U^p} \shget$, such that the natural maps $\shgeinf \to \shgetinf$ and $\shgetinf \to \shginf$ are $G(\afp)$-equivariant. 

\subsubsection{} \label{Sec:TheCube}
Let $\mathcal{P}, \mathcal{P}'$ be the parahoric group schemes with $\mathcal{P}(\zp)=M_p$ and $\mathcal{P}'(\zp)=M_p'$, see Section \ref{Sec:ChangeOfParahoric}. Let $\shegsp$ and $\shgsp$ be the perfections of the geometric special fibers of the schemes (introduced in Section \ref{Sec:ChangeOfParahoric})
\begin{align}
  \mathscr{S}_{M^pM_p'}(G_V, \mathcal{H}_V)\otimes_{\Z_{(p)}}\mathcal{O}_{E,v} \quad \text{and} \quad \mathscr{S}_{M^pM_p}(G_V, \mathcal{H}_V)\otimes_{\Z_{(p)}}\mathcal{O}_{E,v},
\end{align}
respectively. Now consider the following commutative diagram deduced from \eqref{Eq:ParahoricForgetfulGandSiegel} (which is commutative by \cite[Theorem 4.3.1]{PappasRapoportShtukas}, see Footnote \ref{FN:PR})
\begin{equation} \label{Eq:TheCube}
  \begin{tikzcd}
  \shge \arrow{r} \arrow{drr} & \shget \arrow{rr} \arrow{dd} \arrow[dr, densely dotted] & & \shtgemu \arrow{dr} \arrow[-,d] \\
  & & \shegsp \arrow{rr} \arrow{dd} & \arrow{d} & \operatorname{Sht}_{G_V, \mathcal{P}', \{\mu\}} \arrow{dd}\\
  & \shg \arrow[-,r] \arrow{dr} & \arrow{r} & \shtgmu \arrow{dr} \\
  & & \shgsp \arrow{rr} && \operatorname{Sht}_{G_V, \mathcal{P}, \{\mu\}}.
  \end{tikzcd}
\end{equation}

\begin{Lem} \label{Lem:CartesianGSp}
The front face of the cube, that is, the square involving $\shegsp$, $\operatorname{Sht}_{G_V, \mathcal{P}', \{\mu\}}$, $\shgsp$ and $\operatorname{Sht}_{G_V, \mathcal{P}, \{\mu\}}$ is Cartesian.
\end{Lem}
\begin{proof}
The stack $\operatorname{Sht}_{G_V, \mathcal{P}', \{\mu\}}$ is\footnote{To be precise, the stack $\operatorname{Sht}_{G_V, \mathcal{P}', \{\mu\}}$ is a stack of $\mathcal{L}'$-chains of polarised Dieudonn\'e modules. By \cite[Theorem 1.2]{Lau}, for a perfect ring $R$ there is an equivalence of categories between $\mathcal{L}'$-chains of polarised Dieudonn\'e modules over $W(R)$ and $\mathcal{L}'$-chains of polarised $p$-divisible groups over $\spec R$, which gives the desired description of $\operatorname{Sht}_{G_V, \mathcal{P}', \{\mu\}}(R)$.}a moduli stack of $\mathcal{L}'$-chains of (polarised) $p$-divisible groups and the stack $\operatorname{Sht}_{G_V, \mathcal{P}, \{\mu\}}$ is a moduli stack of $\mathcal{L}$-chains of polarised $p$-divisible groups. The natural map $\shegsp \to \operatorname{Sht}_{G_V, \mathcal{P}', \{\mu\}}$ sends an $\mathcal{L}'$-chain of abelian varieties to the corresponding $\mathcal{L}'$-chain of $p$-divisible groups. The map $\shgsp \to \operatorname{Sht}_{G_V, \mathcal{P}, \{\mu\}}$ has a similar description. Moreover, the map $\operatorname{Sht}_{G_V, \mathcal{P}', \{\mu\}} \to \operatorname{Sht}_{G_V, \mathcal{P}, \{\mu\}}$ sends an $\mathcal{L}'$-chain of (polarised) $p$-divisible groups to the underlying $\mathcal{L}$-chain of (polarised) $p$-divisible groups.

The statement of the lemma now comes down to the following claim: Given an $\mathcal{L}$-chain $A_{\mathcal{L}}$ of (weakly polarised) abelian varieties, an $\mathcal{L}'$-chain $X_{\mathcal{L}'}$ of (polarised) $p$-divisible groups and an isomorphism from $A[p^{\infty}]_{\mathcal{L}}$ to the underlying $\mathcal{L}$-chain of $X_{\mathcal{L}}$, then there is a unique $\mathcal{L}'$-chain of abelian varieties $A_{\mathcal{L}'}$ with underlying $\mathcal{L}$-chain given by $A_{\mathcal{L}}$ and with $p$-divisible group $A[p^{\infty}]_{\mathcal{L}'}=X_{\mathcal{L}'}$. This claim follows from the following simpler claim: Given an abelian variety $A$ and a quasi-isogeny of $p$-divisible groups $f:A[p^{\infty}] \dashrightarrow X$, there is a unique triple $(B,\alpha,g)$ where $B$ is an abelian variety, where $\alpha:B[p^{\infty}] \to X$ is an isomorphism and $g:A \dashrightarrow B$ is a $p$-power quasi-isogeny such that $\alpha \circ g=f$. The proof of this simpler claim is explained in \cite[Section 6.13]{RapoportZink}.
\end{proof}
\begin{Lem} \label{Lem:DottedArrowCube}
The dotted arrow in \eqref{Eq:TheCube} exists.
\end{Lem}
\begin{proof}
This is an immediate consequence of Lemma \ref{Lem:CartesianGSp} and the universal property of the fiber product.
\end{proof}
\begin{Prop} \label{Prop:ClosedImmersion}
The morphism $\iota:\shge \to \shget$ induced by the universal property of $\shget$ is a closed immersion. 
\end{Prop}
We start by recalling a lemma.
\begin{Lem} \label{Lem:ProperBijection}
If $f:X \to Y$ is a perfectly proper morphism between pfp algebraic spaces over $\ovfp$ that is injective on $\ovfp$-points, then $f$ is a closed immersion.
\end{Lem}
\begin{proof}
The proof immediately reduces to the case that $X$ and $Y$ are pfp schemes over $\ovfp$. Then the image $f(X) \subset Y$ is closed and we can consider it as a subscheme with the reduced induced structure. The natural map $f:X \to f(X)$ is a bijection on $\ovfp$-points and thus an isomorphism by \cite[Corollary 6.10]{BhattScholze}; the result follows. 
\end{proof}
\begin{proof}[Proof of Proposition \ref{Prop:ClosedImmersion}]
The map $\iota$ is a morphism of perfect algebraic spaces that are perfectly proper over $\shg$, and $\iota$ is therefore perfectly proper.\footnote{Here we are using the cancellation theorem for proper morphisms, see e.g. \cite[Theorem 11.1.1]{VakilNotes}.}
By Lemma \ref{Lem:ProperBijection}, it thus suffices to prove that $\iota$ induces an injective map on $\ovfp$-points. \smallskip 
 
Now \cite[Corollary 6.3]{Z} tells us that a point $x \in \shge(\ovfp)$ is determined by its image in $\shegsp(\ovfp)$ and the tensors in the Dieudonné module of its $p$-divisible group. The tensors are determined by the image of $x$ in $\shtgemu(\ovfp)$. By Lemma \ref{Lem:DottedArrowCube}, the morphism $\shge \to \shegsp$ factors through $\shget$ and so the image of $x$ in $\shget(\ovfp)$ remembers both the image of $x$ in $\shegsp(\ovfp)$ and the image of $x$ in $\shtgemu(\ovfp)$; the lemma is proved. 
\end{proof}
\begin{Lem} \label{Lem:finite}
The morphism $f:\shget \to \shegsp$ constructed in Lemma \ref{Lem:DottedArrowCube} is finite.
\end{Lem}
\begin{proof}
By the proof of Proposition \ref{Prop:ClosedImmersion} there is a commutative diagram
\begin{equation} \label{eq:ForgetfulGAndSiegel}
  \begin{tikzcd}
  \shget \arrow{d}{\xi} \arrow{r}{f} & \shegsp \arrow{d}{\chi} \\
  \shg \arrow{r}{f'} & \shgsp
  \end{tikzcd}
\end{equation}
with $f'$ finite. It suffices to show that $f$ is quasi-finite, since its source and target are perfectly proper over
$\shgsp$.\footnote{Indeed, take a deperfection $h:Z_1 \to Z_2$ of $f$ with $Z_i$ finite type algebraic spaces, which exists by \cite[Proposition A.1.8.(3)]{XiaoZhu}. Then $h$ is proper and quasi-finite and hence finite, which implies that $f$ is finite.} It suffices moreover to prove that $f$ has finite fibers on $\ovfp$-points, by choosing a finite type deperfection using \cite[Proposition A.1.8.(3)]{XiaoZhu} and applying the usual argument to the deperfection.
\begin{Claim} \label{Claim:Fibers}
For $x \in \shg(\ovfp)$ with image $y=f'(x)$ the map
\begin{align}
  f:\xi^{-1}(x) \to \chi^{-1}(y)
\end{align}
is injective. 
\end{Claim}
Granting the claim for now, we will finish the proof: To show that $f'$ has finite fibers, we choose $y' \in \shegsp(\ovfp)$ and set $\chi(y')=y$ with inverse images $x_1, \cdots, x_n \in \shg(\ovfp)$ under $f$. Then each element of $f^{-1}(y')$ maps to $x_i$ for some $i$, which gives 
\begin{align}
    f^{-1}(y')=\coprod_{i=1}^{n} f^{-1}(y')_i.
\end{align}
But the natural maps $f^{-1}(y')_i \to \chi^{-1}(y)$ are injective by Claim \ref{Claim:Fibers}, and so each $f^{-1}(y')_i$ contains at most one element. This implies that the cardinality of $f^{-1}(y')$ is bounded by $n$.
\end{proof}
\begin{proof}[Proof of Claim \ref{Claim:Fibers}]
To prove this injectivity on fibers we return to the commutative cube from Section \ref{Sec:TheCube}, see equation \ref{Eq:TheCube}. The square involving the four objects with subscript $G$ is Cartesian by construction, and the square involving the four objects with subscript $G_V$ is Cartesian by Lemma \ref{Lem:CartesianGSp}. To prove the claim, we will make use of the following fact: Given a Cartesian diagram (of presheaves of groupoids on any category)
\begin{equation}
    \begin{tikzcd}
        A \arrow{r}{h_1}\arrow{d} \arrow{r} & A' \arrow{d}{h_2} \\
        A'' \arrow{r} & A''',
    \end{tikzcd}
\end{equation}
then for any map $x:B \to A''$ the natural map $B \times_{A'''} A' \to B \times_{A''} A$ is an equivalence. In other words, Cartesian squares induce isomorphisms on fibers of maps. Using this fact, the injectivity of the map on fibers in \eqref{eq:ForgetfulGAndSiegel} can instead be proved for the square
\begin{equation}
  \begin{tikzcd}
    \shtgemu \arrow{d} \arrow{r} & \operatorname{Sht}_{G_V, \mathcal{P}', \{\mu\}} \arrow{d} \\
    \shtgmu \arrow{r} & \operatorname{Sht}_{G_V, \mathcal{P}, \{\mu\}}.
  \end{tikzcd}
\end{equation}
Moreover, since the spaces of shtukas of type $\{\mu\}$ sit inside the spaces of all shtukas, we can reduce to showing the injectivity of the map on fibers for
\begin{equation}
  \begin{tikzcd}
    \operatorname{Sht}_{G, \emptyset} \arrow{d} \arrow{r} & \operatorname{Sht}_{G_V, \mathcal{P}'} \arrow{d} \\
    \operatorname{Sht}_{G, K} \arrow{r} & \operatorname{Sht}_{G_V, \mathcal{P}}.
  \end{tikzcd}
\end{equation}
Recall from the proof of Corollary \ref{Cor:ForgetfulRepresentable} the Cartesian diagrams (equation \eqref{Eq:ForgetfulLocalCartesianClassifying})
\begin{equation} 
  \begin{tikzcd}
  \operatorname{Sht}_{G, \emptyset} \arrow{r} \arrow{d} & \operatorname{Sht}_{G, K} \arrow{d} \\
  \mathbf{B} \lpGe \arrow{r} & \mathbf{B} \lpGk,
  \end{tikzcd} \qquad
  \begin{tikzcd}
  \operatorname{Sht}_{G_V, \mathcal{P}'} \arrow{r} \arrow{d} & \operatorname{Sht}_{G_V, \mathcal{P}} \arrow{d} \\
  \mathbf{B} L^+ \mathcal{P}' \arrow{r} & \mathbf{B} L^+\mathcal{P},
  \end{tikzcd}
\end{equation}
that fit into a commutative cube that we will not draw. This reduces the problem to showing the injectivity statement for the map on fibers in the diagram
\begin{equation}
  \begin{tikzcd}
    \mathbf{B} \lpGe \arrow{r} \arrow{d} & \mathbf{B} L^+ \mathcal{P}' \arrow{d} \\
    \mathbf{B} \lpGk \arrow{r} & \mathbf{B} L^+\mathcal{P},
  \end{tikzcd}
\end{equation}
which comes down to showing injectivity of the map of partial flag varieties
\begin{align}
  \frac{\lpGk}{\lpGe} \to \frac{L^+\mathcal{P}}{L^+ \mathcal{P}'}.
\end{align}
This last statement follows from the fact that the intersection of $L^+ \mathcal{P}'$ with $\lpGk$ is equal to $\lpGe$. This is true by construction of $\mathcal{P}, \mathcal{P}'$ and the fact that $\mathcal{G}_K$ and $\mathcal{G}_{\emptyset}$ are connected parahoric subgroups (the first by assumption, the second by Lemma \ref{Lem:IwahoriConnected}).
\end{proof}

\subsection{A perfect local model diagram} \label{Sec:PerfectlySmooth}
Consider the composition $\hat{\lambda}$ (the last arrow comes from the diagram in Remark \ref{Rem:ShtukaLocalModelII})
\begin{align}
  \shget \xrightarrow{} \shtemu \to \shtemu^{(m,1)} \to \left[ \mloce/\overline{\mG}_{\emptyset} \right].
\end{align}
We will think of this as a local model diagram for $\shget$.
\begin{Prop} \label{Prop:perfectlysmooth}
The morphism $\shget \to \left[ \mloce/\overline{\mG}_{\emptyset} \right]$ is weakly perfectly smooth and $\shget$ is equidimensional of the same dimension as $\shge$.
\end{Prop}
\begin{proof}
We will use the results of Section \ref{Sec:RestrictedLocalShtukas}.
Fix $n \ge 2$ and choose $m \gg 0$ such that the action $\operatorname{Ad}_{\sigma} \lpGk$ on $\mlocn$ factors through $\lmG$ and such that $m$ satisfies the assumption of Proposition \ref{Prop:CartesianInfiniteToFinite}. As explained in Section \ref{Sec:RestrictedShtukasShimuraVarieties} the natural morphism
\begin{align}
  \shg \to \shtmu^{(m,n), \mathrm{loc}}
\end{align}
is perfectly smooth. Combining this with the discussion in Section \ref{Sec:GeneralRestrictedShtukas}, we find that (after possibly increasing $n$) the composition with the natural map
\begin{align}
  \shtmu^{(m,n), \mathrm{loc}} \to \left[\frac{\ker \gamma \backslash \mlocinf}{\operatorname{Ad}_{\sigma} \lmG} \right]
\end{align}
is weakly perfectly smooth. Proposition \ref{Prop:CartesianInfiniteToFinite} implies that the right square in the following diagram is Cartesian
\begin{equation} \label{Eq:CartesianTimesTwo}
  \begin{tikzcd}
  \shget \arrow{r} \arrow{d} & \shtemu \arrow{d} \arrow{r} & \left[\frac{\ker \gamma \backslash \mloceinf}{\operatorname{Ad}_{\sigma} H_m} \right] \arrow{d}\\
  \shg \arrow{r} & \shtmu \arrow{r} & \left[\frac{\ker \gamma \backslash \mlocinf}{\operatorname{Ad}_{\sigma} \lmG} \right].
  \end{tikzcd}
\end{equation}
Since the left square is Cartesian by construction, it follows that the outer square is also Cartesian. Moreover, Lemma \ref{Lem:DimensionStacks} tells us that the stack in the bottom right corner of \eqref{Eq:CartesianTimesTwo} is equidimensional. We know that $\shg$ is also equidimensional and that the map 
\begin{align}
    \shg \to \left[\frac{\ker \gamma \backslash \mlocinf}{\operatorname{Ad}_{\sigma} \lmG} \right]
\end{align}
is weakly perfectly smooth. Thus by Lemma \ref{Lem:DimensionResult} this map must be weakly perfectly smooth of constant relative dimension $M$. Because the diagram in \eqref{Eq:CartesianTimesTwo} is Cartesian, it follows that the natural map
\begin{align}
  \shget \to \left[\frac{\mloceone}{\operatorname{Ad}_{\sigma} H_m} \right]
\end{align}
is also weakly perfectly smooth of constant relative dimension $M$. By Lemma \ref{Lem:DimensionStacks}, both stacks in the rightmost column of \eqref{Eq:CartesianTimesTwo} are equidimensional of the same dimension. We deduce from Lemma \ref{Lem:DimensionResult} that $\shget$ is equidimensional of the same dimension as $\shg$ and thus equidimensional of the same dimension as $\shge$.

After possibly increasing $m$, we may choose $0 \ll m' \ll m$ and invoke Lemma \ref{Lem:OneMoreWeaklyPerfectlySmooth}, which tells us that the natural map
\begin{align}
    \left[\frac{\mloceone}{\operatorname{Ad}_{\sigma} H_m} \right] \to \shtemu^{(m',1)}
\end{align}
is weakly perfectly smooth. It follows from \cite[Proposition 4.2.5]{ShenYuZhang} that the natural map
\begin{align}
  \shtemu^{(m',1)} &\to \left[ \mloce/\overline{\mG}_{\emptyset} \right]
\end{align}
is weakly perfectly smooth. Therefore the map $\hat{\lambda}:\shget \to \left[ \mloce/\overline{\mG}_{\emptyset} \right]$ is a composition of weakly perfectly smooth maps, and hence weakly perfectly smooth. 
\end{proof}
\subsubsection{} For $w \in \admu$, we define the Kottwitz--Rapoport (KR) stratum $\shget(w)$ to be the inverse image of the locally closed substack
\begin{align}
  \left[ \mloce(w)/\overline{\mG}_{\emptyset} \right] \subset \left[ \mloce/\overline{\mG}_{\emptyset} \right]
\end{align}
under the weakly perfectly smooth map $\hat{\lambda}:\shget \to \left[ \mloce/\overline{\mG}_{\emptyset} \right]$. Similarly, we define
\begin{align} \label{Eq:DefinitionKR}
  \shget(\le \! w):=\bigcup_{w' \le w} \shget(w'),
\end{align}
which is the same as the closure of $\shget(w)$ because $\hat{\lambda}$ is open and since the closure relations hold on $\mloce$, see Section \ref{Sec:RelPos}.
\begin{Cor} \label{Cor:Normal}
For $w \in \admu$, the closure $\shget(\le \! w)$ has dimension $\ell(w)$ and is normal. 
\end{Cor}
\begin{proof}
Let $d=\operatorname{Dim} \shge=\operatorname{Dim} \mloce$. Then the local model $\mloce$ is the union of $\mloce(\le \! w)$ for $w \in \admu$ of length $d$, and for such $w$ the KR stratum $\mloce(\le \! w)$ is equidimensional of dimension $d$ and stable under the action of $\overline{\mG}_{\emptyset}$. Using $\hat{\lambda}$, we see that 
\begin{align}
  \shget = \bigcup_{\substack{w \in \admu \\ \ell(w)=d}} \shget(\le \! w),
\end{align}
and since $\shget$ is equidimensional of dimension $d$, it follows that for $w$ with $\ell(w)=d$ we have that $\shget(\le \! w)$ is equidimensional of dimension $d=\ell(w)$. We can now apply Lemma \ref{Lem:DimensionResult} to deduce that $\hat{\lambda}$ is weakly perfectly smooth of relative dimension $0$. We can apply Lemma \ref{Lem:DimensionResult} again to deduce the dimension results for $\shget(\le \! w)$ for arbitrary $w$, from the dimension results for $\mloce(\le \! w)$ from Section \ref{Sec:RelPos}.

The morphism $\shget \to \left[ \mloce/\overline{\mG}_{\emptyset} \right]$ is (by definition) the same as a diagram
\begin{equation}
  \begin{tikzcd}
  & \tshget \arrow[dl,swap,"s"] \arrow{dr}{t} \\
  \shget & & \mloce,
  \end{tikzcd}
\end{equation}
where $s:\tshget \to \shget$ is a $\overline{\mathcal{G}_{\emptyset}}=L^1 \mG_{\emptyset}$-torsor. Since both $s$ and $t$ are surjective and weakly perfectly smooth, the normality of $\shget(\le \! w)$ follows from the normality of $\mloce(\le \! w)$ by Lemma \ref{Lem:WeaklyPerfectlySmoothNormal}. 
\end{proof}
We now give a corollary of Lemma \ref{Lem:finite}.
\begin{Cor} \label{Cor:KRQuasiAffine}
For $w \in \admu$, the KR stratum $\shget(w)$ is quasi-affine.
\end{Cor}
\begin{proof}
Section \ref{Sec:ChangeOfParahoric} and in particular equation \eqref{Eq:ParahoricForgetfulGandSiegel} shows that there is a commutative diagram where all the horizontal maps are finite
\begin{equation} \label{Eq:ParahoricForgetfulGandSiegelII}
  \begin{tikzcd}
   \mathscr{S}_{U'}(G,X) \arrow{d}{\pi_{\emptyset,K}} \arrow{r} & \mathscr{S}_{M^pM_p'}(G_V, \mathcal{H}_V)\otimes \mathcal{O}_{E,(v)} \arrow{d} \arrow{r} &  \mathscr{S}_{J^{'p} J_p'}(G_{V_{\mathcal{L}'}} \mathcal{H}_{V_{\mathcal{L}'}}) \otimes \mathcal{O}_{E,(v)} \\
   \mathscr{S}_{U}(G,X) \arrow{r} & \mathscr{S}_{M^pM_p}(G_V, \mathcal{H}_V)\otimes \mathcal{O}_{E,(v)} \arrow{r} & \mathscr{S}_{J^{p} J_p}(G_{V_{\mathcal{L}}},\mathcal{H}_{V_{\mathcal{L}}}) \otimes \mathcal{O}_{E,(v)}.
  \end{tikzcd}
\end{equation}
Using Zarhin's trick as in \cite[Remark 2.1.4]{ShenYuZhang} or \cite[Section 1.3.3]{KisinPoints}), there is moreover a finite map $\mathscr{S}_{J^{'p} J_p'}(G_{V_{\mathcal{L}'}} \mathcal{H}_{V_{\mathcal{L}'}}) \to \mathscr{S}_{Q^p Q_p}(G_{V''}, \mathcal{H}_{V''})$, where $V''=V_{\mathcal{L'}}^{\oplus 4} \oplus V_{\mathcal{L'}}^{\ast, \oplus 4}$ and where $\psi''$ is given by a certain explicit matrix. Here $Q_p$ corresponds to the self dual lattice $V_{\mathcal{L'},p}^{\oplus 4} \oplus V_{\mathcal{L'},p}^{\ast,\oplus 4}$ and $Q^p \subset G_{V''}(\afp)$ is sufficiently small. By Lemma \ref{Lem:finite}, the pullback $\mathcal{E}$ of the (ample) Hodge bundle from the perfection of $\mathscr{S}_{Q^p Q_p,\ovfp}(G_{V''}, \mathcal{H}_{V''})$ to $\shget$ is ample. \smallskip 

By construction, see Lemma \ref{Lem:CartesianGSp}, the left square in the following diagram commutes 
\begin{equation}
  \begin{tikzcd}
  \shget \arrow{r} \arrow{d} & \shegsp \arrow{d} \arrow{r} & \mathscr{S}_{Q^p Q_p,\ovfp}(G_{V''}, \mathcal{H}_{V''})^{\mathrm{perf}} \arrow{d} \\ 
  \shtemu  \arrow{r} & \operatorname{Sht}_{G_V, \mathcal{P}', \{\mu\}} \arrow{r} & \operatorname{Sht}_{G_{V''},Q_p, \{\mu\}}.
  \end{tikzcd}
\end{equation}
The right square moreover commutes because Zarhin's trick is given by a morphism of Shimura data, and then we can use \cite[Corollary 4.3.2]{PappasRapoportShtukas} as in Lemma \ref{Lem:NotCartesianYet}. \smallskip

The arguments in the proof of \cite[Theorem 3.5.9]{ShenYuZhang} now show that the restriction of $\mathcal{E}$ to the KR stratum $\shget(w)$ for $w \in \admu$ is a torsion ample line bundle from which it follows that $\shget(w)$ is quasi-affine. To elaborate, their arguments show that the Hodge bundle on $\mathscr{S}_{Q^p Q_p,\ovfp}(G_{V''}, \mathcal{H}_{V''})^{\mathrm{perf}}$ comes via pullback from a line bundle $\mathcal{F}$ on $\operatorname{Sht}_{G_{V''},Q_p, \{\mu\}}$. They then show that if we pull back $\mathcal{F}$ to $\shtemu$ and restrict to a KR stratum, that the resulting line bundle is torsion. 
\end{proof}
\subsection{Connected components of closures of KR strata} \label{Sec:KRConnected} The goal of this section is to understand, for $w \in \admu$, the fibers of 
\begin{align*}
\pi_0(\shget(\le \! w)) \to \pi_0(\shg).
\end{align*}
Here $\shget(\le \! w)$ is the closure of the KR stratum $\shget(w)$, see equation \eqref{Eq:DefinitionKR}. We will eventually reduce this to understanding the fibers of
\begin{align*}
\shget(\tau) \to \pi_0(\shg),
\end{align*}
where $\tau \in \admu$ is the unique element of length zero. To make this reduction, we will show that each connected component of $\shget(\le \! w)$ intersects $\shget(\tau)$. This will require us to assume that either $\mathbf{Sh}_{U}(G,X)$ is proper or that $G_{\mathbb{Q}_p}$ is unramified. More generally, we require that Conjecture \ref{Conj:KR} below holds. Recall that there are EKOR strata $\shg\{w\}$ for $w \in \kadmu$, see Section \ref{Sec:EKOR} and Section \ref{Sec:RestrictedShtukasShimuraVarieties}, with closures $\shg\{\preceq w\}$.
\begin{Conj} \label{Conj:KR}
If $V$ is an irreducible component of $\shg\{\preceq \! w \}$ for some $w \in \kadmu$, then $V$ intersects the unique $0$-dimensional EKOR stratum $\shg\{\tau\}$.
\end{Conj}
\begin{Rem} \label{Rem:ConjKRUnramified}
When $\mathcal{G}_K$ is hyperspecial, then Conjecture \ref{Conj:KR} is \cite[Proposition 6.20]{WedhornZiegler}; the assumption made in the statement of this proposition is proved in \cite{Andreatta}. When $G^{\mathrm{ad}}$ is $\mathbb{Q}$-simple, a proof of the conjecture will appear in the forthcoming PhD thesis of Shengkai Mao, see \cite{MaoCompact}. When $\mathbf{Sh}_U(G,X)$ is proper, we will circumvent the conjecture using Lemma \ref{Lem:RaynaudArgument} below. This is where the 'either unramified or proper' assumption in Theorems \ref{Thm:Uniformisation}, \ref{Thm:HeRapoport} and \ref{Thm:EKOR} comes from.
\end{Rem}

\subsubsection{} We start by proving a lemma, where we recall that $\tau \in \admu$ is the unique element of length zero.
\begin{Lem} \label{Lem:RaynaudArgument}
Let $Z$ be a connected component of $\shget(\le \! w)$. Suppose that there exists a KR stratum $\shget(x)$ such that $Z \cap \shget(x)$ is non-empty and such that $\shget(\le \! x)$ is perfectly proper over $\spec k$. Then $Z$ intersects $\shget(\tau)$. 
\end{Lem}
\begin{proof}
Let $\shget(x)$ be as in the statement of the lemma. Then there is an $x' \le x$ of minimal length such that $\shget(x') \cap Z \not= \emptyset$, and it suffices to prove that this length is equal to zero. The minimality tells us that
\begin{align} \label{Eq:proper}
  \shget(x') \cap Z = \shget(\le \! x') \cap Z,
\end{align}
since $\shget(\le \! x') \setminus \shget(x')$ is a union of KR strata associated to $x'' \in \admu$ of length strictly smaller than $x'$. Next, we note that $Z \cap \shget(x')$ is a union of connected components of $\shget(x')$, because $\shget(x') \subset \shget(\le \! w)$ and so connected components of $\shget(x')$ are either disjoint from $Z$ or contained in $Z$. 

Since $\shget(x')$ is quasi-affine by Corollary \ref{Cor:KRQuasiAffine}, we find that $\shget(x') \cap Z$ is quasi-affine. Moreover \eqref{Eq:proper} implies that $\shget(x') \cap Z \subset \shget(\le \! x)$ is closed, hence perfectly proper over $\spec k$. Therefore $\shget(x') \cap Z$ is perfectly proper and quasi-affine, and thus zero-dimensional. Since it is a union of connected components of $\shget(x')$, it follows from Corollary \ref{Cor:Normal} that $x'$ has length zero and must therefore be equal to $\tau$.
\end{proof}
We will deduce the same result from Conjecture \ref{Conj:KR} when the Shimura variety is not proper. 
\begin{Prop} \label{Prop:KRZero}
If Conjecture \ref{Conj:KR} holds, then for $w \in \admu$ every connected component $Z$ of $\shget(\le \! w)$ intersects $\shget(\tau)$. 
\end{Prop} 
First, we prove two lemmas. Recall from \cite[Section 1.3]{HZ} that an element $w\in \tilde W$ is said to be $\sigma$-straight if $$n\ell(w)=\ell(w\sigma(w)\dotsc\sigma^{n-1}(w))$$ for all positive integers $n$. 
\begin{Lem} \label{Lem:SigmaStraight}
Let  $Z \subset \shget(\le \! w)$ be a connected component. If $x\in \admu$ is of minimal length with the property that $Z \cap \shget(x) \not=\emptyset$, then $x$ is $\sigma$-straight.
\end{Lem}
\begin{proof}
Arguing as in the proof of Lemma \ref{Lem:RaynaudArgument} above, we see that the intersection $Z \cap \shget(x)$ is a union of connected components of $\shget(x)$. Let $V$ be one of these components, then $V$ is closed in $\shget(\le \! x)$ as in the proof of Lemma \ref{Lem:RaynaudArgument}. Moreover, $V$ is actually a connected component of $\shget(\le \! x)$; it is an irreducible component for dimension reasons and thus a connected component since $\shget(\le \! x)$ is normal (see Corollary \ref{Cor:Normal}). 

Let $z \in V(\ovfp)$ with image $\pi(z) \in \shg(\ovfp)$, and consider the uniformisation map
\begin{align}
  i_{\pi(z)}:\xmub(\ovfp) \to \shg(\ovfp),
\end{align}
centered at $\pi(z)$, where $b$ corresponds to $\pi(z)$. By the proof of \cite[Proposition 5.2.2]{HeZhouZhu} we can upgrade this to a morphism of perfect schemes $i_{\pi(z)}:\xmub \to \shg$. As in the proof of Theorem \ref{Thm:LiftingUniformisation}, see the discussion in Section \ref{Sec:BasicUniformisation} below, it follows that there is an induced map
\begin{align}
  i_{z}:\xmube \to \shget
\end{align}
whose image contains $z$. Indeed, this follows from the construction of $i_z$ below and the surjectivity of $\shget \to \shg$. Since the uniformisation map is compatible with the KR stratification, this restricts to a map
\begin{align}
    \xmube(\le \! x) \to \shget(\le \! x)
\end{align}
whose image contains $z$. This means that there is a connected component $Y$ of $\xmube(\le \! x)$ that maps to $V$. Now \cite[Theorem 4.1]{HZ} tells us that there is a $\sigma$-straight element $x' \le x$ in $\admu$ such that $Y \cap \xmube(x') \not=\emptyset$. In particular, $\shget(x') \cap V \not= \emptyset$ and so $\shget(x') \cap Z \not=\emptyset$. Since $x$ has been chosen to be minimal with the property that $\shget(x') \cap Z \not=\emptyset$, we see that $x=x'$ and so $x$ is $\sigma$-straight.
\end{proof}
\begin{Lem} \label{Lem:FiniteEtale}
Let $x \in \admu$ be $\sigma$-straight. Then there is $y \in \kadmu$ such that the natural map $\shget(x) \to \shg$ factors via a finite \'etale map $\shget(x) \to \shg\{y\}$ and such that $\ell(y)=\ell(x)$.
\end{Lem}
\begin{proof}
By the proof of \cite[Theorem 6.17]{He-Rapoport}, there is an element $v \in W_K$ such that $y:= v x \sigma(v)^{-1}$ lies in $\kadmu$ and such that $\ell(y)=\ell(x)$. It follows from \cite[the discussion prior to Theorem 6.10]{He-Rapoport} that the image of $\shtgemu(x)(\ovfp)$ in $\sht(\ovfp)$ is equal to $\shtgmu\{y\}(\ovfp)$.\footnote{Recall that for $y \in \kadmu$ we use $\{y\}$ to denote the corresponding EKOR stratum, see Section \ref{Sec:EKOR}.} Since KR strata and EKOR strata on $\shget$ and $\shg$ respectively are defined as the inverse images of KR strata and EKOR strata in $\shtemu$ and $\shtmu$, and because these strata are determined by their $\ovfp$-points, we deduce that the image of $\shget(x) \to \shg$ is equal to $\shg\{y\}$.

To prove that the induced map is finite \'etale, we may use diagram \eqref{Eq:CartesianDiagramDef} to reduce to checking finite \'etale-ness of $\shtgemu(x) \to \shtgmu\{y\}$, and then Lemma \ref{Lem:LocalUniformisation} to reduce to checking this for $\xmube(x) \to \xmub\{y\}$. By \cite[Proposition 4.6]{HZ}, the locally perfectly of finite type perfect schemes $\xmube(x)$ and $\xmub\{y\}$ are zero-dimensional. Thus they have an affine open cover by zero-dimensional perfectly of finite type affine perfect schemes, which must be finite disjoint unions of $\spec \ovfp$ (since zero-dimensional reduced finite type affine schemes over $\spec \ovfp$ are). This implies that both $\xmube(x)$ and $\xmub\{y\}$ are disjoint unions of $\spec \ovfp$, which in particular implies that they are \'etale over $\spec \ovfp$. Thus the map $\xmube(x) \to \xmub\{y\}$ is \'etale and it suffices to show that it is finite \'etale, which comes down to showing it is quasi-finite. For thus, we note that $J_b(\qp)$ acts transitively on $\xmube(x)(\ovfp)$ by \cite[Theorem 5.1]{HZ}, with stabiliser a compact open subgroup, cf. \cite[Proposition 3.1.4]{ZhouZhu}, and the same holds for $\xmub\{y\}(\ovfp)$. Thus we may identify $\xmube(x) \to \xmub\{y\}$ with
\begin{align}
    \coprod_{J_b(\qp)/N} \spec \ovfp \to \coprod_{J_b(\qp)/N'} \spec \ovfp,
\end{align}
where $N \subset N'$ are compact open subgroups of $J_b(\qp)$. Since $N$ has finite index in $N'$, it follows that $\xmube(x) \to \xmub\{y\}$ is finite \'etale. 
\end{proof}
\begin{proof}[Proof of Proposition \ref{Prop:KRZero}]
Let $x \in \admu$ be of minimal length with the property that $Z \cap \shget(x) \not=\emptyset$, we would like to show that $\ell(x)=0$. Arguing as in the proof of Lemma \ref{Lem:RaynaudArgument} above, we see that the intersection $Z \cap \shget(x)$ is a union of connected components of $\shget(x)$ and that $Z \cap \shget(x)$ is closed in $\shget(\le \! x)$. Let $V$ be one of these components, then $V$ has dimension $\ell(x)$ and $V$ is closed inside $\shget(\le \! x)$. Thus $V$ must be an irreducible component of $\shget(\le \! x)$.

By Lemma \ref{Lem:SigmaStraight} we see that $x$ is $\sigma$-straight. By Lemma \ref{Lem:FiniteEtale} there exists $y \in \kadmu$ such that the natural map $\shget(w) \to \shg$ factors via a finite \'etale map $\shget(w) \to \shg\{y\}$ and such that $\ell(x)=\ell(y)$. We conclude that the image of $V$ in $\shg\{y\}$ is an irreducible component of $\shg\{y\}$. Since $V$ is closed in $\shget(\le \! x)$ and thus in $\shget$, and since the map $\shget \to \shg$ is perfectly proper, it follows that $\pi(V)$ is closed in $\shg$. Therefore $\pi(V)$ is closed inside $\shg\{\preceq y\}$, the closure of $\shg\{y\}$, and therefore an irreducible component of $\shg\{\preceq y\}$.

Conjecture \ref{Conj:KR} tells us that $\pi(V)$ intersects the zero-dimensional EKOR stratum $\shg\{\tau\}$, and since $\pi(V) \subset \shg\{y\}$ it follows that $\tau=y$ and so that $0=\ell(y)=\ell(x)$. It follows that $x=\tau$ and so we are done.
\end{proof}
\subsubsection{} \label{Sec:BasicUniformisation} We will explicitly analyse the basic KR stratum $\shget(\tau)$, where $\tau \in \admu$ is the unique element of length zero. Let $x \in \shgetinf(\tau)(\ovfp)$ with image $\pi(x) \in \shg(\ovfp)$ and choose an isomorphism $\mathbb{D}_x \simeq V_{\mathbb{Z}_p} \otimes_{\zp} \zpbr$ sending $s_{\alpha, 0,x}$ to $s_{\alpha} \otimes 1$. Let $b \in G(\qpbr)$ be the element corresponding to the Frobenius of $\mathbb{D}_x$ under this isomorphism. Let $I_x$ be the algebraic group $I_{\pi(x)}$ introduced in Section \ref{Sec:IsogenyClasses} and let
\begin{align} \label{Eq:InducedMorphisms}
    j_{x}^p:I_{x, \afp} &\to G_{\afp} \\
    j_{x,p}:I_{x, \qp} &\to J_b
\end{align}
be the maps induced by the choices made above. Then by \cite[Proposition 5.2.2]{HeZhouZhu}, there is an isomorphism of perfect schemes (where $\shgb \subset \shg$ denotes the Newton stratum associated to $[b]$)
\begin{align}
  i_{\pi(x)}:I_x(\mathbb{Q}) \backslash \left( \xmub \times G(\afp) / U^p \to \shgb \right),
\end{align}
where $I_x(\mathbb{Q})$ acts on $G(\afp)$ via $j_{x}^p$ and on $\xmub$ via $j_{x,p}:I(\mathbb{Q}) \to J_b(\qp)$ and then the natural action of $J_b(\qp)$ on $\xmub$. Here we consider the discrete topological space $G(\afp) / U^p$ as a discrete scheme, and we are taking the quotient of $\xmub \times G(\afp) / U^p$ by $I_x(\mathbb{Q})$ in the pro-\'etale topology. Moreover, it follows from \cite[Proposition 5.2.6]{HeZhouZhu} that $j_{x}^p$ and $j_{x,p}$ are isomorphisms and that $I(\mathbb{R})$ is compact mod centre. 

\subsubsection{} Consider the Cartesian diagram \newcommand{\shtmueb}{\operatorname{Sht}_{G,\emptyset, \mu,[b]}}
\begin{equation} \label{Eq:DiagramNotYetCartesianBasic}
  \begin{tikzcd}
   \shgetb \arrow{d} \arrow{r} & \shtmueb \arrow{d} \\
  \shgb \arrow{r} & \shtmub,
  \end{tikzcd}
\end{equation}
Applying Lemma \ref{Lem:LocalUniformisation} to $\shtmueb$ and $\shtmub$ and using $i_x$ we can identify \eqref{Eq:DiagramNotYetCartesianBasic} with
\begin{equation}
  \begin{tikzcd}
  \shgetb \arrow{d} \arrow{r} & \left[\frac{\xmube}{\underline{J_b(\mathbb{Q}_p)}} \right] \arrow{d} \\
  I_x(\mathbb{Q}) \backslash \xmub \times G(\afp) / U^p \arrow{r} & \left[\frac{\xmub}{\underline{J_b(\mathbb{Q}_p)}} \right].
  \end{tikzcd}
\end{equation}
By Lemma \ref{Lem:CompatibilityUniformisationJointStratification} the map (induced by the bottom horizontal map)
\begin{align}
  \xmub \times G(\afp) \to\left[\frac{\xmub}{\underline{J_b(\mathbb{Q}_p)}} \right]
\end{align}
is the natural projection map onto the first factor followed by the natural map to the quotient. As in the proof of Theorem \ref{Thm:LiftingUniformisation}, it follows that there is an isomorphism
\begin{align} \label{eq:BasicShget}
i_x:I_x(\mathbb{Q}) \backslash \xmube \times G(\afp) / U^p \to \shgetb
\end{align}
such that the map (coming from the left vertical map in \eqref{Eq:DiagramNotYetCartesianBasic})
\begin{align}
    I_x(\mathbb{Q}) \backslash \xmube \times G(\afp) / U^p \to I_x(\mathbb{Q}) \backslash \xmub \times G(\afp) / U^p,
\end{align}
is induced by the natural projection $\xmube \to \xmub$ and the identity of $G(\afp)$. 

\subsubsection{} To analyse the fibers of $\shget(\tau) \to \pi_0(\shg)$, we will first analyse the fibers of $\xmube(\tau) \to \pi_1(G)_I^{\sigma}$. Let $J_b^{\mathrm{sc}} \to J_b^{\mathrm{der}}$ be the simply connected cover of the derived group $J_b^{\mathrm{der}}$ of $J_b$.
\begin{Lem} \label{Lem:LocalTransitiveAction}
The group $J_b^{\mathrm{sc}}(\mathbb{Q}_p)$ acts transitively on the fibers of
\begin{align*}
\xmube(\tau) \to \pi_1(G)_I^{\sigma}.
\end{align*}
\end{Lem}
\begin{proof}
The element $\tau$ is $\sigma$-straight and so $J_b(\qp)$ acts transitively on $\xmube(\tau)$ by \cite[Theorem 4.8]{He1}. The stabiliser of a point is a parahoric subgroup $N_p \subset J_b(\mathbb{Q}_p)$ by \cite[Proposition 3.1.4]{ZhouZhu}. Therefore our map can be identified with the natural map
\begin{align}
  \xmube(\tau)=\frac{J_b(\mathbb{Q}_p)}{N_p} \to \frac{J_b(\mathbb{Q}_p)}{N_p J^{\mathrm{sc}}(\mathbb{Q}_p)}= \pi_1(J_b)_I^{\sigma} = \pi_1(G)_I^{\sigma},
\end{align}
using Lemma \ref{Lem:Fundamentalgroup} and the fact that $b$ is basic in the last step, and the result follows.
\end{proof}
\subsubsection{} \label{Sec:Borovoi} The goal of this subsection is to prove an auxiliary result. Let $\mathbb{G}$ and $\mathbb{H}$ be connected reductive groups over $\mathbb{Q}$ that are inner forms of each other, and such that they are isomorphic over $\afp$. Fix an identification $\mathbb{G} \otimes \afp \simeq \mathbb{H} \otimes \afp$ and an inner twisting $\Psi:\mathbb{G}_{\qbar} \to \mathbb{H}_{\qbar}$, which induces an isomorphism of centers $Z(\mathbb{G}) \to Z(\mathbb{H})$ and $\pi_1(\mathbb{G})_I^{\sigma} \simeq \pi_1(\mathbb{H})_I^{\sigma}$. Recall the notation $\mathbb{G}(\mathbb{R})_{+}$ and $\mathbb{G}(\mathbb{Q})_{+}$ from Section \ref{Sec:UniformisationAndPiNaught}.
\begin{Prop}[Borovoi] \label{Prop:Borovoi}
The images of $\mathbb{G}(\mathbb{Q})_{+}$ and $\mathbb{H}(\mathbb{Q})_{+}$ in
\begin{align}
  \frac{\mathbb{G}(\afp)}{\rho(\mathbb{G}^{\mathrm{sc}}(\afp))} \times \pi_1(\mathbb{G})_I^{\sigma}
\end{align}
are equal (after applying our fixed identifications).
\end{Prop}
The following arguments have been reproduced and adapted with permission from Mikhail Borovoi's Mathoverflow answer \cite{BorovoiMO}; we will use \cite[Section 3]{Borovoi}. We consider the \emph{crossed module} $(\mathbb{G}^{\rm sc}\to \mathbb{G})$ and the \emph{hypercohomology}
$$H^0_{\rm ab}({\mathbb Q},\mathbb{G}):=H^0({\mathbb Q},\mathbb{G}^{\rm sc}\to \mathbb{G}),$$
where $\mathbb{G}$ is in degree 0;
see \cite{Borovoi}. The cohomology set $H^0_{\rm ab}({\mathbb Q},\mathbb{G})$ is naturally an abelian group that does not change under inner twisting of $\mathbb{G}$. The short exact sequence
$$1\to(1\to \mathbb{G})\to (\mathbb{G}^{\rm sc}\to \mathbb{G})\to (\mathbb{G}^{\rm sc}\to 1)\to 1$$
induces a hypercohomology exact sequence
$$ \mathbb{G}^{\rm sc}({\mathbb Q})\to \mathbb{G}({\mathbb Q})\to H^0_{\rm ab}({\mathbb Q},\mathbb{G})\to H^1({\mathbb Q},\mathbb{G}^{\rm sc}),$$
where 
$${\rm ab}^0\colon \mathbb{G}({\mathbb Q})\to H^0_{\rm ab}({\mathbb Q},\mathbb{G})$$
is the \emph{abelianisation map}. Let $Z$ be the center of $\mathbb{G}$, then it follows from the definition of $\mathbb{G}({\mathbb R})_+$ and the connectedness of $\mathbb{G}^{\rm sc}({\mathbb R})$, that $$\mathbb{G}({\mathbb R})_+=Z({\mathbb R})\cdot \rho(\mathbb{G}^{\rm sc}({\mathbb R})),$$ and hence
$$\mathbb{G}({\mathbb R})_+/\rho (\mathbb{G}^{\rm sc}({\mathbb R}))={\rm ab}^0(Z({\mathbb R}))\subset
{\rm ker}[ H^0_{\rm ab}({\mathbb R},\mathbb{G})\to H^1({\mathbb R}, \mathbb{G}^{\rm sc})].$$
We see that the image of $\mathbb{G}(\mathbb{Q})_{+}$ in $H^0_{\rm ab}({\mathbb Q},\mathbb{G})$ can be identified with the preimage of ${\rm ab}^0(Z({\mathbb R}))\subset H^0_{\rm ab}({\mathbb R},\mathbb{G})$ in ${\rm ker}[H^0_{\rm ab}({\mathbb Q},\mathbb{G})\to H^1({\mathbb Q},\mathbb{G}^{\rm sc})]$ under the natural map 
\begin{align}
  f:{\rm ker}[H^0_{\rm ab}({\mathbb Q},\mathbb{G})\to H^1({\mathbb Q},\mathbb{G}^{\rm sc})] \to {\rm ker}[ H^0_{\rm ab}({\mathbb R},\mathbb{G})\to H^1({\mathbb R}, \mathbb{G}^{\rm sc})].
\end{align}

\begin{Lem} The preimage of ${\rm ab}^0(Z({\mathbb R}))\subset H^0_{\rm ab}({\mathbb R},\mathbb{G})$ in ${\rm ker}[H^0_{\rm ab}({\mathbb Q},\mathbb{G})\to H^1({\mathbb Q},\mathbb{G}^{\rm sc})]$ under $f$
coincides with the preimage of ${\rm ab}^0(Z({\mathbb R}))$ in $H^0_{\rm ab}({\mathbb Q},\mathbb{G})$.
\end{Lem}
\begin{proof}
Let $\xi\in H^0_{\rm ab}({\mathbb Q},\mathbb{G})$ lie in the preimage of
$${\rm ab}^0(Z({\mathbb R}))\subset
{\rm ker}[ H^0_{\rm ab}({\mathbb R},\mathbb{G}) \to H^1({\mathbb R}, \mathbb{G}^{\rm sc})].$$
Then the image of $\xi$ in $H^1({\mathbb R},\mathbb{G}^{\rm sc})$ is trivial, and therefore,
the image of $\xi$ in $H^1({\mathbb Q},\mathbb{G}^{\rm sc})$ lies in the kernel of the localisation map $$ H^1({\mathbb Q}, \mathbb{G}^{\rm sc})\to H^1({\mathbb R},\mathbb{G}^{\rm sc}).$$
By the Hasse principle for simply connected groups (\cite[Theorem 6.6]{PR}), this kernel is trivial.
Thus the image of $\xi$ in $H^1({\mathbb Q},\mathbb{G}^{\rm sc})$ is trivial, and hence $\xi$ lies in the preimage of ${\rm ab}^0(Z({\mathbb R}))$ in ${\rm ker}[H^0_{\rm ab}({\mathbb Q},\mathbb{G})\to H^1({\mathbb Q},\mathbb{G}^{\rm sc})]$, as required. \end{proof}

\begin{Cor} \label{Cor:Borovoi} The image of the abelianisation map $\mathbb{G} ({\mathbb Q})_{+}\to H^0_{\rm ab}({\mathbb Q},\mathbb{G})$ is the preimage of ${\rm ab}^0(Z({\mathbb R}))\subset H^0_{\rm ab}({\mathbb R},\mathbb{G} )$ in $H^0_{\rm ab}({\mathbb Q},\mathbb{G} )$. 
\end{Cor}
\begin{proof}[Proof of Proposition \ref{Prop:Borovoi}]
It is clear from Corollary \ref{Cor:Borovoi} and the discussion above that the image of $\mathbb{G}({\mathbb Q})_{+} \to H^0_{\rm ab}({\mathbb Q},\mathbb{G})$ is the same for all inner forms. Thus the images of $\mathbb{H}(\mathbb{Q})_{+}$ and $\mathbb{G}(\mathbb{Q})_{+}$ in $H^0_{\rm ab}({\mathbb Q},\mathbb{G})=H^0_{\rm ab}({\mathbb Q},\mathbb{H})$ are equal. 

To prove the proposition, we simply note that the following diagram commutes
\begin{equation}
  \begin{tikzcd}
  \frac{\mathbb{G}(\mathbb{Q})_{+}}{\rho(\mathbb{G}^{\mathrm{sc}}(\mathbb{Q}))} \arrow[r, hook] \arrow{d} & H^0_{\rm ab}({\mathbb Q},\mathbb{G}) \arrow{d} \\
  \frac{\mathbb{G}(\af)}{\rho(\mathbb{G}^{\mathrm{sc}}(\af))} \arrow[r, hook] & \prod_{v \not=\infty} H^0_{\rm ab}({\mathbb Q}_v,\mathbb{G}).
  \end{tikzcd}
\end{equation}
and that $\pi_1(\mathbb{G})_I^{\sigma}$ is a quotient of $\mathbb{G}(\qp)/\rho (\mathbb{G}^{\mathrm{sc}}(\qp))$ by Lemma \ref{Lem:Fundamentalgroup}.
\end{proof}

\begin{Prop} \label{Prop:GlobalActionTransitive}
Let $\Sigma$ be a finite set of primes with $p \in \Sigma$. Then $\gsc(\afs)$ acts transitively on the fibers of 
\begin{align*}
\shgetinf(\tau) \to \pi_0(\shginf).
\end{align*}
\end{Prop}
\begin{proof}
Let $z \in \shgetinf(\tau)(\ovfp)$, where $\shgetinf(\tau):=\varprojlim_{U^p} \shget(\tau)$, with image $x \in \shginf(\ovfp)$. Choose an isomorphism $\mathbb{D}_x \simeq V_{\mathbb{Z}_p} \otimes_{\zp} \zpbr$ sending $s_{\alpha, 0,x}$ to $s_{\alpha} \otimes 1$ and let $b \in G(\qpbr)$ be the element corresponding to the Frobenius of $\mathbb{D}_x$ under this isomorphism. Then as explained in Section \ref{Sec:BasicUniformisation}, we get an isomorphism
\begin{align}
    i_z:\varprojlim_{U^p} I_x(\mathbb{Q}) \backslash \xmube\times G(\afp)/U^p \to \varprojlim_{U^p} \shgetb=:\shgetinfb.
\end{align}
Since the uniformisation is compatible with the KR stratifications on both sides, see Lemma \ref{Lem:CompatibilityUniformisationJointStratification}, this induces an isomorphism
\begin{align}
    \varprojlim_{U^p} I_x(\mathbb{Q}) \backslash \xmube(\tau) \times G(\afp)/U^p \to \varprojlim_{U^p} \shget(\tau)=\shgetinf(\tau).
\end{align}
We also note that the natural map $I_x(\mathbb{Q}) \backslash \xmube(\tau) \times G(\afp) \to \varprojlim_{U^p} I_x(\mathbb{Q}) \backslash \xmube(\tau) \times G(\afp)/U^p$ is a bijection by \cite[Lemma 4.20]{Milne}, as in the second paragraph of Section \ref{Sec:UniformisationAndPiNaught}.

Using the base point $x \in \shginf(\ovfp)$ to trivialise the $\pi(G)$-torsor $\pi_0(\shginf)$, see the beginning of Section \ref{Sec:UniformisationAndPiNaught}, we get an isomorphism of profinite sets
\begin{align}
    \pi(G) &\to \pi_0(\shginf) \\
    g &\mapsto g \cdot Z_{x},
\end{align}
where $Z_{x}$ is the connected component containing $x$. By the discussion in Section \ref{Sec:StrongApproximation} there is an isomorphism of topological groups
\begin{align}
    \pi(G) = G(\mathbb{Q})_{+} \backslash \pi_1(G)_I^{\sigma} \times \frac{G(\afp)}{\rho(\gsc(\afp))}.
\end{align}
By Proposition \ref{Prop:UniformisationAndConnectedComponents} the map
\begin{align}
    \alpha_x:I_x(\mathbb{Q}) \backslash \pi_0(\xmube) \times G(\afp) \to \pi_0(\shginf),
\end{align}
induced by $i_x$, satisfies $\alpha_x(y, g^p) = (\kappa(y), g^p) \cdot Z_{x}$, where $\kappa(y) \in \pi_1(G)_I^{\sigma}$ is the image of $y$ and $g^p \in G(\afp)$. Hence our identifications fit in a commutative diagram
\begin{equation}
\begin{tikzcd} \label{Eq:ZeroDimensionalKRStratum}
\pi_0(\shgetinf(\tau)) \arrow{d} & \arrow[l, "\sim"] I_x(\mathbb{Q}) \backslash \pi_0(\xmube(\tau)) \times G(\afp) \arrow{d} \\
\pi_0(\shginf)  & \arrow[l, "\sim"] G(\mathbb{Q})_{+} / \pi_1(G_{\qp})_I^{\sigma} \times \frac{G(\afp)}{\rho(\gsc(\afp))},
\end{tikzcd}
\end{equation}
where the map 
\begin{align}
  \xmube(\tau) \times G(\afp) \to \pi_1(G_{\qp})_I^{\sigma} \times G(\afp)
\end{align}
is the product of the natural map $\kappa$ of Section \ref{Sec:UniformisationAndPiNaught} and the identity map on $G(\afp)$. By \cite[Theorem 7.8]{PR}, which is a strong approximation result, the group $I^{\mathrm{sc}}(\mathbb{Q})$ is dense in (using $j_{x,p}$ and $j_{x}^p$ from Section \ref{Sec:BasicUniformisation} to make the identification)
\begin{align}
  \prod_{\ell \in \Sigma} I^{\mathrm{sc}}(\mathbb{Q}_{\ell}) = J_b^{\mathrm{sc}}(\mathbb{Q}_p) \times \prod_{\ell \in \Sigma \setminus \{p\}} \gsc(\mathbb{Q}_{\ell}).
\end{align}
Recall that we sometimes write $\gsc(\afp) \subset G(\afp)$ for $\rho(\gsc(\afp)) \subset G(\afp)$. Using the discussion above, we can identify the right vertical map in \eqref{Eq:ZeroDimensionalKRStratum} with the natural map\footnote{Note that $\pi(G)$ is Hausdorff, which follows from the discussion in the second paragraph of Section \ref{Sec:UniformisationAndPiNaught}. Therefore the kernel of $G(\afp) \times J_b(\qp) \to G(\afp) \times \pi_1(G_{\qp})_I^{\sigma} \to \pi(G)$ is closed and thus contains the closure of $I_x(\mathbb{Q})$.}
\begin{align}
  I_x(\mathbb{Q}) \backslash \frac{\pi_0(\xmube(\tau))}{J_b^{\mathrm{sc}}(\qp)} \times \frac{G(\afp)}{\prod_{\ell \in \Sigma \setminus \{p\}} \gsc(\mathbb{Q}_{\ell})} \to G(\mathbb{Q})_{+} / \left( \pi_1(G_{\qp})_I^{\sigma} \times \frac{G(\afp)}{\gsc(\afp)}\right).
\end{align}
Lemma \ref{Lem:LocalTransitiveAction} tells us that $\kappa$ induces an isomorphism $\tfrac{\pi_0(\xmube(\tau))}{J_b^{\mathrm{sc}}(\qp)} \to \pi_1(G_{\qp})_I^{\sigma}$ and thus we get 
\begin{align} \label{Eq:MorphismOfAbelianGroups}
  I_x(\mathbb{Q}) \backslash \pi_1(G_{\qp})_I^{\sigma} \times \frac{G(\afp)}{\prod_{\ell \in \Sigma \setminus \{p\}} \gsc(\mathbb{Q}_{\ell})} \to G(\mathbb{Q})_{+} / \left( \pi_1(G_{\qp})_I^{\sigma} \times \frac{G(\afp)}{\gsc(\afp)}\right).
\end{align}
The fibers of the natural map
\begin{align} \label{Eq:InfiniteLevel}
  \pi_1(G_{\qp})_I^{\sigma} \times \frac{G(\afp)}{\prod_{\ell \in \Sigma \setminus \{p\}} \gsc(\mathbb{Q}_{\ell})} \to \pi_1(G_{\qp})_I^{\sigma} \times \frac{G(\afp)}{\gsc(\afp)}
\end{align}
clearly have a transitive action of $\gsc(\afs)$. To show that the same is true for the fibers of \eqref{Eq:MorphismOfAbelianGroups}, we need to show that the images of the two natural maps
\begin{align} \label{Eq:InclusionQuotients}
  I_x(\mathbb{Q}), G(\mathbb{Q})_{+} \to \pi_1(G_{\qp})_I^{\sigma} \times \frac{G(\afp)}{\gsc(\afp)}.
\end{align}
are equal. Now note that $I_x(\mathbb{Q})=I_x(\mathbb{Q})_{+}$ because $I_x^{\mathrm{ad}}(\mathbb{R})$ is compact and thus connected, see \cite[Cor. 1 on page 121]{PR}. Then the required identification of the images of \eqref{Eq:InclusionQuotients} is exactly what is proved in Proposition \ref{Prop:Borovoi}.
\end{proof}
\begin{Prop} \label{Prop:GlobalActionTransitiveKR}
Let $\Sigma$ be a finite set of primes with $p \in \Sigma$. If either $\mathbf{Sh}_{U}(G,X)$ is proper or Conjecture \ref{Conj:KR} holds, then $\gsc(\afs)$ acts transitively on the fibers of
\begin{align}
  \pi_0(\shgetinf(\le \! w)) \to \pi_0(\shginf).
\end{align}
\end{Prop}
\begin{proof}
There is a $G(\afp)$-equivariant commutative diagram
\begin{equation}
  \begin{tikzcd} \label{Eq:PiNaughtDiagram}
  \shgetinf(\tau) \arrow{dr} \arrow[rr] & & \pi_0(\shgetinf(\le \! w)) \arrow{dl} \\
  & \pi_0(\shginf).
  \end{tikzcd}
\end{equation}
If Conjecture \ref{Conj:KR} holds, then by Proposition \ref{Prop:KRZero} every connected component of $\shget(\le \! w)$ intersects $\shget(\tau)$. If $\mathbf{Sh}_{U}(G,X)$ is proper, then $\mathscr{S}_{U}(G,X)$ is proper by the main result of \cite{MP}. Therefore $\shg$ is perfectly proper and moreover $\shget$ is perfectly proper since $\shget \to \shg$ is perfectly proper. Now Lemma \ref{Lem:RaynaudArgument} tells us that every connected component of $\shget(\le \! w)$ intersects $\shget(\tau)$. Thus under the assumptions of the proposition the horizontal arrow in \eqref{Eq:PiNaughtDiagram} is surjective. Indeed, it is a continuous morphism of profinite sets that is a countable inverse limit of surjective maps between finite sets.

We see that the fibers of the left diagonal map surject onto the fibers of the right diagonal map. Now $\gsc(\afs)$ acts transitively on the fibers of the left diagonal map by Proposition \ref{Prop:GlobalActionTransitive}, and therefore also on the fibers of the right diagonal map.
\end{proof}

\subsection{Proof of the main theorems}
\begin{Thm} \label{Thm:DiagramCartesian}
If either $\mathbf{Sh}_{U}(G,X)$ is proper or Conjecture \ref{Conj:KR} holds, then the natural map $\iota:\shge \to \shget$ is an isomorphism.
\end{Thm}
\begin{proof}
We know that $\iota$ is a closed immersion by Proposition \ref{Prop:ClosedImmersion}, whose source and target are equidimensional of the same dimension by Proposition \ref{Prop:perfectlysmooth}. To prove that this closed immersion is an isomorphism, it suffices to show that for each $w \in \admu$ of maximal length, the closed immersion
\begin{align}
\shge(\le \! w) \to \shget(\le \! w)
\end{align}
is an isomorphism. Now source and target are normal by Corollary \ref{Cor:Normal}, and so the source is a union of connected components of the target. To show that the inclusion
\begin{align}
  \pi_0(\shge(\le \! w)) \to \pi_0(\shget(\le \! w))
\end{align}
is an isomorphism, we will use the $G(\afp)$-equivariance of the map $\pi_0(\shgeinf(\le \! \! w)) \to \pi_0(\shgetinf(\le  \! w))$. We know by Corollary \ref{Cor:NonEmptiness} that
\begin{align}
\shge(\le \! w) \to \pi_0(\shg)
\end{align}
is surjective for all $U^p$, and therefore it is enough to show that $\gsc(\afp)$ acts transitively on the fibers of $\shgetinf(\le \! w) \to \pi_0(\shginf)$. Under our assumptions, this follows from Proposition \ref{Prop:GlobalActionTransitiveKR}.
\end{proof}
\subsubsection{Proofs of the main theorems} In this section we deduce the main theorems of the introduction.
\begin{proof}[Proof of Theorem \ref{Thm:Uniformisation}]
Recall that we assumed in Theorem \ref{Thm:Uniformisation} that $\pi_1(G)_I$ is torsion free, which implies that all parahoric subgroups of $G(\qp)$ are connected by Lemma \ref{Lem:TorsionFreeParahoricConnected}. Part 1 of Theorem \ref{Thm:Uniformisation} is Theorem \ref{Thm:VerySpecialUniformisationAndLifting}.

To prove part 2 of the theorem, we let $U_p''$ be an arbitrary parahoric subgroup. We choose an Iwahori subgroup $U_p' \subset U_p''$ and a very special parahoric $U_p \supset U_p'$, this is possible as explained in Section \ref{Sec:VerySpecialParahoric}. The result for $\scrs_{U'}(G,X)$ now follows from Theorem \ref{Thm:LiftingUniformisation} in combination with Theorem \ref{Thm:DiagramCartesian}. Here to apply Theorem \ref{Thm:DiagramCartesian} we need to verify that either $\mathbf{Sh}_{U}(G,X)$ is proper or Conjecture \ref{Conj:KR} holds. In the statement of Theorem \ref{Thm:Uniformisation} we are assuming that either $\mathbf{Sh}_{U}(G,X)$ is proper or that $G_{\qp}$ is unramified. Now we recall that Conjecture \ref{Conj:KR} holds if $G_{\qp}$ is unramified by \cite[Proposition 6.20]{WedhornZiegler} and the main result of \cite{Andreatta}, see Remark \ref{Rem:ConjKRUnramified}. The result for $\scrs_{U''}(G,X)$ follows from the result for  $\scrs_{U'}(G,X)$ in combination with \cite[Proposition 7.7]{Z}.  
\end{proof}
\begin{proof}[Proof of Theorem \ref{Thm:CMLifting}]
Recall that we assumed in Theorem \ref{Thm:Uniformisation} that $\pi_1(G)_I$ is torsion free, which implies that all parahoric subgroups of $G(\qp)$ are connected by Lemma \ref{Lem:TorsionFreeParahoricConnected}. Theorem \ref{Thm:CMLifting} is therefore a direct consequence of Corollary \ref{Cor:Uniformisation}.
\end{proof}
\begin{proof}[Proof of Theorem \ref{Thm:HeRapoport}]
By \cite[Theorem 8.1.(ii)]{Z}, uniformisation of isogeny classes, as proved in Theorem \ref{Thm:Uniformisation}, implies that the He--Rapoport axioms hold.
\end{proof}
\begin{proof}[Proof of Theorem \ref{Thm:EKOR}]
This follows from Theorem \ref{Thm:IrredEKOR} below by noting that when $\gad$ is $\mathbb{Q}$-simple as in the assumptions of Theorem \ref{Thm:IrredEKOR}, then $\mathbb{Q}$-non basic just means non-basic. Note that Theorem \ref{Thm:IrredEKOR} has the assumption that either $\mathbf{Sh}_{U}(G,X)$ is proper or Conjecture \ref{Conj:KR} holds, which is true if either $\mathbf{Sh}_{U}(G,X)$ is proper or if $G_{\qp}$ is unramified, see the proof of Theorem \ref{Thm:Uniformisation} above.
\end{proof}
\subsection{Consequences for irreducible components}
In this section we prove a generalisation of Theorem \ref{Thm:EKOR}. Before we can state it, we need to introduce some notation. Let $\gad=\prod_{i} G_i$ be the decomposition of $\gad$ into simple groups over $\mathbb{Q}$ and consider the induced maps of Kottwitz sets
\begin{align}
  B(G_{\qp}) \to B(\gad_{\qp}) \to \prod_i B(G_{i,\qp}).
\end{align}
\begin{Def}[Definition 5.3.2 of \cite{KretShin}]
An element $[b] \in B(G)$ is called $\mathbb{Q}$\emph{-non-basic} if the image of $[b]$ in $B(G_{i, \mathbb{Q}_p})$ is non-basic for all $i$. A Newton stratum $\shgb$ is called $\mathbb{Q}$\emph{-non-basic} if $[b]$ is $\mathbb{Q}$-non-basic.
\end{Def}
Recall that $K \subset \mathbb{S}$ corresponds to a very special parahoric.
\begin{Thm} \label{Thm:IrredEKOR}
Let $w \in \kadmu$ such that the EKOR stratum $\shg\{w\}$ intersects a $\mathbb{Q}$-non-basic Newton stratum. If either $\mathbf{Sh}_{U}(G,X)$ is proper or Conjecture \ref{Conj:KR} holds, then
\begin{align}
  \shg\{w\} \to \shg
\end{align}
induces a bijection on $\pi_0$.
\end{Thm}
We start by proving a lemma.
\begin{Lem} \label{Lem:ForgetfulEKOR}
For $w \in \kadmu$, viewed as an element of $\admu$ via $\kadmu \subset \admu$, the forgetful map $\shge(w) \to \shg$ factors through $\shg\{w\}$, via a surjective map $\shge(w) \to \shg\{w\}$.
\end{Lem}
\begin{proof}
The factorisation is \cite[Theorem 5.4.5.(3)]{ShenYuZhang} and the surjectivity is proved there under the assumption that Axiom 4(c) of \cite{He-Rapoport} holds, which is true by Theorem \ref{Thm:HeRapoport}.
\end{proof}
\begin{proof}[Proof of Theorem \ref{Thm:IrredEKOR}]
We will prove that if $w \in \admu$ such that $\shge(w)$ intersects a $\mathbb{Q}$-non-basic Newton stratum, then the natural map $\pi_0(\shge(w)) \to \pi_0(\shg)$ is a  bijection. By Lemma \ref{Lem:ForgetfulEKOR}, this will imply Theorem \ref{Thm:IrredEKOR}. \medskip 

\textbf{Step 1:} We first deal with the case of $\sigma$-straight $w \in \admu$. Then $\shge(w)$ is contained in a unique Newton stratum $\shgeb$, see \cite[Theorem 1.3.5]{ShenYuZhang}, which by assumption is $\mathbb{Q}$-non-basic. We deduce from Theorem \ref{Thm:DiagramCartesian}, Corollary \ref{Cor:Normal} and Proposition \ref{Prop:GlobalActionTransitiveKR} that for any finite set of primes $\Sigma$ with $p \in \Sigma$, the group $\gsc(\afs)$ acts transitively on the fibers of 
\begin{align}
  \pi_0(\shgeinf(w)) \to \pi_0(\shginf).
\end{align}
By Lemma \ref{Lem:FiniteEtale}, there exists $y \in \kadmu$ such that the natural map $\shget(w) \to \shg$ factors via a finite \'etale map $\shget(w) \to \shg\{y\}$. We want to apply \cite[Theorem 3.4.1]{OrdinaryHO} to the $G(\afp)$-equivariant finite \'etale cover $\shgeinf(w) \to \shginf\{y\}$. Note that \cite[Hypothesis 2.3.1]{OrdinaryHO} follows from Theorem \ref{Thm:CMLifting}, see \cite[Remark 2.3.3]{OrdinaryHO}. Moreover, \cite[Hypothesis 3.4.1]{OrdinaryHO} is satisfied since $\shg$ is normal because $\mathcal{G}_K$ is very special, see \cite[Corollary 4.6.26]{KisinPappas}. The assumption that every connected component of $\shg\{y\}$ intersects a $\mathbb{Q}$-non-basic Newton stratum holds since $\shge\{y\}$ is contained in a single $\mathbb{Q}$-non-basic Newton stratum since $\shge(w)$ is. 

Therefore the assumptions of \cite[Theorem 3.4.1]{OrdinaryHO} are satisfied and we conclude that if $\Sigma$ contains all the primes $\ell$ where $G^{\mathrm{ad}}_{\ql}$ has a compact factor, then $\gsc(\afs)$ acts trivially on the fibers of $\pi_0(\shgeinf(w)) \to \pi_0(\shginf)$. Since it also acts transitively on these fibers by Proposition \ref{Prop:GlobalActionTransitiveKR} and since the map $\pi_0(\shgeinf(w)) \to \pi_0(\shginf)$ is surjective by Corollary \ref{Cor:NonEmptiness}, we deduce that $\pi_0(\shgeinf(w)) \to \pi_0(\shginf)$ is a bijection.
\smallskip

\textbf{Step 2:} For general $w \in \admu$ intersecting a $\mathbb{Q}$-non-basic Newton stratum $\shgeb$, there is a $\sigma$-straight element $w' \le w$ with $\shge(w') \cap \shgeb$ non-empty; this follows from \cite[Theorem 4.1]{HZ} as in the proof of Lemma \ref{Lem:SigmaStraight}. It follows from our assumptions that every connected component $V$ of $\shge(\le \! w)$ intersects $\shge(\tau)$.\footnote{Recall that $\tau$ is the unique element zero element of $\admu$.} Thus the intersection $V \cap \shge(\le \! w')$ is non-empty, and it is therefore a union of connected components of $\shge(\le \! w')$. Hence $V \cap \shge(\le \! w')$ is equidimensional of dimension $\ell(w')$ and must therefore intersect $\shge(w')$. We see that the natural map $\shge(w') \to \shge(\le \! w)$ induces a surjective map on $\pi_0$. Consider the commutative diagram
\begin{equation}
  \begin{tikzcd}
  \pi_0(\shgeinf(w')) \arrow{dr} \arrow[rr] & & \pi_0(\shgeinf(\le \! w) \arrow{dl} \\
  & \pi_0(\shginf).
  \end{tikzcd}
\end{equation}
The right diagonal map is surjective by Corollary \ref{Cor:NonEmptiness}, the horizontal map is surjective by the discussion above and the left diagonal arrow is a bijection by step $1$. It follows that $\pi_0(\shgeinf(\le \! w)) \to \pi_0(\shginf)$ is a bijection and since $\pi_0(\shgeinf(w)) \to \pi_0(\shgeinf(\le \! w))$ is a bijection by the normality of $\pi_0(\shgeinf(\le \! w))$, see Corollary \ref{Cor:Normal}, we are done.
\end{proof}
\paragraph{Acknowledgments}
I am grateful to James Newton for his guidance during this project and for his careful reading of early versions of this work. I would like to thank Ana Caraiani, Andrew Graham, Paul Hamacher, Ashwin Iyengar, Kai-Wen Lan, Yasuhiro Oki, Benoit Stroh, Giorgi Vardosanidze, Luciena Xiao Xiao , Chia-Fu Yu, Zhiyu Zhang, Paul Ziegler, and Rong Zhou for helpful conversations and correspondence. I would particularly like to thank Rong Zhou for writing the appendix, for encouraging me to work on this project and for his generous help at all stages of the preparation of this paper. Furthermore I would like to thank Mikhail Borovoi for allowing me to reproduce his Mathoverflow answer. I am very grateful to Thomas Haines and Tony Scholl for being my PhD examiners and for giving detailed feedback on this work. Last, I would like to thank the anonymous referees for their thorough reading and countless helpful suggestions. 

\appendix 
\section{Connected components of affine Deligne--Lusztig varieties with very special level structure, by Rong Zhou} \label{Sec:VerySpecialADLV}
	
As explained in the introduction, proving uniformisation of isogeny classes in Shimura varieties of Hodge type with parahoric level is closely related to the problem of understanding connected components of affine Deligne--Lusztig varieties with parahoric level. In this Appendix, we study the connected components of affine Deligne--Lusztig varieties with very special level structure and prove for instance that $J_b(\mathbb{Q}_p)$ acts transitively on these connected components. These results will be used in Section \ref{Sec:UniformisationVerySpecial} to prove uniformisation of isogeny classes in Shimura varieties of Hodge type with very special level.
  \subsection{The main result} We follow the notation of Section \ref{Sec:Local}. Thus $G$ is a reductive group over $\bbQ_p$ and $\{\mu\}$ is a geometric conjugacy class of cocharacters of $G_{\overline{\bbQ}_p}$. We assume that $G$ is quasi-split, and we let $\calI$ be the Iwahori group scheme corresponding to a $\sigma$-stable alcove $\fka$ in the building for $G$. We fix $\calG$ a very special standard parahoric group scheme for $G$. Then $\calG$ corresponds to a $\sigma$-stable special point $\fks$ lying in the closure of $\fka$ and we write $K\subset \bbS$ for the subset of simple affine reflections which preserve $\fks$. The projection $\tilde W \rightarrow W_0$ induces an isomorphism $W_K\cong W_0$. 
	
	As explained in \cite[\S9]{Z}, we have an identification $$W_K\backslash\tilde W/W_K\cong X_*(T)^+_I.$$
	
	By \cite{He-Rapoport}, there exists a reduced root system $\Sigma$ (the \'echelonnage root system) such that $$W_a\cong W(\Sigma)\ltimes Q^\vee(\Sigma),$$ where $W(\Sigma)$ (resp. $Q^\vee(\Sigma)$) is the Weyl group (resp. coroot lattice) of $\Sigma$. We define a partial order $\preccurlyeq$ on $X_*(T)^+_I$ by setting $\lambda\preccurlyeq \lambda'$ if $\lambda-\lambda'$ can be written as a sum of positive coroots in $Q^\vee(\Sigma)$ with positive integral coefficients. Then by \cite[pp. 210]{Lu}, the Bruhat order on $W_K\backslash\tilde W/W_K$ agrees with the partial order $\preccurlyeq$. It follows that for $\mu\in X_*(T)_I^+$, we have $$\Adm(\mu)_K=\{\lambda\in X_*(T)^+_I| \lambda\preccurlyeq \mu\}.$$
	
	Let $b\in G(\brQ)$ such that $[b]\in \bgmu$. We have the affine Deligne--Lusztig variety $X(\mu,b)_K$ defined in Section \ref{Sec:ADLVSet}. We also set $$\Adm(\mu)^K:=W_K\Adm(\mu)W_K\subset \tilde W$$ and define $$X(\mu,b)^K:=\bigcup_{w\in \Adm(\mu)^K}X_w(b),$$
	which is a locally closed subscheme of the Witt vector affine flag variety $\mathrm{Gr}_{\calI}$. Then there is a natural map \begin{equation}\label{eqn: fibration ADLV}X(\mu,b)^K\rightarrow X(\mu,b)_K\end{equation} which is equivariant for the action of the $\sigma$-centraliser group $J_b(\qp)$. In fact, \eqref{eqn: fibration ADLV} is a fibration with connected fibers and hence induces a $J_b(\qp)$-equivariant bijection \begin{equation}\label{eqn: fibration ADLV 2}\pi_0(X(\mu,b)^K)\xrightarrow{\sim} \pi_0(X(\mu,b)_K).\end{equation}
	
	\subsubsection{} Recall from e.g \cite[Section 1.1.2]{KMPS} that associated to $[b] \in B(G)$ there is a \emph{Newton cocharacter} $\overline\nu_{[b]}$. Let $M$ denote the centraliser of $\overline\nu_{[b]}$ and we fix a representative $b$ of $[b]$ such that $\nu_{[b]}=\overline\nu_{[b]}$. The existence of such a representative follows from the same argument as in \cite[Lemma 2.5.2]{CKV} which also shows that $b\in M(\brQ)$. Then $b$ is a basic element of $M$, in other words $\overline{\nu}_{[b]}$ is central in $M$.
	
	We use a subscript $M$ to denote the corresponding objects for $M$. Thus $\tilde W_M$ (resp. $W_{a,M}$) denotes the Iwahori--Weyl group (resp. affine Weyl group) for $M$. The intersection $M(\brQ)\cap\calG(\brZ)$ arises as the $\brZ$-points of a very special parahoric group scheme $\calM$ for $M$, which is standard for the alcove $\fka_M$ for $M$ determined by $\fka$. We write $\calI_M$ for the Iwahori group scheme of $M$ determined by $\fka_M$ and we let $K_M\subset \bbS_M$ denote the subset of simple affine reflections for $M$ corresponding to $\calM$.
	We let $\Sigma_M$ denote the \'echelonnage root system for $M$ so that $$W_{a,M}\cong W(\Sigma_M)\ltimes Q^\vee(\Sigma_M).$$ 
	
	For $x\in \pi_1(M)_I$, we write $\tau_x\in \Omega_M$ for the corresponding length 0 element, and we write $\tau_x=t^{\mu_x}w_x$ for a unique $w_x\in W_K$. Then the map $x\mapsto \mu_x$ induces a bijection \begin{equation}\label{eqn: I set}\pi_1(M)_I\cong \{\lambda\in X_*(T)_I|\text{$\lambda$ is $M$-dominant and $M$-minuscule}\},\end{equation} here $M$-minuscule means minuscule with respect to the root system $\Sigma_M$. We define the set $$I_{\mu,b,M}:=\{x\in \pi_1(M)_I| \kappa_M(b)=x, \mu_x\preccurlyeq\mu\}.$$
	
	Via the bijection \eqref{eqn: I set}, we also consider $I_{\mu,b,M}$ as a subset of the set of $M$-minuscule and $M$-dominant elements in $X_*(T)_I$. For each $\lambda\in I_{\mu,b,M}$, we have the affine Deligne--Lusztig variety $X^M(\lambda,b)_{K_M}$ for the group $M$. It is a closed subscheme of the partial affine flag variety for $M$ with respect to the parahoric subgroup $\mathcal{M}$, and its $\ovfp$-points are given by 
	$$\{m\in M(\brQ)/\calM(\brZ)| m^{-1}b\sigma(m)\in \calM(\brZ)\dot t^{\lambda'}\calM(\brZ), \lambda'\preccurlyeq_M\lambda\}.$$
	It is equipped with a natural map \begin{equation}\label{eqn: ADLV Levi}X^M(\lambda,b)_{K_M} \rightarrow X(\mu,b)_K\end{equation}
	which is equivariant for the action of the $\sigma$-centraliser group $J_b(\bbQ_p)$.
	
	\subsubsection{}Our main theorem on the connected components of affine Deligne--Lusztig varieties is the following.
	
	\begin{Thm}\label{thm: conn ADLV}
		$J_b(\bbQ_p)$ acts transitively on $\pi_0(X(\mu,b)_K)$. In particular, for any $\lambda\in I_{\mu,b,M}$ the map \eqref{eqn: ADLV Levi} induces a surjection $$\pi_0(X^M(\lambda,b)_{K_M})\rightarrow \pi_0(X(\mu,b)_K).$$
	\end{Thm}
	\begin{Rem}
		The theorem is stated for $G$ a quasi-split reductive group over $\bbQ_p$. However, the result makes sense for general quasi-split groups over any local field $F$ and can be proved in exactly the same way.
	\end{Rem}
	
	\subsubsection{}We follow the strategy of \cite{CKV} and \cite{Nie1} where this result was proved for unramified groups $G$. The result follows from the following two propositions.
	
	\begin{prop}\label{prop: conn ADLV Levi 1}The natural map $$\coprod_{\lambda\in I_{\mu,b,M}} X^M(\lambda,b)_{K_M}\rightarrow X(\mu,b)_K$$ induces a surjection $$\coprod_{\lambda\in I_{\mu,b,M}} \pi_0(X^M(\lambda,b)_{K_M})\rightarrow \pi_0(X(\mu,b)_K).$$
	\end{prop}
	
	\begin{prop}\label{prop: conn ADLV Levi 2} Let $\lambda\in I_{\mu,b,M}$. The image of the natural map $$\pi_0(X^M(\lambda,b)_{K_M})\rightarrow \pi_0(X(\mu,b)_K)$$ does not depend on the choice of $\lambda\in I_{\mu,b,M}$.
		
	\end{prop}
	
	\begin{proof}[Proof of Theorem \ref{thm: conn ADLV}] Fix $\lambda \in I_{\mu,b,M}$. By Proposition \ref{prop: conn ADLV Levi 1} and Proposition \ref{prop: conn ADLV Levi 2}, the map $$\pi_0(X^M(\lambda,b)_{K_M})\rightarrow \pi_0(X(\mu,b)_K)$$ is surjective. By \cite[Theorem 4.1 and Theorem 5.1]{HZ}, $J_b(\bbQ_p)$ acts transitively on $\pi_0(X^M(\lambda,b)_{K_M})$, and hence on $\pi_0(X(\mu,b)_K)$.
		
	 \end{proof}
	
	\subsubsection{}We now proceed to prove the two propositions. Note that by a standard reduction (see \cite[Section 6]{HZ}), it suffices to prove the propositions when $G$ is adjoint and $\bbQ_p$-simple. We may and do assume this from now on.

	\subsection{Proof of Proposition \ref{prop: conn ADLV Levi 1}}
	\subsubsection{}In the case of unramified groups, Proposition \ref{prop: conn ADLV Levi 1} is \cite[Proposition 3.4.1]{CKV}. Here we prove the general case using a different method based on the Deligne--Lusztig reduction method for affine Deligne--Lusztig varieties in the affine flag variety.
	
	We begin with some preliminaries regarding $\sigma$-conjugacy classes in Iwahori--Weyl groups. For any element $w\in \tilde W$, we let $n$ be a sufficiently divisible integer such that $\sigma^n$ acts trivially on $\tilde W $ and $w\sigma(w)\dotsc\sigma^{n-1}(w)=t^\lambda$ for some $\lambda\in X_*(T)_I$. We set $\nu_w:=\frac{t^\lambda}{n}\in X_*(T)_{I,\bbQ}$ and $\overline{\nu}_w\in X_*(T)_{I,\bbQ}^+$ for the dominant representative of $\nu_w$. We let $\kappa(w)\in \pi_1(G)_\Gamma$ denote the image of $w$ under the projection $\tilde{W}\rightarrow \pi_1(G)_I\rightarrow\pi_1(G)_\Gamma$. We write $B(\tilde W,\sigma)$ for the set of $\sigma$-conjugacy classes in $\tilde W$. Then $w\mapsto [\dot{w}]$ induces a well-defined map $\Psi:B(\tilde W,\sigma)\rightarrow B(G)$ and we have a commutative diagram (see \cite[Section 3.3, Theorem 3.5]{He1}):
  \[\xymatrix{B(\tilde W,\s)\ar[rr]^\Psi \ar[rd]_{(\overline\nu,\kappa)}&& B(G)\ar[ld]^{(\overline\nu,\kappa)}\\
	& (X_*(T)_{I,\mathbb{Q}}^+)^\sigma\times \pi_1(G)_{\Gamma}&.}\]
	
	\subsubsection{}We will need the following lemma. We write $J\subset K\cong\bbS_0$ for the subset corresponding to $M$. Recall that an element $w\in \tilde W$ is said to be $\sigma$-straight if $$n\ell(w)=\ell(w\sigma(w)\dotsc\sigma^{n-1}(w))$$ for all $n$. 
	\begin{Lem}\label{lem: short element} Let $w\in \tilde W$ be $\sigma$-straight such that $\dot w\in[b]$; in particular $\overline\nu_b=\overline \nu_w$. Let $u\in {^JW_0}$ such that $u(\nu_w)=\overline{\nu}_b$ and set $w_\sharp:= uw\sigma(u)^{-1}$. Then $w_\sharp\in \Omega_M$.
	\end{Lem}
	\begin{proof}It suffices to show that $w_\sharp\in W_M$ and $$w_\sharp\calI_M(\brZ)w_\sharp^{-1}=\calI_M(\brZ).$$
		The first statement follows since $w_\sharp(\overline\nu_w)=\overline\nu_w$.
	By \cite[Theorem 5.2]{HZ}, the element $w$ is $(\nu_w,\sigma)$-fundamental. Thus $$w\sigma(\calI_{M_{\nu_w}}(\brZ))w^{-1}=\calI_{M_{\nu_w}}(\brZ)$$ where $M_{\nu_w}$ is the centraliser of $\nu_w$ and $\calI_{M_{\nu_w}}(\brZ):=\calI(\brZ)\cap M_{\nu_w}(\brQ)$. Since $u\in {^JW_0}$, we have $u\calI_{M_{\nu_w}}(\brZ)u^{-1}=\calI_M(\brZ)$. It follows that \begin{align*}w_\sharp\calI_M(\brZ)w_\sharp^{-1} &= w_\sharp\sigma(\calI_M(\brZ))w_\sharp^{-1}\\
		&=uw \sigma(\calI_{M_{\nu_w}}(\brZ))w^{-1}u^{-1}\\
		&=\calI_M(\brZ).
		\end{align*}
		as desired.
	 \end{proof}

	\begin{proof}[Proof of Proposition \ref{prop: conn ADLV Levi 1}] By \eqref{eqn: fibration ADLV 2}, it suffices to show the natural map $$\coprod_{\lambda\in I_{\mu,b,M}} X^M(\lambda,b)^{K_M}\rightarrow X(\mu,b)^K$$  induces a surjection $$\coprod_{\lambda\in I_{\mu,b,M}} \pi_0(X^M(\lambda,b)^{K_M})\rightarrow \pi_0(X(\mu,b)^K).$$ 
		
		Let $Y$ be a connected component of $X(\mu,b)^K$. Then by \cite[Theorem 4.1]{HZ}, there exists a $\sigma$-straight element $w\in\Adm(\mu)^K$, such that $Y\cap X_w(b)\neq \emptyset$. Let $w_\sharp$ denote the element constructed in Lemma \ref{lem: short element} and $u\in {^JW_0}$ such that $uw\sigma(u)=w_\sharp$. Then we claim that $[b]_M= [\dot w_\sharp]_M\in B(M)$. Indeed, we have $\nu_{ w_\sharp}=\overline\nu_w=\overline \nu_b$. Therefore the image of $[b]_M$ and $[\dot w_\sharp]_M$ in $\pi_1(M)_I$ coincide up to torsion. On the other hand, the images of $[b]_M$ and $[\dot w_\sharp]_M$ in $\pi_1(G)_I$ coincide and $\ker(\pi_1(M)_I\rightarrow \pi_1(G)_I)$ is torsion free. It follows that $\kappa_M([b]_M)=\kappa_M([\dot w_\sharp]_M)$ and hence $[b]_M=[\dot w_\sharp]_M$. Thus we may replace $b$ by $\dot w_\sharp$.
		
		We will show that $Y\cap X_{w_\sharp}^M(\dot w_\sharp)\neq \emptyset$. Since $w_\sharp\in\Adm^M(\lambda)^{K_M}$, where $\lambda\in I_{\mu,b,M}$ corresponds to the image of $w_\sharp$ in $\pi_1(M)_I$, it follows that $X_{w_\sharp}(\dot w_\sharp)\subset X^M(\lambda,\dot w_\sharp)^{K_M}$; this implies the proposition.
		
		For any affine root $\alpha$, we let $\calU_{-\alpha_i}$ denote the affine root subgroup corresponding to $-\alpha_i$ over $\brZ$. By \cite[Section 4.3.2, 4.3.5 and 4.3.7]{BT2}, $\calU_{-\alpha_i}$ is the group scheme associated to a finite free $\brZ$-module. For any $\epsilon>0$, we let $\calU_{-\alpha_i+\epsilon}(\brZ)$ be the subgroup of $\calU_{-\alpha_i}(\brZ)$ corresponding to the affine function $-\alpha+\epsilon$. Similarly, we write $\calU_{-\alpha_i+}(\brZ)$ for the union of $\calU_{-\alpha_i+\epsilon}(\brZ)$ over all $\epsilon>0$. As the notation suggests, these arise as the $\brZ$-points of group schemes $\calU_{-\alpha_i+\epsilon}$ and $\calU_{-\alpha_i+\epsilon}$ over $\brZ$, and the quotient $\calU_{-\alpha_i}(\brZ)/\calU_{-\alpha_i+}(\brZ)$ is a 1-dimensional vector space over $ k$. We choose a group scheme homomorphism $$f:\bbG_a\rightarrow \calU_{-\alpha}$$ which lifts the map $k\cong \calU_{-\alpha_i}(\brZ)/\calU_{-\alpha_i+\epsilon}(\brZ)$. 
		
		For $R$ a perfect $k$-algebra and $a\in R$, the map $$h_\alpha:R\rightarrow \calU_{-\alpha}(W(R))$$ $$a\mapsto f([a])$$ where $[a]\in W(R)$ is the Teichm\"uller lift of $a$, induces a $ k$-scheme morphism $$h_\alpha:\bbA^{1,\mathrm{perf}}\rightarrow L^+\calU_{-\alpha},$$
		where $\bbA^{1,\mathrm{perf}}$ denotes the perfection of $\bbA^1$ over $k$. The induced morphism $\bbA^{1,\mathrm{perf}}\rightarrow \mathrm{Gr}_{\calI}$ extends to a morphism $\bbP^{1,\mathrm{perf}}\rightarrow \mathrm{Gr}_{\calI}$ also denoted $h_\alpha$. Then we have \begin{equation}\label{eqn: pt at infty}h_\alpha(\infty)=\dot{s}_\alpha\calI\in \mathrm{Gr}_{\calI}.\end{equation}
		
		Let $g_0\calI\in Y\cap X_w(\dot w_\sharp)$ with $g_0\in G(\brQ)$. By \cite[Theorem. 5.2]{HZ}, we may choose $g_0$ such that $g_0^{-1}\dot w_\sharp g_0= \dot{w}$. Let $s_n\dotsc s_1$ be a reduced word decomposition for $u$ (note that $s_i\in K$) and we write $u_i$ for the element $s_1\dotsc s_i\in \tilde W$ and $u_0=e$. We write $g_i\in G(\brQ)$for the element $g_0\dot u_i$. We will prove by induction that $g_i\calI\in Y$ for $i=0,1,\dotsc,n$; clearly this is true for $i=0$. 
		
		Assume $g_i\in Y$ and we let $\alpha_{i+1}$ denote the positive affine root corresponding to $s_{i+1}$. We consider the map $$g:=g_ih_{\alpha_{i+1}}:\bbP^{1,\mathrm{perf}}\rightarrow \mathrm{Gr}_{\calI}.$$
		
		Since $\mathcal{U}_{-\alpha_{i+1}}\subset \mathcal{I}\dot{s}_{i+1}\mathcal{I}$, for any $s\in \mathbb{A}^{1,\mathrm{perf}}(k)$ we have  \begin{align*}g(s)^{-1}\dot{w}_{\sharp} \sigma(g(s))&=h_{\alpha_{i+1}}(s)^{-1}g_i^{-1}\dot{w}_{\sharp} \sigma(g_i)\sigma(h_{\alpha_{i+1}}(s))
  \\ &\in \calI\dot{s}_{\alpha_{i+1}}\calI \dot{w}_i\calI \sigma(\dot{s}_{\alpha_{i+1}})\calI\\ &\subset \bigcup_{x\in A} \calI\dot{x}\calI\end{align*} where $w_i:=u_i^{-1}w_\sharp\sigma(u_i)\in \Adm(\mu)^K$ and $A\subset \tilde{W}$ is the subset $$A=\{w_i,s_{i+1}w_i,w_i\sigma(s_{i+1}),s_{i+1}w_i\sigma(s_{i+1})\}.$$ Since $\Adm(\mu)^K$ is closed under left and right multiplication by $W_K$, we have $A\subset \Adm(\mu)^K$, and hence $$g(s)^{-1}\dot{w}^{\sharp} \sigma(g(s))\in \bigcup_{v\in\Adm(\mu)^K} \calI\dot{v}\calI$$
		for any $s\in \bbP^{1,\mathrm{perf}}(k)$. Moreover we have $g(0)=g_i\calI$ and $g(\infty)=g_{i+1}\calI$, where the latter equality follows from \eqref{eqn: pt at infty}. Thus the image of $g$ is a curve in $X(\mu,b)^K$ which connects $g_i\calI$ and $g_{i+1}\calI$ and hence $g_{i+1}\calI\in Y$. Then by definition $g_n$, $g_n\calI$ lies in the image of $X^M_{w_\sharp}({\dot{w}_\sharp})$ as desired.
	 \end{proof}

	\subsection{Proof of Proposition \ref{prop: conn ADLV Levi 2}}
	\subsubsection{}When $G$ is unramified, this proposition follows from the proof of \cite[Proposition 4.1.12]{CKV} when $\mu$ is minuscule; the general case is proved in \cite[Proposition 5.1]{Nie1}. The main input is the construction of explicit curves in $X(\mu,b)_K$ which connect points in $X^M(\lambda,b)_{K_M}$ and $X^M(\lambda',b)_{K_M}$ for $\lambda\neq\lambda'\in I_{\mu,b,M}$. The construction of these curves relied on certain combinatorial results concerning the root system for $G$. The exact same method of proof works in our setting; however, there are a few subtleties which we now explain. 
	
	First, the explicit curves were constructed in \cite{CKV} and \cite{Nie1} using root subgroups of $G_{\brQ}$ which are all isomorphic to $\bbG_a$ when the group is unramified. In general, the root subgroups are more complicated and thus one needs to be more careful. However, we are still able to give a uniform construction of the curves that we need.
	
	Secondly, we need to generalise the combinatorial results to general quasi-split groups $G$. It turns out there is a systematic way to deduce these combinatorial results for quasi-split $G$ from the case of unramified groups which we now explain. 
	
	\subsubsection{} \label{Sec:UnramifiedReduction} Recall we have assumed $G$ is adjoint and simple. As in \cite[\S7.2]{He1}, see \cite[Proof of Theorem A.3.1]{ZhouZhu} for an explicit construction, there is an unramified adjoint group $H$ over $\qp$ such that the pair $(W',\sigma')$ consisting of the Iwahori--Weyl group for $G'$ and the action of Frobenius is identified with the pair $(W,\sigma)$. Moreover the \'echelonnage root system $\Sigma$ is identified with the absolute root system $\Sigma'$ for $G'$, and we have an isomorphism 
	$$\pi_1(G)_I\cong\pi_1(G')$$
	$$X_*(T)_I\cong X_*(T'),$$ where $T'$ is a suitable maximal $\brQ$-split torus of $G'$. 
	
	We use a superscript $'$ to denote the corresponding objects for $G'$. Then $\calG$ determines a hyperspecial subgroup $\calG'$ for $G'$ and we write $K'\subset \bbS' $ for the corresponding subset of simple reflections. Then $M$ determines a Levi subgroup $M'$ of $G'$ and hence a subset $J'\subset K'$. It follows that the combinatorial data \begin{equation}\label{eqn: comb data}(\Sigma,X_*(T)_I,\sigma, J,\mu,\kappa_M(b))\end{equation} is identified with the corresponding data for $G'$. Thus any result which only depends on the data (\eqref{eqn: comb data}), can be reduced to the case of unramified groups. The combinatorial results that we need are already proved in the case of unramified groups in \cite{Nie1} and \cite{CKV}. We therefore take the convention that whenever we need certain results which depend on the data (\eqref{eqn: comb data}), we will refer to the relevant result in \cite{CKV} or \cite{Nie1}.

	\subsubsection{}We now proceed with the proof of Proposition \ref{prop: conn ADLV Levi 2}. Let $x,x'\in \pi_1(M)_I$. We write $x \overset{(\alpha,r)}\rightarrow x'$ for some $\alpha\in \Sigma$ and $r\in\bbN$ if $x-x'=\alpha^\vee-\sigma^r(\alpha^\vee)$ and $$\mu_x, \mu_{x+\alpha^\vee}, \mu_{x-\sigma^r(\alpha^\vee)},\mu_{x'}\preccurlyeq \mu.$$ We write $x\overset{(\alpha,r)}\rightarrowtail x'$ if $x \overset{(\alpha,r)}\rightarrow x'$ and neither
	$$x \overset{(\alpha,i)}\rightarrow x+\alpha^\vee-\sigma^i(\alpha^\vee) \overset{(\sigma^{i}(\alpha),r-i)}\rightarrow x'$$ nor
	$$x \overset{(\sigma^i(\alpha),r-i)}\rightarrow x+\sigma^i(\alpha^\vee)-\sigma^r(\alpha^\vee) \overset{(\alpha,i)}\rightarrow x'$$ for any $i\in[1,r-1]$.

	We let $$\langle\ ,\ \rangle: \left(Q(\Sigma)\otimes_{\bbZ}\bbR\right)\times \left(X_*(T)_I\otimes_{\bbZ}\bbR\right) \rightarrow \bbR$$ be the natural pairing, where $Q(\Sigma)$ is the root lattice of $\Sigma$. For any element $\alpha\in \Sigma$, we write $\calO_\alpha$ for the $\sigma$-orbit of $\alpha$. We let $h$ denote the number of connected components of the Dynkin diagram of $G$ over $\brQ$; then we have $\#\calO_\alpha\in\{h,2h,3h\}$.

	\begin{Lem}[{\cite[Lemma 7.7]{Nie1}}]\label{lem: Nie 7.7}Let $x\neq x'\in I_{\mu,b,M}$. Then there exists $x_j\in \pi_1(M)_I$, $\alpha_j\in \Sigma-\Sigma_M$ and $r_j\in\bbN$ for $j\in[0,m-1]$ such that 
		\begin{enumerate}
			\item $\alpha_j^\vee$ is $M$-dominant and $M$-minuscule.
			
			\item $r_j\in [1,h]$ if $\#\calO_{\alpha_j}\in\{h,2h\}$ and $r_j\in [1,2h-1]$ if $\#\calO_{\alpha_j}=3h$.
			
			\item $x_0=x, x_m=x'$ and we have $x_j \overset{(\alpha_j,r_j)}\rightarrow x_{j+1}$ for $j\in[0,m-1]$.
			
		\end{enumerate}
	\end{Lem}
	\begin{proof}This follows from \cite[Lemma 7.7]{Nie1} by discussion in Section \ref{Sec:UnramifiedReduction} above. Note that in loc. cit., the result is stated for $M$ a Levi subgroup such that $b$ is superbasic in $M$. However, one checks that the same proof works for any $M$ as long as $I_{\mu,b,M}$ contains a weakly dominant element. Here $\lambda\in X_*(T)_I$ is weakly dominant if $\langle \alpha, \lambda\rangle\geq-1$ for any positive root $\alpha\in \Sigma$. But as in \cite[Lemma 4.1]{Nie1}, any element $\lambda\in I_{\mu,b,M}$ is weakly dominant, so the result applies to our $M$.

	 \end{proof}
	
	\subsubsection{}We now construct certain curves inside $LG$ which we will use to connect points in $X(\mu,b)_K$. 
	Let $\alpha\in \Sigma$ be a root. Then $\alpha$ determines a relative root $\widetilde{\alpha}$ of $G$ over $\brQ$ which we always take to be the short root. We let $G_\alpha$ denote the simply connected cover of the (semi-simple) group generated by $U_{\widetilde\alpha}$ and $U_{-\widetilde\alpha}$ and we write $$i_\alpha:G_\alpha\rightarrow G$$ for the natural map. We let $\calG_\alpha$ denote the very special parahoric of $G_\alpha$ such that $\calG(\brZ):=G_\alpha(\brQ)\cap i_\alpha^{-1}(\calG(\brZ))$.
	
	If $\widetilde\alpha$ is not divisible, then we have an isomorphism 
	$$G_\alpha\cong \mathrm{Res}_{\breve K/\brQ}SL_2,$$ where $\breve K/\brQ$ is a finite extension. Then up to conjugacy $\calG_\alpha$ is identified with the very special parahoric $SL_2(\calO_{\breve K})\subset G_\alpha(\brQ)$ and there is an isomorphism $$f_{\widetilde\alpha}:\mathrm{Res}_{\breve K/\brQ}\bbG_a\xrightarrow{\sim}U_{\widetilde\alpha}.$$

	If $\widetilde\alpha$ is divisible, then there is an isomorphism 
	\begin{equation}G_\alpha\cong \mathrm{Res}_{\breve K/\brQ}SU_3,\end{equation} where $SU_3$ is the special unitary group over $\breve K$ associated to a quadratic extension $\breve K'/\breve K$. \subsubsection{}We recall the presentation of the $\breve K$-group $\SU_3$ in \cite[Example 1.15]{Ti1}. 
	
	We let $\tau\in \mathrm{Gal}(\breve K'/\breve K)$ denote the nontrivial element, and we consider the Hermitian form on $\breve K'^3$ given by $$\langle(x_{-1},x_0,x_1),(y_{-1},y_0,y_1)\rangle=\tau(x_{-1})y_1+\tau(x_0)y_0+\tau(x_1)y_{-1}.$$
	The group $\SU_3$ is the special unitary group attached to this form. For $i=-1,1$ and $c,d\in \breve K'$ such that $\tau(c)c+d+\tau(d)=0$, we define $$u_{i}(c,d)=I_3+(g_{rs})$$ where $I_3$ is the identity matrix and $(g_{rs})$ is the matrix with entries $g_{-i,0}=-\tau(c)$, $g_{0,i}=c$, $g_{-i,i}=d$ and $g_{rs}=0$ otherwise. The root subgroups are then given by $$U_{\pm\widetilde{\alpha}}(\breve K)=\{u_{\pm 1}(c,d)|c,d\in \breve K', \tau(c)c+\tau(d)+d=0\}$$
	$$U_{\pm2\widetilde{\alpha}}(\breve K)=\{u_{\pm 1}(0,d)|c,d\in \breve K', \tau(d)+d=0 \}.$$
	
	We consider the very special parahoric $G_{\alpha}(\breve F)\cap GL_3(\calO_{\breve K'})$ of $G_\alpha(\breve F)$; we call this the standard parahoric. Let $\pi\in\breve K'$ be a uniformiser such that $\tau(\pi)=-\pi$ and let $s\in GL_3(\breve K')$ denote the element $\mathrm{diag}(\pi,1,1)$. Then the subgroup of $G_\alpha(\brQ)$ defined by $$G_\alpha(\brQ)\cap s\breve GL_3(\calO_{\breve K'}) s^{-1}$$
	is a very special parahoric subgroup of $G_\alpha(\brQ)$, which we shall call the non-standard parahoric. Up to conjugacy, these are the only very special parahorics of $SU_3$.
	
	\subsubsection{}For $\alpha\in \Sigma$, we define a map $u_\alpha:\bbA^{1,\mathrm{perf}}\rightarrow LU_{\widetilde{\alpha}}\subset LG$ as follows. Let $R$ be a perfect ring of characteristic $p$ and $a\in R$ will denote an arbitrary element. We consider the following three separate cases.

	\begin{enumerate}
		\item $G_\alpha\cong \mathrm{Res}_{\breve K/\brQ}SL_2$ and $\calG_\alpha(\brZ)=SL_2(\calO_{\breve K})$.
		
		We define $u_\alpha$ to be the map induced by $$a\mapsto i_\alpha(f_{\tilde\alpha}(\pi^{-1} \cdot [a])).$$

		\item $G_\alpha\cong \mathrm{Res}_{\breve K/\brQ}SU_3$ and $\calG_\alpha$ is the standard parahoric subgroup.
		
		We define $u_\alpha$ to be the map $$a\mapsto i_\alpha(u_{1}(0,\pi^{-1} \cdot [a])).$$
		
		\item $G_\alpha\cong \mathrm{Res}_{\breve K/\brQ}SU_3$ and $\calG_\alpha$ is the non-standard parahoric subgroup.
		
		We define $u_\alpha$ to be the map $$a\mapsto i_\alpha(u_1([a],\frac{[a]^2}{2})).$$
	\end{enumerate}
	
	\subsubsection{}A calculation using the presentations of $SL_2$ or $SU_3$ above gives the following lemma (cf. \cite[Lemma 7.14]{Nie1}).
	
	\begin{Lem}\label{lem: Nie 7.14}
		\begin{enumerate}
			\item Let $\lambda,\delta\in X_*(T)_I$ and $\alpha,\beta\in\Sigma$ such that $Q(\Sigma)\cap(\bbZ\alpha+\bbZ\beta)$ is of type $A_2$, $A_1\times A_1$ or $A_1$ and such that $$\delta,\delta+\alpha^\vee,\delta-\beta^\vee,\delta+\alpha^\vee-\beta^\vee\preccurlyeq \lambda.$$
			Then for all $y,z\in k$, we have $$u_\alpha(z) \dot{t}^\delta u_\beta(y)\in \bigcup_{\lambda'\preccurlyeq\lambda}L^+\calG\dot t^{\lambda'} L^+\calG.$$
			\item Let $\alpha,\beta\in \Sigma$ and $\lambda\in X_*(T)_I$ such that $\langle\alpha,\beta^\vee\rangle=\langle\beta,\alpha^\vee\rangle=-1$ and $\langle\alpha,\lambda\rangle\geq 2$. Then for any $y,z\in k$ we have $$u_\beta(z)(\dot t^\lambda u_\alpha(y)\dot t^{-\lambda})u_{\beta}(-z)\in L^+\calG$$
		\end{enumerate}
	\end{Lem}

	\subsubsection{} The following lemma is the analogue of \cite[Lemma 7.8]{Nie1}.
	
	\begin{Lem}\label{lem: Nie 7.8}Let $x,x'\in \pi_1(M)_I$, $\alpha\in\Sigma-\Sigma_M$ and $r\in \bbN$ such that 
		
		\begin{enumerate}
			\item $\alpha^\vee$ is $M$-dominant and $M$-minuscule.
			
			\item $r\in [1,h]$ if $\#\calO_\alpha\in\{h,2h\}$ and $r\in [1,2h-1]$ if $\#\calO_\alpha=3h$.
			
			\item $x \overset{(\alpha,r)}\rightarrow x'$.
			
		\end{enumerate}
		
		Then for any $P\in X^M(\mu_x,b)_{K_M}$, there exists $P'\in X^M(\mu_{x'},b)_{K_M}$ such that $P$ and $P'$ lie in the same connected component of $X(\mu,b)_K$ and we have $$\kappa_M(P)-\kappa_M(P')=\sum_{i=0}^{r-1}\sigma^i(\alpha^\vee)\in\pi_1(M)_I.$$
	\end{Lem}
	\begin{proof}As in \cite[Lemma 7.5]{Nie1}, we may assume that $x\overset{(\alpha,r)}\rightarrowtail x'$. Moreover, arguing as in \cite[Lemma 7.15]{Nie1}, it suffices to show that there exists $P\in X^M(\mu_x,b)_{K_M}$ and $P'\in X^M(\mu_{x'},b)_{K_M}$ such that $P$ and $P'$ lie in the same connected component of $X(\mu,b)_K$ and we have $\kappa_M(P)-\kappa_M(P')=\sum_{i=0}^{r-1}\sigma^i(\alpha^\vee)\in\pi_1(M)_I.$

		Let $b_x=\dot t^{\mu_x}\dot w_x$; then $b_x$ is basic in $M$ and since $\kappa_M(b)=\kappa_M(b_x)\in \pi_1(M)_\Gamma$, there exists $g_x\in M(\breve F)$ such that $g_x^{-1}b\sigma(g_x)=b_x$. We define $P:=g_xL^+\calM$ so that $P\in X^M(\mu_x,b)_{K_M}$.
		
		We first consider the case $r\in [1,h]$. For an element $g\in LG$, we write ${^{b_x\sigma}g}$ for the element $b_x\sigma(g)b_x^{-1}$. We define a map $u:\bbA^{1,\mathrm{perf}}\rightarrow \mathrm{Gr}_{\calG}$ given by $$u(z)= g_xu_{\alpha}(z){^{b_x\sigma}u_{\alpha}(z)}\dotsc {^{(b_x\sigma)^{r-1}}u_{\alpha}(z)}L^+\calG$$
		Then by ind-projectivity of $\mathrm{Gr}_{\calG}$, $u$ extends to a map $g:\bbP^{1,\mathrm{perf}}\rightarrow \mathrm{Gr}_{\calG}$. As in \cite[Lemma 7.8]{Nie1}, for any $z\in k$ we have 
		\begin{align*}
		g(z)^{-1}b\sigma(g(z))&\in L^+\calG u_\alpha(-z)b_x\sigma\ {^{(b_x\sigma)^{r-1}}u_{\alpha}(z)}L^+\calG\\
		&=L^+\calG u_\alpha(-z)\dot{t}^{\mu_x}u_{w_x\sigma^r(\alpha)}(c\sigma^r(z)) L^+\calG
		\end{align*}
	for some $c\in k^\times$. Here we use \cite[Corollary 7.12]{Nie1}, which shows that $w_x\sigma^{i}(\alpha)=\sigma^{i}(\alpha)$ and $\langle\sigma^i(\alpha),\mu_x\rangle=0$ for $i\in [1,r-1]$. By \cite[Lemma 4.4.5]{CKV}, we have $$\mu_x+\alpha^\vee,\mu_x-w_x(\sigma^r(\alpha^\vee)),\mu_x+\alpha^\vee-w_x(\sigma^r(\alpha^\vee))\preccurlyeq\mu.$$Thus by Lemma \ref{lem: Nie 7.14} (1), we have $$g(z)^{-1}b\sigma(g(z))\in \bigcup_{\mu'\preccurlyeq \mu}L^+\calG\dot t^{\mu'} L^+\calG$$ and hence $g$ factors through $X(\mu,b)_K$. Moreover one computes that $$P':=g(\infty)=g_x\dot t^{-\sum_{i=0}^{r-1}\sigma^i(\alpha^\vee)}L^+\calG,$$ which lies in the image of $X^M(\mu_{x'},b)_{K_M}$.
		
		We now consider the case $r\in [h+1,2h-1]$. In this case, $\#\calO_\alpha=3h$ and each connected component of the Dynkin diagram of $G$ over $\brQ$ is of type $D_4$. Then either $J=\emptyset$ or $J=\calO_\beta$ where $\beta$ is the unique root in $\Sigma$ with $\sigma^h(\beta)$ and such that $\beta,\alpha$ lie in the same connected component of $\Sigma$. We consider the following two cases.
		
		Case (i): Either $\langle\beta,\mu_x\rangle=0$ or $\langle\beta,\alpha^\vee\rangle=0$. Then, as in \cite[Lemma 7.15, Case 2.2]{Nie1}, we have $\langle\sigma^j(\alpha),\mu_x\rangle=0$ and $w_x(\sigma^j(\alpha))=\sigma^j(\alpha)$ for $j\in[1,r-1]$. Then we may define $u:\bbA^{1,\mathrm{perf}}\rightarrow \mathrm{Gr}_{\calG}$ by $$u(z)= g_xu_{\alpha}(z){^{b_x\sigma}u_{\alpha}(z)}\dotsc {^{(b_x\sigma)^{r-1}}u_{\alpha}(z)}L^+\calG$$ as above. Then $u$ extends to $g:\bbP^{1,\mathrm{perf}}\rightarrow \mathrm{Gr}_{\calG}$ and the same computation shows that $g$ is a curve connecting $P=g(0)\in X^M(\mu_x,b)$ and $P':=g(\infty)=g_x\dot t^{-\sum_{i=0}^{r-1}\sigma^i(\alpha^\vee)}L^+\calG\in X^M(\mu_{x'},b)$. 
		
		Case (ii): $\langle\beta,\alpha^\vee\rangle=-1$ and $\langle\beta,\mu_x\rangle=1$. Then by \cite[Lemma 7.15, Case 2]{Nie1}, upon switching the roles of $x$ and $x'$ if necessary, we may assume that $$\langle \sigma^r(\beta),\mu_x\rangle=\langle\sigma^{r-h}(\alpha),\mu_x\rangle=\langle\sigma^h(\alpha),\mu_x\rangle=0.$$
		
		We define $u:\bbA^{1,\mathrm{perf}}\rightarrow \mathrm{Gr}_{\calG}$ by $$u(z)= g_x{^{(b_x\sigma)^{r-1}}u_{\alpha}(z)}{^{(b_x\sigma)^{r-2}}u_{\alpha}(z)}\dotsc u_{\alpha}(z)L^+\calG$$
		Then $u$ extends to $g:\bbP^{1,\mathrm{perf}}\rightarrow \mathrm{Gr}_{\calG}$ and we have 
		\begin{align*}
		g(z)^{-1}b\sigma(g(z))&\in L^+\calG u_{\sigma^{r-h}(\alpha)}(-c_2z)(\dot t^{\lambda}u_{\sigma^r(\alpha)+\sigma^r(\beta)}(c_1z)\dot t^{-\lambda} )u_{\sigma^{r-h}(\alpha)}(c_2z) u_\alpha(-z)b_xL^+\calG
		\end{align*}where $\lambda\in X_*(T)_I$ satisfies $\langle \sigma^r(\alpha)+\sigma^r(\beta),\lambda\rangle\geq2$.
		By Lemma \ref{lem: Nie 7.14} (1) we have $$u_\alpha(-z)b_x\in \bigcup_{\mu'\preccurlyeq\mu} L^+\calG \dot{t}^{\mu'} L^+\calG$$ and by Lemma \ref{lem: Nie 7.14} (2) we have $$u_{\sigma^{r-h}(\alpha)}(-c_2z)(\dot t^{\lambda}u_{\sigma^r(\alpha)+\sigma^r(\beta)}(c_1z)\dot t^{-\lambda} )u_{\sigma^{r-h}(\alpha)}(c_2z).$$
		It follows that 
		$$g(z)^{-1}b\sigma(g(z))\in \bigcup_{\mu'\preccurlyeq\mu} L^+\calG \dot{t}^{\mu'} L^+\calG$$ and hence $g$ factors through $X(\mu,b)_K$. A similar calculation to the above shows that $g$ is a curve connecting $P=g(0)\in X^M(\mu_x,b)$ and $P':=g(\infty)\in X^M(\mu_{x'},b)$ with $\kappa_M(P)-\kappa_M(P')=\sum_{i=0}^{r-1}\sigma^i(\alpha^\vee)$.
	 \end{proof}
	
	\begin{proof}[Proof of Proposition \ref{prop: conn ADLV Levi 2}]
		This follows by combining Lemma \ref{lem: Nie 7.7} and Lemma \ref{lem: Nie 7.8}.
	 \end{proof}

\subsection{Uniformisation in the case of very special level structure} \label{Sec:UniformisationVerySpecial}
	\subsubsection{} We will use Theorem \ref{thm: conn ADLV} to give a description of the isogeny classes in $\scrS_{\rmK}(G,X)$. We assume that $p>2$, $p\nmid |\pi_1(G_{\mathrm{der}})|$ and that $G_{\bbQ_p}$ is quasi-split and splits over a tamely ramified extension of $\bbQ_p$. We now follow the notation in Section \ref{Sec:IntegralModels}, so $(G,X)$ is a Shimura datum of Hodge type. We let $U=U^pU_p$ where $U_p \subset G(\afp)$ is a compact open subgroup and $U_p$ is a very special connected parahoric subgroup of $G(\qp)$; we write $\calG$ for the corresponding parahoric group scheme. 
	
	\subsubsection{} Recall that for $x \in \shg(\ovfp)$ there is an attached abelian variety $\calA_x$ with contravariant Dieudonn\'e module $\bbD_x$ equipped with tensors $s_{\alpha,0,x}$. Moreover for all $\ell \neq p$ the $\ell$-adic Tate module $T_\ell\calA_x$ is equipped with tensors $s_{\alpha,\ell,x}\in T_{\ell}\calA_x^\otimes$. By \cite[Section 5.6]{Z}, there is an isomorphism
	$$V_{\bbZ{(p)}}\otimes_{\bbZ_{(p)}}\brZ\cong \bbD_x,$$ taking $s_\alpha$ to $s_{\alpha,0,x}$. Under this identification, the Frobenius on $\bbD_x$ is of the form $\varphi=b\sigma$ for some $b \in G(\brQ)$; then $b$ is well-defined up to $\sigma$-conjugation by $\G_K(\brZ)$.
	
	We let $\mu'\in X_*(T)^+_I$ denote the image of a dominant representative of the conjugacy class $\{\mu_h^{-1}\}$, and we define $\mu=\sigma(\mu')$ as in Section \ref{Sec:HamacherKimShtukas}. Then by the argument in \cite[Section 5.6]{Z}, we have $$b\in \calG_{K}(\brZ)\dot{w}\calG_{K}(\brZ)$$ for some $w\in\Adm(\mu)_K$; it follows that $1\in \xmub(\ovfp)$. As in \cite[Section 6.7]{Z}, there is a natural map $$i_x':X(\mu,b)_K(\ovfp)\rightarrow \mathscr{S}_{U_V}(G_V, \mathcal{H}_V)(\ovfp)$$ defined using Dieudonn\'e theory, which sends $1$ to the image of $x$ under $\shg(\ovfp)=\mathscr{S}_{U}(G,X)(\ovfp) \to \mathscr{S}_{U_V}(G_V, \mathcal{H}_V)(\ovfp)$.
	
	Let $r$ be the residue degree of the extension $E_v/\bbQ_p$. Then $X(\mu,b)_K$ is equipped with an action $\Phi$ given by $\Phi(g)=(b\sigma)^r(g)$.
	
	\begin{prop}[{cf. \cite[Proposition 6.5]{Z}}] \label{Prop:UniformisationVerySpecial} Suppose $U_p$ is a very special connected\footnote{See Section \ref{Sec:Parahorics} for the definition of connected parahoric.} parahoric subgroup of $G(\bbQ_p)$. Then there exists a unique map $$i_x: X(\mu,b)_K(\ovfp)\rightarrow \scrS_{U}(G,X)(\ovfp)$$ lifting $i'_x$ such that $s_{\alpha,0,i_x(g)}=s_{\alpha,0,x}$ and $\Phi\circ i_x=i_x\circ \Phi$, where $\Phi$ acts on $\scrS_{U}(G,X)(\ovfp)$ via the geometric $r$-Frobenius.

	\end{prop}
	\begin{proof}For notational simplicity, we write $X(\mu,b)_K$ for $X(\mu,b)_K(\ovfp)$. The uniqueness and compatibility with $\Phi$ is proved in the same way as \cite[Proposition 6.5]{Z}. We may thus define $X(\mu,b)_K^\circ\subset X(\mu,b)_K$ as the maximal subset which admits such a lifting. We therefore need to show that $X(\mu,b)_K^\circ=X(\mu,b)_K$. To do this, we follow the strategy of \cite[Proposition 6.5]{Z}. 
		
		Arguing as in \cite[Lemma 6.10]{Z}, we have that $X(\mu,b)^\circ_K$ is (the set of $\ovfp$-points of) a union of connected components. Note that the key input \cite[Proposition 6.9]{Z} needed for this can be proved verbatim in our setting. 
		
		It therefore suffices to show that the map $$X(\mu,b)_K^\circ\rightarrow \pi_0(X(\mu,b)_K)$$ is a surjection. Let $M\subset G_{\bbQ_p}$ be the standard Levi subgroup given by the centraliser of the Newton cocharacter $\overline\nu_b$. By Theorem \ref{thm: conn ADLV}, there exists $\lambda\in I_{\mu,b,M}$ and an element $$g\in X(\mu,b)^\circ_K\cap X^M(\lambda,b)_M.$$
		Upon replacing $x$ by $i_x(g)$ and using the diagram \cite[Equation (6.7)]{Z}, we may assume $b\in M(\brQ)$. Since $b$ is basic in $M$ and using \cite[Theorem 4.1]{HZ}, we may further assume that $b=\dot\tau_\lambda$ where $\tau_\lambda\in\Omega_M$ corresponds to $\kappa_M(b)\in \pi_1(M)_I$.
		
		Arguing as in \cite[Lemma 6.11]{Z}, we find that \cite[Assumption 5.12]{Z} is satisfied, in other words, the Hodge filtration on $\bbD_x\otimes \ovfp$ lifts to a filtration on $\bbD_x\otimes_{} \mathcal{O}_K$ for some $K/\brQ$ finite which is induced by an $M$-valued cocharacter $\mu_y$. We may therefore let $\tilde{\scrG}/\calO_{\brK'}$ be an $(M,\mu_y)$-adapted lifting of $\scrG$ (cf. \cite[Definition 4.6]{Z}) which corresponds to a point $\tilde{x}\in\scrS_{U}(G,X)(\calO_{\brK'})$. The construction in \cite[Proposition 6.5]{Z} gives us a map $$\iota:M(\bbQ_p)/\calM(\bbZ_p)\rightarrow X^M(\lambda,b)_{K_M},\ \ g\mapsto g_0$$
		which induces a surjection $$M(\bbQ_p)/\calM(\bbZ_p)\rightarrow \pi_0(X^M(\lambda,b)_{K_M})$$ by \cite[Proposition 5.19]{Z}.
		Moreover, the image of $\iota$ lands in $\xmub^\circ$. Therefore by Theorem \ref{thm: conn ADLV}, $\xmub^\circ$ intersects every connected component of $X^M(\lambda,b)_{K_M}$, and hence $\xmub^\circ\rightarrow \pi_0(\xmub)$ is a surjection as desired.

		
		

	 \end{proof}
	
	\subsubsection{} Proposition \ref{Prop:UniformisationVerySpecial} implies that \cite[Assumption 6.17]{Z} is satisfied, hence we obtain Theorem \ref{Thm:VerySpecialUniformisationAndLifting} below. 
\begin{Thm} \label{Thm:VerySpecialUniformisationAndLifting}
		Let $p>2$ and $(G,X)$ a Shimura datum of Hodge type with $G_{\bbQ_p}$ tamely ramified and quasi split. We assume that $p\nmid|\pi_1(G_{\mathrm{der}})|$ and that $U_p$ is a very special connected parahoric subgroup of $G(\bbQ_p)$. 
		\begin{enumerate}
			\item Let $x\in \scrS_{U_p}(G,X)(\ovfp)$ and $b\in G(\brQ)$ the associated element. Then there is a $G(\afp)$-equivariant bijection (where $\mathscr{I}_x \subset \scrS_{U_p}(G,X)(\ovfp)$ is the isogeny class of $x$)
			\begin{align}
			    I_x(\bbQ) \backslash \xmub(\ovfp) \times G(\bbA_f^p) \to \mathscr{I}_x.
			\end{align}
			\item Each isogeny class of $\scrS_{U_p}(G,X)(\ovfp)$ contains a point $x$ which is the reduction of a special point on $\mathbf{Sh}_{U_p}(G,X)$. This confirms \cite[Conjecture 1]{KMPS}.
		\end{enumerate}
\end{Thm}

\DeclareRobustCommand{\VAN}[3]{#3}
\bibliographystyle{amsalpha}
\bibliography{references}

\end{document}